\pdfoutput=1
\documentclass[11pt]{article}

\usepackage{amsthm,amsmath,amsfonts,amssymb}
\usepackage[authoryear,round]{natbib}
\usepackage[colorlinks,citecolor=blue,urlcolor=blue]{hyperref}
\usepackage{graphicx}
\usepackage[margin=1in]{geometry}
\usepackage[T1]{fontenc}
\usepackage{bm}
\usepackage{xcolor}

\theoremstyle{plain}

\newtheorem{theorem}{Theorem}[section]
\newtheorem{lemma}[theorem]{Lemma}
\newtheorem{proposition}[theorem]{Proposition}
\newtheorem{corollary}[theorem]{Corollary}
\theoremstyle{definition}

\newtheorem*{example}{Example}
\newtheorem*{remark}{Remark}

\usepackage{array}
\usepackage{mathtools}
\mathtoolsset{centercolon}  

\providecommand{\mathbbm}{\mathbb} 

\newcommand{\R}{\mathbbm{R}}

\newcommand{\sfr}{\mathsf{r}}

\usepackage[font = small, margin=30pt]{caption}

\usepackage{pgfplots}
\pgfplotsset{compat=1.18}
\usepgfplotslibrary{groupplots}
\usepackage[]{algorithm}
\usepackage{algpseudocode}
\usepackage{enumerate}

\usepackage{enumitem}

\usepackage{chngcntr}
\usepackage{mathrsfs}
\counterwithin{table}{section}
\counterwithin{algorithm}{section}

\usepackage{subcaption}

\newcommand{\indicator}{{\bf{1}}}

\newcommand{\equald}{\stackrel{d}{=}}

\newcommand{\E}{\mathbb{E}}

\renewcommand{\P}{\mathbb{P}}

\newcommand{\tildet}{\widetilde{t}}

\newcommand{\hatt}{\hat{t}}

\newcommand{\hattheta}{\hat{\theta}}

\newcommand{\overbar}[1]{\mkern 1.5mu\overline{\mkern-1.5mu#1\mkern-1.5mu}\mkern 1.5mu}

\newcommand{\barY}{\overbar{Y}}

\newcommand{\mcE}{\mathcal{E}}

\newcommand{\mcG}{\mathcal{G}}

\newcommand{\mcN}{\mathcal{N}}

\newcommand{\mcS}{\mathcal{S}}

\usepackage{multirow}

\newcommand{\inparen}[1]{\left(#1\right)}             
\newcommand{\inbraces}[1]{\left\{#1\right\}}           
\newcommand{\insquare}[1]{\left[#1\right]}             

\newcommand{\inp}[1]{\ensuremath{\left\langle #1 \right\rangle}}

\DeclareMathOperator*{\argmax}{arg\,max}
\DeclareMathOperator*{\argmin}{arg\,min}

\renewcommand{\epsilon}{\varepsilon}

\title{Risk reversal for least squares estimators\\ under nested convex constraints}
\author{Omar Al-Ghattas\\
\small Broad Institute of MIT and Harvard\\
\small \texttt{oag@mit.edu}}
\date{}

\begin{document}
\maketitle

\begin{abstract}
In constrained stochastic optimization, one naturally expects that restricting the feasible set, provided it still contains the true parameter, should not increase the statistical risk of the corresponding projection estimator. We show that this intuition can fail, even in basic settings.

We investigate this phenomenon in the Gaussian sequence model. Given a compact, convex set $\Theta \subseteq \R^d$, one observes
\[
Y = \theta^\star + \sigma Z, \qquad Z \sim N(0, I_d),
\]
and seeks to estimate an unknown parameter $\theta^\star \in \Theta$. 

In this setting, the maximum likelihood estimator over $\Theta$ coincides with the least squares estimator (LSE) and is given by the Euclidean projection of $Y$ onto $\Theta$. We construct an explicit example exhibiting \emph{risk reversal}: for sufficiently large noise, there exist nested compact convex sets $\Theta_S \subsetneq \Theta_L$ and a parameter $\theta^\star \in \Theta_S$ such that the LSE constrained to $\Theta_S$ has strictly larger squared-error risk than the LSE constrained to $\Theta_L$. Moreover, we demonstrate that risk reversal can persist at the level of worst-case risk, in the sense that the supremum of the risk over the smaller constraint set exceeds that over the larger one. Finally, we show that the phenomenon is not specific to Gaussian noise or squared-error risk, extending our results beyond both settings.

We clarify this phenomenon by contrasting the behavior of the LSE across noise regimes. In the vanishing-noise limit, the risk admits a first-order expansion governed by the statistical dimension of the tangent cone at $\theta^\star$, and risk reversal cannot occur at the leading $\sigma^2$ scale. In the diverging-noise regime, the risk is instead determined by global geometric interactions between the constraint set and random noise directions. Here, the embedding of $\Theta_S$ within $\Theta_L$, rather than the local geometry at $\theta^\star$, plays a decisive role in the behavior of the LSE and can reverse the risk ordering.

These results reveal a previously unrecognized failure mode of the LSE, and of projection-based estimators more broadly. They demonstrate that in sufficiently noisy settings, tightening a constraint can paradoxically degrade statistical performance.
\end{abstract}

\section{Introduction}

In statistical estimation, it is standard to analyze estimators under the assumption that the unknown parameter belongs to a prescribed feasible set. Such constraints encode prior structural information, reduce the effective complexity of the parameter space, and are therefore widely believed to improve estimation accuracy. This suggests that, when a constraint is correctly specified, further restricting the feasible set in a way that preserves the truth should not worsen performance.

In applied work, constraints rarely appear fully formed. More often, a statistician begins with a simple modeling assumption, such as boundedness or smoothness, and then gradually refines the model as additional insight becomes available. New constraints may be introduced to reflect domain knowledge, improve interpretability, or enforce desired qualitative behavior. Each refinement is typically viewed as a step toward a more faithful representation of the underlying problem, and thus toward better statistical performance. From this perspective, it is far from obvious that adding valid structure could ever worsen statistical performance. 

We study this question in the Gaussian sequence model under compact convex constraints. Here, the natural estimator is the least squares estimator (LSE), which coincides with the maximum likelihood estimator (MLE) and admits a geometric characterization as the Euclidean projection onto the constraint set. Although this formulation is simple, existing analyses of the LSE typically rely on local approximations near the true parameter and therefore do not capture how the estimator's risk depends on the global geometry of the constraint set. When two compact, convex constraint sets are nested and the smaller set contains the true parameter, one might expect the corresponding LSE to have lower risk. We show that this intuition can fail: there exist settings in which the LSE associated with the smaller constraint has strictly larger risk than that associated with the larger constraint, both pointwise at a fixed parameter and at the level of worst-case risk.

\subsection{Problem Setup}
We now introduce the mathematical framework that will serve as the basic setting for our analysis throughout the paper. Let $\Theta_S$ and $\Theta_L$ be compact, convex subsets of $\R^d$ satisfying
\[
\Theta_S \subsetneq \Theta_L,
\]
and suppose the unknown parameter $\theta^\star$ lies in $\Theta_S$. We observe a single realization 
from the Gaussian sequence model
\begin{align}\label{eq:GSM}
Y = \theta^\star + \sigma Z, \qquad Z \sim N(0,I_d),
\end{align}
where $\sigma>0$ denotes the noise level. Our goal is to estimate $\theta^\star$ from the noisy 
observation $Y$. For a nonempty, compact and convex set $\Theta \subseteq \R^d$, define the Euclidean projection of a point
$y \in \R^d$ onto $\Theta$ by
\[
\Pi_\Theta(y)
:= \argmin_{\theta \in \Theta} \|y - \theta\|^2.
\]
Since $\Theta$ is compact and convex, the minimizer exists and is unique. The estimator $\Pi_\Theta(Y)$ is the LSE, which coincides with the MLE under the constraint $\theta^\star \in \Theta$. 

We consider the two LSEs
\[
\hattheta_S := \Pi_{\Theta_S}(Y), \qquad
\hattheta_L := \Pi_{\Theta_L}(Y),
\]
and evaluate their performance using the squared-error risk
\[
R_\sigma(\theta^\star; \Theta)
:= \E_{\theta^\star}\|\Pi_{\Theta}(\theta^\star + \sigma Z) - \theta^\star\|^2.
\] 
The inclusion $\Theta_S \subsetneq \Theta_L$ suggests a natural comparison. The estimator $\hattheta_S$ 
enforces a stronger constraint and thus incorporates more prior information than $\hattheta_L$, 
while remaining correctly specified whenever $\theta^\star \in \Theta_S$. This leads to the central 
question of this work: can strengthening the constraint ever increase the estimation risk? In 
particular, do there exist settings in which
\[
R_\sigma(\theta^\star; \Theta_S) > R_\sigma(\theta^\star; \Theta_L)?
\]
We refer to this phenomenon as \emph{risk reversal}. Beyond this pointwise comparison, we also ask whether an analogous reversal can occur at the level of worst-case risk over the constraint set.

\begin{remark}[Multiple observations]
For notational simplicity, we state results for a single observation. This entails no loss of generality, since, if instead one observes $n\ge1$ independent samples
$
Y_1,\dots,Y_n \sim N(\theta^\star,\sigma^2 I_d)
$,
the LSE is simply the orthogonal projection of the sample mean $\barY$ onto the constraint set. Since $\barY$ is Gaussian with mean $\theta^\star$ and covariance matrix $\frac{\sigma^2}{n} I_d$, the $n$-sample problem is equivalent to the single-sample model with effective noise level $\frac{\sigma}{\sqrt n}$. Consequently, all results of this paper extend immediately to the $n$-sample setting after replacing $\sigma$ by the effective noise level.
\end{remark}

\begin{figure}
    \centering
    \resizebox{\linewidth}{!}{%
        \begin{tabular}{c@{\hspace{0.5cm}}c} 

            
            \begin{tikzpicture}[scale=2.5]
                \def\c{0.6}
                \pgfmathsetmacro{\invc}{1/\c}
                \pgfmathsetmacro{\alphac}{1 + (\invc)^2} 
                
                \def\bound{2.5}       
                \def\rightlimit{2.8}  
                
                \tikzset{
                    every node/.style={font=\small}, 
                    regionlabel/.style={fill=white, fill opacity=0.8, text opacity=1, align=center, rounded corners=2pt, inner sep=3pt}
                }

                \clip (-\bound+1, -\bound+1) rectangle (\rightlimit, \bound);

                \fill[black!20] 
                    (-3, -3) -- (3, -3) -- (3, {-3*\invc}) -- (-3, {-(-3)*\invc}) -- cycle;

                \fill[blue!20] 
                    (-3, {3*\invc}) -- 
                    (4, {-4*\invc}) -- 
                    (4, {-4*\invc + \alphac}) -- 
                    (-3, {3*\invc + \alphac}) -- cycle;

                \fill[violet!40] 
                    (-3, {3*\invc + \alphac}) -- 
                    (4, {-4*\invc + \alphac}) -- 
                    (4, 4) -- (-3, 4) -- cycle;

                \draw[->, thick, gray] (-\bound+1.2, 0) -- (\rightlimit-0.2, 0) node[right, black] {$y_1$};
                \draw[->, thick, gray] (0, -\bound+1.2) -- (0, \bound-0.2) node[above, black] {$y_2$};

                \draw[thick, black] (0,0) -- (\invc, 1);
                
                \node[regionlabel, fill opacity=0.0] at (-0.5, -0.5) {$A_1$\\$\hattheta_S=(0,0)$};
                \node[regionlabel, fill opacity=0.0] at (0.5, 1) {$A_{12}$\\$\hattheta_S=(\frac{y_1+cy_2}{1+c^2}, \frac{c(y_1+cy_2)}{1+c^2} )$};
                \node[regionlabel, fill opacity=0.0] at (\invc + 0.4, 1.75) {$A_2$\\$\hattheta_S=(\frac1c,1)$};

                \fill (0,0) circle (0.5pt);
                \node[anchor=north east, inner sep=3pt] at (0,0) {$v_1=(0,0)$};
                
                \fill (\invc, 1) circle (0.5pt);
                \node[anchor=south west, inner sep=3pt, yshift=2pt] at (\invc,1) {$v_2=(\frac1c,1)$};

                \node[anchor=north east, inner sep=3pt] at (1.1,0.5) {$\Theta_S$};

                \node[text=blue!80!black] at (0,0) {$\star$};
                \node[text=blue!80!black, anchor=south east] at (0,0) {$\theta^\star$};

            \end{tikzpicture}
            & 

            \begin{tikzpicture}[scale=2.5]
                \def\c{0.6}
                \pgfmathsetmacro{\invc}{1/\c}
                \def\bound{2.5}       
                \def\rightlimit{2.8}  
                
                \def\yRowTop{1.8}    
                \def\yRowMid{0.75}   
                \def\yRowBot{-0.5}   
                
                \tikzset{
                    every node/.style={font=\small}, 
                    regionlabel/.style={fill=white, fill opacity=0.8, text opacity=1, align=center, rounded corners=2pt, inner sep=3pt}
                }

                \clip (-\bound+1, -\bound+1) rectangle (\rightlimit, \bound);

                \fill[teal!40] (0,0) -- (0,1) -- (\invc, 1) -- cycle;
                \fill[red!40] (0,1) -- (-\bound, 1) -- (-\bound, \bound) -- (0, \bound) -- cycle;
                \fill[orange!40] (0,0) -- (-\bound, 0) -- (-\bound, 1) -- (0,1) -- cycle;
                \fill[black!20] (0,0) -- (-\bound, 0) -- (-\bound, -\bound) -- (\c*\bound, -\bound) -- cycle; 
                \fill[blue!40] (0,1) -- (0, \bound) -- (\invc, \bound) -- (\invc, 1) -- cycle;
                \fill[violet!40] (\invc, 1) -- (\invc, \bound) -- (\bound+2, \bound) -- (\bound+2, {1 - \invc*(\bound+2 - \invc)}) -- cycle;
                \fill[blue!10] (0,0) -- (\invc, 1) -- (\bound+2, {1 - \invc*(\bound+2 - \invc)}) -- (\bound+2, -\bound) -- (\c*\bound, -\bound) -- cycle;

                \draw[->, thick, gray] (-\bound+1.2, 0) -- (\rightlimit-0.2, 0) node[right, black] {$y_1$};
                \draw[->, thick, gray] (0, -\bound+1.2) -- (0, \bound-0.2) node[above, black] {$y_2$};
                \draw[thick, black] (0,0) -- (0,1) -- (\invc, 1) -- cycle;
       
                \node[regionlabel, fill opacity=0.0] at (-0.8, \yRowTop) {$A_3$\\$\hattheta_L=(0,1)$};
                \node[regionlabel, fill opacity=0.0] at (0.5*\invc, \yRowTop) {$A_{23}$\\$\hattheta_L=(y_1,1)$};
                \node[regionlabel, fill opacity=0.0] at (\invc + 0.5, \yRowTop) {$A_2$\\$\hattheta_L=(\frac1c,1)$};

                \node[regionlabel, fill opacity=0.0] at (-0.8, \yRowMid) {$A_{13}$\\$\hattheta_L=(0,y_2)$};
                \node[regionlabel, fill opacity=0.0] at (0.3*\invc, \yRowMid) {$\Theta_L$\\$\hattheta_L=(y_1,y_2)$};
                
                \node[regionlabel, fill opacity=0.0] at (-0.5, \yRowBot) {$A_1$\\$\hattheta_L=(0,0)$};
                \node[regionlabel, fill opacity=0.0] at (1.3, \yRowBot) {$A_{12}$\\$\hattheta_L=(\frac{y_1+cy_2}{1+c^2}, \frac{c(y_1+cy_2)}{1+c^2} )$};

                \fill (0,0) circle (0.5pt);
                \node[anchor=north east, inner sep=3pt] at (0,0) {$v_1=(0,0)$};
                
                \fill (0,1) circle (0.5pt);
                \node[anchor=south east, inner sep=3pt, yshift=2pt] at (0,1) {$v_3 =(0,1)$};
                
                \fill (\invc,1) circle (0.5pt);
                \node[anchor=south west, inner sep=3pt, yshift=2pt] at (\invc,1) {$v_2=(\frac1c,1)$};

                \node[text=blue!80!black] at (0,0) {$\star$};
                \node[text=blue!80!black, anchor=south east] at (0,0) {$\theta^\star$};
            \end{tikzpicture}
        \end{tabular}
    }
    \caption{
    Geometry of the LSE in the two-dimensional observation space
    $y=(y_1,y_2)\in\R^2$ over the constraint sets
    $\Theta_S=\operatorname{conv}(v_1,v_2)$ (left) and
    $\Theta_L=\operatorname{conv}(v_1,v_2,v_3)$ (right).
    The shaded regions illustrate the partition of the sample space
    based on the location of the constrained solution.
    Regions $A_j$ denote observations projected onto vertex $v_j$,
    while regions $A_{ij}$ denote observations projected onto the line
    segment connecting $v_i$ and $v_j$.
    }
    \label{fig:side_by_side}

\end{figure}

\subsection{Risk reversal: an explicit example}
\label{sec:GSM-Counterexample}

We begin with a two-dimensional example in which a tighter, correctly specified constraint produces larger risk. The risks under both constraints admit explicit formulas, allowing us to track their dependence on the noise level. The example also illustrates the global geometric mechanism that motivates the general vanishing- and diverging-noise analyses developed below.

\begin{example}[Running example] 
Fix a constant $c>0$ and define the points
\[
v_1 := (0,0), \qquad 
v_2 := (c^{-1},1), \qquad 
v_3 := (0,1).
\]
Let
\[
\Theta_L := \operatorname{conv}(v_1, v_2, v_3)
\]
denote the convex hull of these points, forming a closed triangular region, and let
\[
\Theta_S := \operatorname{conv}(v_1, v_2)
\]
denote the line segment joining $v_1$ and $v_2$. We observe a single realization from the Gaussian sequence model \eqref{eq:GSM} with true parameter $\theta^\star = v_1$. Let $\hattheta_S$ and $\hattheta_L$
denote the corresponding LSEs under the constraints $\Theta_S$ and $\Theta_L$, respectively.
\end{example}

 Although we take $\theta^\star=v_1$, a vertex of the constraint set, this choice is made only to simplify the risk calculations. The phenomenon is not tied to placing the true parameter at a vertex, or even at the boundary, as shown in Corollary~\ref{cor:ellipse}. We depict the constraint sets and the action of the LSEs across different regions of the sample space in Figure~\ref{fig:side_by_side}.

\definecolor{myblue}{RGB}{0, 114, 178}
\definecolor{mycyan}{RGB}{86, 180, 233}
\definecolor{mygreen}{RGB}{0, 158, 115}
\definecolor{myorange}{RGB}{230, 159, 0}
\definecolor{myred}{RGB}{213, 94, 0}
\definecolor{mypurple}{RGB}{204, 121, 167}

\begin{figure}
    \centering
    \begin{tikzpicture}
    \begin{semilogxaxis}[
      width=14cm,
      height=6.5cm,
      xlabel={$\sigma$},
      ylabel={$R_\sigma(\theta^\star; \Theta_S)-R_\sigma(\theta^\star; \Theta_L)$}, 
      xlabel style={font=\normalsize},
      ylabel style={font=\normalsize},
      tick label style={font=\small},
      xmin=1e-2,
      xmax=50,
      grid=both,
      major grid style={gray!30},
      minor grid style={gray!15},
      legend style={
        font=\small,
        draw=none,
        fill=white,
        fill opacity=0.8,
        text opacity=1,
        at={(0.02,0.98)},
        anchor=north west,
        legend cell align={left}
      },
      cycle list={
        {myblue,   very thick, line cap=round},
        {mycyan,   very thick, line cap=round},
        {mygreen,  very thick, line cap=round},
        {myorange, very thick, line cap=round},
        {myred,    very thick, line cap=round},
        {mypurple, very thick, line cap=round}
      }
    ]

    \addplot [gray, dashed, thin, domain=1e-2:50, forget plot] {0};

    \pgfplotstableread[col sep=comma]{risk_difference_single_plot_multi_c.csv}\datatable

    \addplot table[x=sigma,y=diff_c02] {\datatable};
    \addlegendentry{$c=0.2$}

    \addplot table[x=sigma,y=diff_c05] {\datatable};
    \addlegendentry{$c=0.5$}

    \addplot table[x=sigma,y=diff_c09] {\datatable};
    \addlegendentry{$c=0.9$}

    \addplot table[x=sigma,y=diff_c10] {\datatable};
    \addlegendentry{$c=1$}

    \addplot table[x=sigma,y=diff_c20] {\datatable};
    \addlegendentry{$c=2$}

    \addplot table[x=sigma,y=diff_c50] {\datatable};
    \addlegendentry{$c=5$}

    \end{semilogxaxis}
    \end{tikzpicture}
    
    \caption{Risk difference as a function of noise level $\sigma$. The difference $R_\sigma(\theta^\star; \Theta_S)-R_\sigma(\theta^\star; \Theta_L)$ is plotted against $\sigma$ for varying values of the geometric parameter $c$. Note that for $c \in \{0.2,0.5,0.9 \}$, $\hattheta_L$ outperforms $\hattheta_S$ whenever the noise level is sufficiently large.}
    \label{fig:varying-c}
\end{figure}

\begin{theorem}\label{thm:counter-example}
Fix $c > 0$ in the running example. For any $\sigma>0$, the risks $R_\sigma(\theta^\star; \Theta_S)$ and $R_\sigma(\theta^\star; \Theta_L)$ admit exact closed-form expressions. Moreover, the following asymptotic comparisons hold.
\begin{enumerate}
\item[\textup{(i)}] \textit{Vanishing-noise regime.}
As $\sigma\to 0^+$,
\begin{equation}\label{eq:vanishing-noise-diff}
R_\sigma(\theta^\star; \Theta_S) - R_\sigma(\theta^\star; \Theta_L)
=
-\frac{\sigma^2}{\pi}\arctan \frac{1}{c} + o(\sigma^2).
\end{equation}
In particular, there exists $\sigma_0(c) > 0$ such that for all $\sigma\in(0,\sigma_0(c))$,
\[
R_\sigma(\theta^\star; \Theta_S)<R_\sigma(\theta^\star; \Theta_L).
\]

\item[\textup{(ii)}] \textit{Diverging-noise regime.}
As $\sigma\to\infty$,
\begin{equation}\label{eq:diverging-noise-diff}
R_\sigma(\theta^\star; \Theta_S) - R_\sigma(\theta^\star; \Theta_L)
=
\frac{1}{4c^2} - \frac{1+c^2}{2\pi c^2}\arctan \frac{1}{c}
+o(1).
\end{equation}
In particular, for $c\neq 1$, there exists $\sigma_1(c)<\infty$
such that for all $\sigma \in (\sigma_1(c), \infty)$,

\[
\begin{cases}
R_\sigma(\theta^\star; \Theta_S) > R_\sigma(\theta^\star; \Theta_L), & 0 < c < 1,\\[4pt]
R_\sigma(\theta^\star; \Theta_S) < R_\sigma(\theta^\star; \Theta_L), & c > 1.
\end{cases}
\]
When $c=1$, the leading constant in \eqref{eq:diverging-noise-diff} vanishes, so that
\[
R_\sigma(\theta^\star; \Theta_S) - R_\sigma(\theta^\star; \Theta_L)=o(1).
\]
\end{enumerate}
\end{theorem}

The proofs of all results in the paper are deferred to Section~\ref{sec:proofs}.

We emphasize that the key geometric mechanism underlying risk reversal in this example is the misalignment between the constraint sets. Although $\Theta_S \subsetneq \Theta_L$, the orientation of $\Theta_S$ within $\Theta_L$ is such that, for certain realizations of $Y$, projection onto $\Theta_S$ yields an estimate that lies farther from $\theta^\star$ than the projection onto $\Theta_L$. Figure~\ref{fig:goodregions} highlights the subset $\mathcal G \subseteq \R^2$ of points for which this pointwise reversal occurs. This observation suggests that the mechanism behind risk reversal is non-local: it depends on the behavior of the projection map on regions of the sample space away from $\theta^\star$, and hence is not captured by geometric characterizations based only on the local structure of the constraint near $\theta^\star$. For these pointwise discrepancies to translate into a reversal of risks, the noise distribution centered at $\theta^\star$ must place sufficient probability mass on such regions.

Theorem~\ref{thm:counter-example} shows that for the running example, risk reversal cannot occur in the vanishing-noise regime. In contrast, the diverging-noise regime is more subtle. Depending on the parameter $c$, which controls the length of the edge connecting vertices $v_2$ and $v_3$ in $\Theta_L$, the risk ordering may either agree with or reverse the vanishing-noise case. In particular, for $0<c<1$ the tighter constraint $\Theta_S$ has strictly larger risk for sufficiently large $\sigma$, whereas
for $c>1$, the ordering is preserved.

The explicit risk expressions enable us to plot the associated risk curves (Figure~\ref{fig:varying-c}), offering further intuition about the dependence of risk reversal on the noise level. Notably, these curves suggest a single-crossing pattern: once the risk difference becomes positive, it appears to remain positive as $\sigma$ increases. This observation motivates an analysis of the diverging-noise regime, since the ordering of the risks stabilizes and can be characterized asymptotically. Similarly, Figure~\ref{fig:heatmap} visualizes $R_\sigma(\theta^\star; \Theta_S)-R_\sigma(\theta^\star; \Theta_L)$ over a range of noise levels $\sigma$ and geometric parameters $c$, showing that risk reversal can occur for moderate to large values of $\sigma$.

\begin{figure}[t]
\centering
\begin{tikzpicture}
\begin{axis}[
    width=12cm,
    height=7cm,
    width=0.72\textwidth,
    xmin=0.001, xmax=1.5,
    ymin=0.2,   ymax=10,
    xlabel={$c$},
    ylabel={$\sigma$},
    enlargelimits=false,
    axis on top,
    tick pos=left,
    tick align=outside,
    tick label style={font=\small},
    label style={font=\large},
    clip=true,
]

\addplot graphics[
    xmin=0.001, xmax=1.5,
    ymin=0.05,  ymax=10,
] {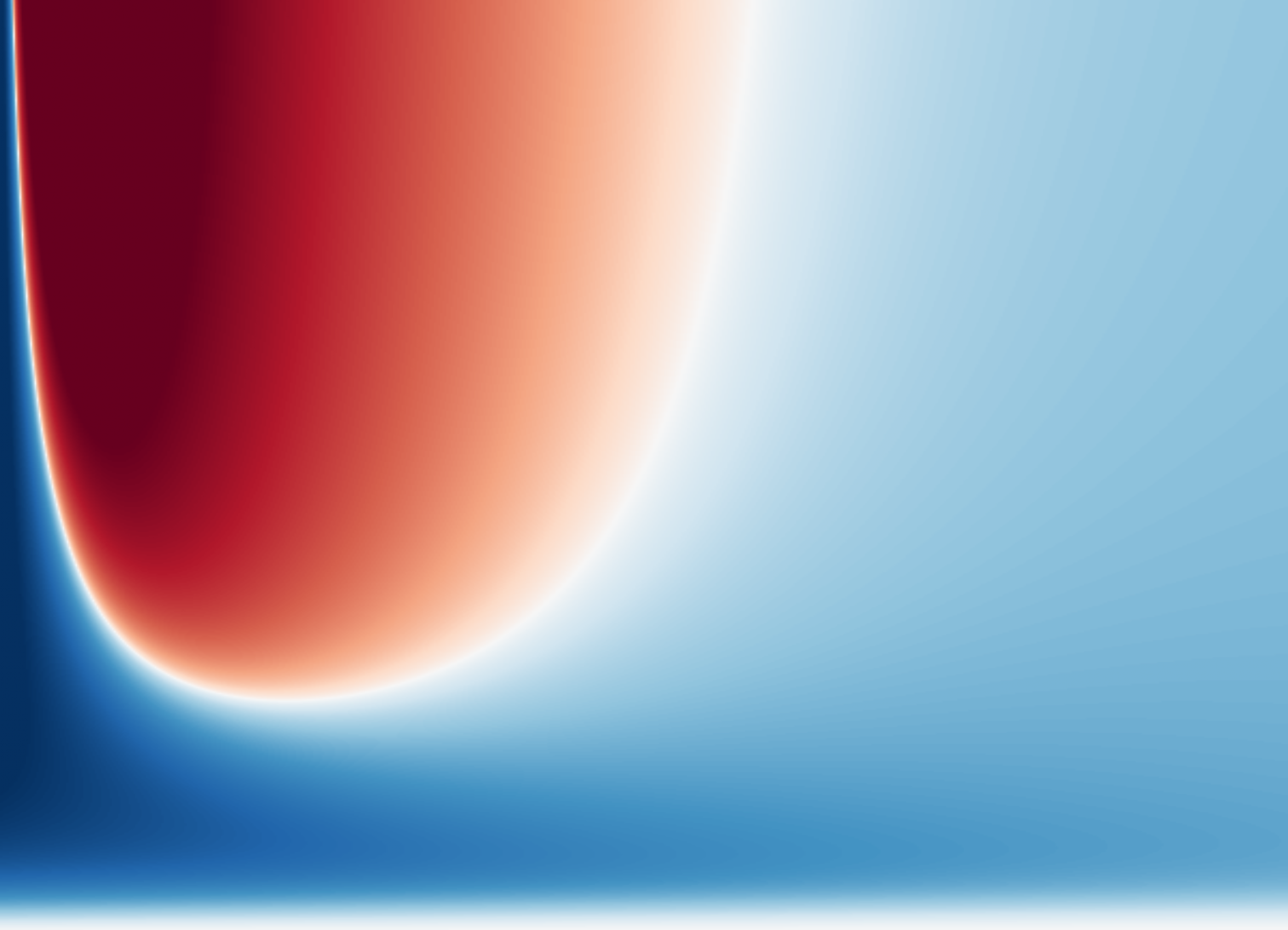};
\end{axis}
\end{tikzpicture}
\caption{Heat map of the risk difference $R_\sigma(\theta^\star; \Theta_S)-R_\sigma(\theta^\star; \Theta_L)$ as a function of the geometric parameter $c$ (horizontal axis) and the noise level $\sigma$ (vertical axis). Red indicates positive risk difference, blue indicates negative risk difference, and white indicates values close to zero. In the running example, risk reversal occurs in a substantial intermediate-to-large noise regime for values of $c\in(0,1)$.}
 \label{fig:heatmap}
\end{figure}

\subsection{Risk reversal: beyond pointwise comparisons}

Up to this point, our analysis has focused on a \emph{pointwise} manifestation of risk reversal. One might reasonably view such pointwise failures as relatively mild. Indeed, statistical theory contains many examples---most famously the Hodges super-efficient estimator \cite[Chapter 8]{van2000asymptotic}---in which an estimator behaves anomalously at isolated points or in small neighborhoods, without representing a genuine global phenomenon. From this perspective, one could ask whether risk reversal is merely a local pathology, in which a tighter constraint degrades performance only at isolated points while still offering uniformly better performance overall, or whether it reflects a more substantial breakdown of the LSE.

A natural way to strengthen the comparison is to consider the worst-case risk over the constraint set. From a decision-theoretic perspective, this corresponds to evaluating the estimator under the least favorable parameter consistent with the assumed constraint. If tightening a correctly specified constraint increases even the worst-case risk, then the phenomenon cannot be dismissed as localized or atypical.

\begin{theorem}\label{thm:worst-case-risk-reversal}
There exist nested compact convex sets $\Theta_S \subsetneq \Theta_L$ and a constant $\sigma_1 < \infty$ such that, for every $\sigma \in (\sigma_1, \infty)$,
\[
\sup_{\theta \in \Theta_S} R_\sigma(\theta; \Theta_S)
>
\sup_{\theta \in \Theta_L} R_\sigma(\theta; \Theta_L).
\]
\end{theorem}

Theorem~\ref{thm:worst-case-risk-reversal} follows immediately from Lemma~\ref{lem:explicit-worst-case-risk}, which in turn relies on a characterization of the LSE in the diverging-noise regime developed in Section~\ref{sec:diverging-noise-asymptotic}. In particular, we construct explicit constraint sets whose limiting risks exhibit worst-case risk reversal as $\sigma \to \infty$. We then show that, in this regime, the risk function converges uniformly over the constraint set. Consequently, worst-case risk reversal holds for all sufficiently large finite noise levels. We provide a detailed description of the result in Section~\ref{sec:sup-risk}.

\subsection{Risk reversal beyond Gaussian noise and squared-error risk}

The preceding results are formulated in the Gaussian sequence model, with risk
evaluated using squared-error risk. One could ask whether risk reversal is
specific to Gaussian noise or to squared-error risk. We show that it is not. In
Section~\ref{ssec:beyondGaussianSqLoss}, we introduce a class of radial
location models whose densities have the form $f_q(w)=C_q^{-1}q(\|w\|)$, where $q$ is a strictly decreasing radial density profile. For this class of models, the constrained MLE is still given by Euclidean projection onto the parameter space. We also
allow the risk to be evaluated using radial loss functions of the form
$\eta^{-1}\|\cdot\|^\eta$. Thus the estimator remains likelihood-based, while
both the noise distribution and the risk are allowed to vary. The following informal theorem summarizes the precise results proved in Theorems~\ref{thm:radial-beta-eta} and~\ref{thm:radial-worst-case-all-eta} in
Section~\ref{ssec:beyondGaussianSqLoss}.

\begingroup
\renewcommand{\thetheorem}{}
\begin{theorem}[Informal]
For a broad class of radial location models and radial loss functions, risk reversal
persists, in both a pointwise and worst-case sense, for all sufficiently large
noise levels.
\end{theorem}
\addtocounter{theorem}{-1}
\endgroup

The key point is that the diverging-noise behavior of the projection depends only on the direction of the noise and on the
geometry of the constraint sets. For radial noise, this direction is uniform on
the sphere, as in the Gaussian case, so the exposed-face selection probabilities
are unchanged. Changing the risk from squared-error risk to a radial loss
changes only the loss assigned to the selected limiting faces, not the
underlying face-selection geometry. As a consequence, the same geometric
mechanism yields risk reversal beyond the Gaussian squared-error setting, with
the limiting inequalities again implying reversal at sufficiently large finite
noise levels.

\begin{remark}[Risk reversal as a genuine statistical phenomenon]
We now give several reasons why the behavior demonstrated above should be
viewed as a genuine statistical phenomenon, rather than as an artifact of the
particular constructions used to exhibit it.

First, the estimator studied here is statistically motivated. In the Gaussian
sequence model~\eqref{eq:GSM}, the constrained LSE coincides with the
constrained MLE over $\Theta$. Thus the estimator is the likelihood-based
estimator under the assumed constraint, rather than an ad hoc projection rule.

Second, the constraints are correctly specified throughout our examples, in the
sense that $\theta^\star\in\Theta_S\subsetneq\Theta_L$. Hence risk reversal
does not arise from model misspecification.

Third, the geometry is deliberately well behaved. Throughout the paper, we work
with compact, convex constraint sets, so the projection estimator is well
defined and its diverging-noise behavior admits a clean geometric description in
terms of exposed faces and supporting hyperplanes; see
Section~\ref{sec:diverging-noise-asymptotic}. Thus the examples show that risk
reversal can occur even under restrictive convex-geometric assumptions.

Moreover, the specific features of the running example are not
intrinsic to risk reversal. In particular, the phenomenon is not tied to the
signal being an extreme point, to nonsmooth constraint sets, or to the smaller
constraint being lower-dimensional than the larger one. Indeed,
Corollary~\ref{cor:ellipse} gives an explicit smooth, full-dimensional example for which $\theta^\star$ lies in the interior of both constraint sets and risk reversal still occurs.

Fourth, the mechanism is not tied to Gaussianity or squared-error risk; as shown in Section~\ref{ssec:beyondGaussianSqLoss}, it persists for a broad class of radial location models and radial risks.

Finally, although part of our analysis is asymptotic in the noise level, the phenomenon itself is finite-sample and finite-dimensional. Risk reversal occurs at finite noise levels, for a single observation, and in fixed dimension. The diverging-noise limit is used only as a tractable lens for isolating the underlying geometric mechanism.
\end{remark}

\begin{figure}
    \centering
\begin{tikzpicture}[line cap=round, line join=round, font=\small, baseline={(0,1.5)}]

    \def\sigmaRad{1.0}

    \fill[orange!40] (-.75, 6) 
        .. controls (-0.5, 4.5) and (2.0, 3) .. (3.5, 3) 
        -- (6, 3) 
        -- (6, 6) 
        -- cycle;

    \coordinate (theta_star) at (0,0);

    \draw[blue!50!gray, thin, dashed, opacity=0.7] (theta_star) circle (\sigmaRad);

    \draw[-, thin, black, shorten <=2pt] (theta_star) -- (300:\sigmaRad) 
    node[midway, above right, inner sep=1pt] {$\sigma$};

    \coordinate (top_left) at (0,3);
    \coordinate (top_right) at (6,3);
    \coordinate (mid_top) at (2.5,3);
    \coordinate (Y) at (2.5,4.5);
    \coordinate (theta_s) at (3.8,1.9);
    
    \draw[blue!60!black, thick] (theta_star) -- (top_left) -- (top_right) -- cycle;

    \draw[gray, thick, dashed] (mid_top) -- (Y);
    \draw[gray, dashed, thick] (Y) -- (theta_s);

    \draw[red, thick, dashed] (theta_star) -- (mid_top) 
        node[midway, above, sloped, black] {$d_L(Y)$};
        
    \draw[red, thick, dashed] (theta_star) -- (theta_s) 
        node[midway, below, sloped, black] {$d_S(Y)$};

    \fill[blue!40!black] (Y) circle (2pt) node[above right, black] {$Y$};
    \fill[blue!40!black] (mid_top) circle (2pt) node[above left, black] {$\hattheta_{L}$};
    \fill[blue!40!black] (theta_star) circle (2pt) node[below left, black] {$\theta^\star$};
    \fill[blue!40!black] (theta_s) circle (2pt) node[below right, black] {$\hattheta_{S}$};
    \node[anchor=north west, inner sep=3pt, orange, scale=2] at (4.5,5) {$\mcG$};
\end{tikzpicture}
\caption{
A schematic illustration of the risk reversal phenomenon.
A realization $Y$ is projected onto the larger constraint set, yielding $\hattheta_{L}$, and onto the smaller set, yielding $\hattheta_{S}$.
In this configuration, the distance from $\theta^\star$ to the projection onto the larger set, $d_L(Y) = \|\Pi_{\Theta_L}(Y) - \theta^\star \|$, is smaller than the corresponding distance from $\theta^\star$ to the projection onto the smaller set, $d_S(Y)= \|\Pi_{\Theta_S}(Y) - \theta^\star \|$. The shaded region denotes the set $\mcG = \{ y: d_L(y) < d_S(y)\}$ of all points for which projecting onto the small set is relatively worse. When the noise level $\sigma$ is sufficiently large, the Gaussian centered at $\theta^\star$ places enough mass on $\mcG$ for the risk associated with projection onto $\Theta_L$ to be smaller than that associated with projection onto $\Theta_S$.
}
    \label{fig:goodregions}
\end{figure}

\subsection{Related literature}

Estimating a signal from Gaussian-corrupted observations is a canonical problem in both statistics and signal processing, where it appears under the names \emph{Gaussian sequence model} and \emph{Gaussian denoising}, respectively. Accordingly, the associated literature is vast. We refer the reader to the monographs of \cite{johnstone2019gaussian} and \cite{tsybakov2009nonparametric} and the references therein for comprehensive treatments. 

A major advance in understanding the risk of the LSE was made in the seminal work of \cite{Chatterjee14}, which provided the first two-sided, pointwise characterization of the risk over general closed, convex constraint sets. That work also established that, although the LSE is not minimax optimal in general, it nevertheless satisfies a form of admissibility. Moreover, \cite{Chatterjee14} emphasized the Gaussian sequence model as a unifying abstraction for studying prominent constrained estimation procedures, including the LASSO and isotonic regression.

Complementary lines of work have highlighted the role of convex and conic geometry in governing the risk of the LSE. In particular, \cite{oymak2016sharp} characterized the vanishing-noise risk of the LSE in terms of the statistical dimension of the tangent cone at the true parameter. The broader role of statistical dimension and related geometric quantities was formalized in the influential work of \cite{amelunxen2014living}, which established sharp phase transition phenomena for the related Gaussian linear measurement model with random Gaussian design. Closely related ideas underlie precise risk analyses for broader classes of estimators, including generalized LASSO and proximal denoising procedures; see, for example, \cite{oymak2013squared}. More recently, these geometric techniques have been extended to characterize the risk of the LSE in the Gaussian linear measurement model in the high-dimensional asymptotic limit; see~\cite{han2023noisy}.

Despite this extensive body of work, the risk reversal phenomenon studied in this paper --- where enlarging a convex constraint set leads to a decrease in risk --- has not, to the best of our knowledge, been previously investigated. Existing analyses of the constrained LSE have primarily focused on suboptimality from a minimax perspective; see, for instance, the recent works of \cite{prasadan2025some}, \cite{aolaritei2025revisiting}, \cite{Matey23} and the discussions therein. Important works in this area include \cite{birge1993rates, kur2020suboptimality, kur2023variance, kur2024convex}, and \cite{zhang2013nearly}. In contrast, our results reveal a qualitatively different failure mode that is not captured by existing minimax analyses of the constrained LSE and persists even when both constraint sets are well-specified.

A related but conceptually distinct phenomenon appears in \cite{fang2019risk}. There, the authors show that for polyhedral constraint sets, the normalized risk of the constrained LSE can be smaller in the misspecified setting, when the true parameter lies outside the constraint set, than in the well-specified case. This comparison uses a notion of misspecified risk that differs from standard squared-error risk measured relative to the true parameter. This behavior is driven by misspecification bias and should be contrasted with the present work, which concerns well-specified models, nested constraint sets, and standard squared-error risk.

Finally, we note a separate line of work --- unrelated to the focus of this paper --- that studies a different notion of \textit{risk monotonicity} studied by \cite{bousquet2022monotone}, \cite{mhammedi2021risk} and \cite{viering2020making}, concerned with how risk evolves as the sample size increases. This notion should not be confused with the risk reversal phenomenon considered here, which arises from geometric effects under fixed sample size and varying constraint sets.

\subsubsection{Chatterjee's risk characterization}
The work of \cite{Chatterjee14} gives one of the broadest available risk characterizations for constrained least squares in the Gaussian sequence model. Given this level of generality, it is natural to ask whether Chatterjee's framework can detect, or rule out, the risk reversal phenomenon studied here. We briefly explain why the resulting one-parameter characterization does not
appear sharp enough for this purpose.

Chatterjee's results are stated in the unit-noise model
$(Y=\theta+Z,\ Z\sim N(0,I_d))$ for the LSE
$\hattheta=\Pi_\Theta(Y)$, where $\Theta\subseteq\mathbb R^d$ is closed and
convex. The risk of the LSE is characterized in terms of the single scale
parameter $t_\theta$, defined as the unique maximizer of
\[
    f_\theta(t)
    =
    \E
    \bigg[
        \sup_{\nu\in\Theta:\ \|\nu-\theta\|\le t}
        \langle Z,\nu-\theta\rangle
    \bigg]
    -
    \frac{t^2}{2}.
\]
Specifically, when $t_\theta\ge1$, Corollary~1.2 of \cite{Chatterjee14} gives the two-sided
characterization
\[
    \E_\theta\|\hattheta-\theta\|^2
    =
    t_\theta^2+O(t_\theta^{3/2}).
\]
For the general noise model $Y=\theta+\sigma Z$, one applies this result to
$\Theta'=\Theta/\sigma$ and $\theta'=\theta/\sigma$, and then multiplies
the resulting unit-noise risk bound by $\sigma^2$. As $\sigma\to0^+$, the
rescaled set $\Theta'$ expands around $\theta'$, so the relevant local
geometry is the tangent cone at $\theta$; see
Section~\ref{sec:vanishing-noise-asymptotic} for a precise definition.

However, this characterization is too coarse for the vanishing-noise comparison.
In that regime, we show that the leading $\sigma^2$ term in the risk expansion is monotone under
nested constraints. Thus any vanishing-noise risk reversal can occur only in a
leading-order degenerate case, where the terms of order $\sigma^2$ agree and the sign
of the risk difference is determined by higher-order corrections. Equivalently,
after rescaling to the unit-noise model, one would need to distinguish risks
that differ by $o(1)$. Since Chatterjee's error bound is of order $t_{\theta'}^{3/2}$ in the regime $t_{\theta'}\ge1$, it is too coarse to resolve the higher-order differences relevant here.

The other regime, $t_\theta<1$, is even less informative for the present purposes. In that case, Corollary~1.2 of \cite{Chatterjee14} gives only a
universal constant-order bound in the unit-noise model. After rescaling back to
noise level $\sigma$, this becomes an $O(\sigma^2)$ upper bound for the
original risk. For bounded constraint sets this can be quite loose in the
diverging-noise regime, where the actual risk remains bounded as
$\sigma\to\infty$.

This issue arises naturally for the bounded polytopes in
the running example. Let
$\operatorname{rad}_\theta(\Theta):=
\sup_{\nu\in\Theta}\|\nu-\theta\|$. Since the supremum in $f_\theta(t)$
saturates once $t\ge \operatorname{rad}_\theta(\Theta)$, the maximizer
satisfies $t_\theta\le \operatorname{rad}_\theta(\Theta)$. Applied to the
rescaled problem, this gives
\[
    t_{\theta'}(\Theta')
    \le
    \frac{\operatorname{rad}_\theta(\Theta)}{\sigma}.
\]
Thus, once $\sigma$ exceeds the geometric radius of $\Theta$, the optimizer
$t_{\theta'}$ necessarily lies below $1$, placing the problem in the
small-$t$ regime of Chatterjee's result.

For the running example, one can compute explicitly that
\[
    t_0\!\left(\frac{\Theta_S}{\sigma}\right)
    =
    \min\left(
        \frac{1}{\sqrt{2\pi}},
        \frac{\|v_2\|}{\sigma}
    \right)
    <1
    \qquad
    \text{for all }\sigma>0,
\]
and
\[
    t_0\!\left(\frac{\Theta_L}{\sigma}\right)
    \le
    \frac{\max(\|v_2\|,\|v_3\|)}{\sigma}.
\]
Hence $t_0(\Theta_L/\sigma)<1$ whenever $\sigma$ exceeds a fixed constant
depending only on the geometry of $\Theta_L$. In this regime, Chatterjee's
corollary yields only an $O(\sigma^2)$ upper bound after rescaling, which is
too coarse to resolve the constant-order diverging-noise limits or the
directional comparisons that drive the risk reversal in this paper.

Thus, although Chatterjee's characterization is highly useful for identifying
the scale of the risk, it does not seem directly suited to detecting the risk
reversal phenomenon studied here.

\section{Asymptotic risk of the least squares estimator} \label{sec:asymptotic-analysis}
Motivated by the findings in Section~\ref{sec:GSM-Counterexample}, we turn to a more systematic investigation of the risk reversal phenomenon. Our goal is to formalize how the risk of the LSE depends on the noise level and the geometry of the constraint set. To this end, we analyze the estimator's behavior in two complementary asymptotic regimes: the vanishing-noise regime $(\sigma \to 0^+)$ and the diverging-noise regime $(\sigma \to \infty)$.

In the vanishing-noise regime, we show that risk reversal cannot occur at the leading $\sigma^2$ scale as $\sigma\to0^+$: for nested correctly specified constraints, the leading-order risk under the smaller constraint is no larger than that under the larger one.
 In contrast, in the diverging-noise regime, we demonstrate that risk reversal can arise much more broadly, extending well beyond the specific example considered in Section~\ref{sec:GSM-Counterexample}. We then use this characterization to construct an example exhibiting worst-case risk reversal, showing that the phenomenon does not merely exist in a pointwise sense.

Taken together, these results clarify the regimes in which risk reversal is structurally possible.

\subsection{Vanishing-noise regime}\label{sec:vanishing-noise-asymptotic}
The risk of the constrained LSE is well understood in the vanishing-noise regime. In particular,~\cite{oymak2016sharp} give a sharp characterization of the limiting risk in terms of the tangent cone at the true parameter. We briefly recall the relevant definitions and results.

For a set $A \subseteq \R^d$ and a vector $b \in \R^d$, we write $\overline{A}$ for the closure of $A$ and $A - b := \{a - b : a \in A\}$. Let $\Theta \subseteq \R^d$ be a compact, convex set and let $\theta \in \Theta$. The \emph{tangent cone} of $\Theta$ at $\theta$ is defined as
\[
T_\Theta(\theta)
:=
\overline{\bigcup_{t \ge 0} t(\Theta - \theta)}.
\]
Since $\Theta$ is compact and convex, $T_\Theta(\theta)$ is a closed convex cone. The \emph{statistical dimension} of this cone is given by
\[
\delta(T_\Theta(\theta))
:=
\E\|\Pi_{T_\Theta(\theta)}(Z)\|^2,
\qquad
Z \sim N(0,I_d).
\]

\cite[Theorem~3.1]{oymak2016sharp} show that, as $\sigma \to 0^{+}$,
\begin{align*}
R_\sigma(\theta; \Theta)
=
\sigma^2\delta(T_\Theta(\theta))
+ o(\sigma^2).
\end{align*}
In Section~\ref{sec:proofs}, we provide a separate, self-contained treatment of this setting in Proposition~\ref{prop:smallSigmaGeneral}, based on a different proof strategy. Our argument relies on classical results on the one-sided directional differentiability of Euclidean projections onto closed convex sets, see \cite{zarantonello1971projections}, \cite{noll1995directional} and \cite{shapiro2016differentiability}. This perspective complements existing analyses and may be of independent interest.

The statistical dimension is monotone under inclusion of closed convex cones \cite[Proposition~3.1]{amelunxen2014living}. In particular, if $\Theta_S\subseteq\Theta_L$ and $\theta^\star\in\Theta_S$, then $T_{\Theta_S}(\theta^\star) \subseteq T_{\Theta_L}(\theta^\star)$ and hence $\delta\!\left(T_{\Theta_S}(\theta^\star)\right) \le \delta\!\left(T_{\Theta_L}(\theta^\star)\right).$ An immediate consequence is that risk reversal cannot occur at order $\sigma^2$ as $\sigma \to 0^+$. 

\begin{theorem}\label{thm:NoRRinSmallNoise} 
Let $\Theta_S \subsetneq \Theta_L$ be nonempty, compact, and convex subsets, and suppose that $\theta^\star\in\Theta_S$. Then
\[
\lim_{\sigma\to0^+}
\frac{
R_\sigma(\theta^\star;\Theta_S)
-
R_\sigma(\theta^\star;\Theta_L)
}{\sigma^2}
\le 0.
\]
\end{theorem} 

\begin{remark}[Interior points in the vanishing-noise limit]
For interior points $\theta^\star \in \operatorname{int}(\Theta)$, standard asymptotic theory for the MLE implies that as $\sigma \to 0^{+}$
\[
\frac{R_\sigma(\theta^\star; \Theta)}{\sigma^2} \to d.
\]
Geometrically, the tangent cone of an interior point satisfies $ T_\Theta(\theta^\star)=\R^d$, since an infinitesimal displacement in any direction remains feasible. Consequently, the associated statistical dimension reduces to
\[
\delta(T_\Theta(\theta^\star))=\delta(\R^d)=\E\|Z\|^2=d.
\]
In particular, if $\theta^\star \in \operatorname{int}(\Theta_S)$, then
$\theta^\star \in \operatorname{int}(\Theta_L)$, and
\[
R_\sigma(\theta^\star; \Theta_S) \sim R_\sigma(\theta^\star; \Theta_L)
\qquad \text{as } \sigma \to 0^{+}.
\]
It follows that any strict risk separation between the two estimators at the leading $\sigma^2$ scale must be driven by boundary effects. 
\end{remark}

We now compute the vanishing-noise risk for the running example. 

We derive the vanishing-noise asymptotics from the tangent-cone characterization, rather than from the explicit risk formulas in Theorem~\ref{thm:hatthetaLExactRisk} and
Theorem~\ref{thm:hatthetaSExactRisk}.

\begin{corollary}\label{cor:VanishingNoiseThetaSLRisks}
In the running example, as $\sigma \to 0^{+}$,
\begin{align*}
R_\sigma(\theta^\star; \Theta_S)
&=
\frac{\sigma^2}{2}
+ o(\sigma^2), \\
R_\sigma(\theta^\star; \Theta_L)
&=
\sigma^2\left(\frac{1}{2} + \frac{1}{\pi}\arctan \frac{1}{c} \right)
+ o(\sigma^2).
\end{align*}

In particular, since $c>0$, we have $R_\sigma(\theta^\star; \Theta_S)
<
R_\sigma(\theta^\star; \Theta_L)$ for all sufficiently small $\sigma>0$.
\end{corollary}

\begin{remark}[Higher-order vanishing-noise behavior] 
We emphasize that the vanishing-noise result above is a leading-order statement. It
rules out risk reversal at the $\sigma^2$ scale, but it does not completely
preclude the possibility of risk reversal at smaller asymptotic scales. Such a
phenomenon could only arise in a degenerate first-order situation, necessarily
when $\delta\!\left(T_{\Theta_S}(\theta^\star)\right)
    =
    \delta\!\left(T_{\Theta_L}(\theta^\star)\right)$, so that the sign of the risk difference is determined by terms beyond the tangent-cone approximation.

It remains an interesting open question whether such examples exist.
Addressing this would require higher-order expansions of the metric projection
map. The proof of Proposition~\ref{prop:smallSigmaGeneral}, which derives the
leading vanishing-noise behavior from first-order directional differentiability of
Euclidean projections, suggests a natural starting point for such an analysis.
\end{remark}

\subsection{Diverging-noise regime} \label{sec:diverging-noise-asymptotic}

In the preceding section, it was demonstrated that in the vanishing-noise regime, the leading-order term of the LSE risk is determined completely by the statistical dimension, which is monotone under set inclusion. As a consequence, risk reversal must arise from either higher-order or non-local effects. This observation, along with our results in Section~\ref{sec:GSM-Counterexample}, suggests that moderate-to-large noise is the natural regime in which this phenomenon becomes visible at leading order.

In this section, we study the diverging-noise regime, in which the LSE is governed by the global geometry of the constraint set. Our main result, Theorem~\ref{thm:large-sigma-projection}, shows that as $\sigma\to\infty$, the estimator concentrates on exposed faces of $\Theta$ selected by the random direction $U=Z/\|Z\|$, and that its limiting risk is determined by the position of $\theta^\star$ relative to these faces.

 We examine the implications of this result in Section~\ref{ssec:generalGeometricCriterion}, deriving a general geometric criterion for diverging-noise risk reversal. In Section~\ref{sec:cvxPolytopes}, we specialize this criterion to compact convex polytopes, where the limiting risk has a tractable representation as a weighted average over vertices. These results provide the machinery for the worst-case analysis in Section~\ref{sec:sup-risk}. Finally, Section~\ref{ssec:beyondGaussianSqLoss} shows that the same mechanism persists beyond the Gaussian squared-loss setting, for radial likelihoods and radial risks whose estimators coincide with Euclidean projection, generalizing the results of the previous two sections.

Our first result applies in the general setting and characterizes the almost sure behavior of the LSE, together with its risk, in the diverging-noise regime.

\begin{theorem}\label{thm:large-sigma-projection}
Let $\Theta\subseteq\R^d$ be nonempty, compact and convex and fix $\theta^\star\in\Theta$.
Let $Z\sim N(0,I_d)$ and for $\sigma>0$ define $\hattheta_\sigma := \Pi_\Theta(\theta^\star+\sigma Z)$. Set
\[
U:=Z/\|Z\|,
\qquad
F_\Theta(u):=\argmax_{\theta\in\Theta}\langle \theta,u\rangle.
\]
Equivalently, $F_\Theta(u)$ is the exposed face of $\Theta$ in direction $u$. 
Then, almost surely as $\sigma\to\infty$,
\begin{align}\label{eq:large-sigma-a.s}
\hattheta_\sigma \to \Pi_{F_\Theta(U)}(\theta^\star),
\end{align}
and consequently, for $R_\infty(\theta^\star; \Theta) := \E\|\Pi_{F_\Theta(U)}(\theta^\star)-\theta^\star\|^2$,
\begin{align}
R_\sigma(\theta^\star; \Theta) \to R_\infty(\theta^\star; \Theta).
\label{eq:large-sigma-risk}
\end{align}
\end{theorem}

    Theorem~\ref{thm:large-sigma-projection} shows that, as $\sigma\to\infty$, the quadratic penalty in the projection problem becomes negligible relative to the linear term induced by the noise direction $U$. Consequently, the estimator concentrates on the exposed face $F_\Theta(U)$, namely the part of $\Theta$ that is extremal in this random direction. The limiting estimator is therefore obtained by first choosing a face of $\Theta$ according to $U$, and then projecting $\theta^\star$ onto that face. This contrasts sharply with the vanishing-noise regime and shows why, when the noise level diverges, risk comparisons become sensitive to the global embedding of nested constraint sets. Because $F_\Theta(u)$ is nonempty, closed, and convex for every $u \in \mcS^{d-1}$, the projection $\Pi_{F_\Theta(U)}(\theta^\star)$ is well defined.

    \subsubsection{A general geometric criterion for risk reversal} \label{ssec:generalGeometricCriterion}
    We now use Theorem~\ref{thm:large-sigma-projection} to obtain a general geometric criterion for risk reversal in the diverging-noise regime. Since the limiting estimator is obtained by projecting $\theta^\star$ onto the exposed face selected by the random direction $U$, the limiting risk is determined by the average squared distance from $\theta^\star$ to these exposed faces. This motivates the following distance functional. For a compact convex set $\Theta$, define
\[
    \rho_\Theta(u;\theta^\star)
    :=
    \operatorname{dist}\!\left(\theta^\star,F_\Theta(u)\right)
    =
    \|\Pi_{F_\Theta(u)}(\theta^\star)-\theta^\star \|.
\]
Thus $\rho_\Theta(u;\theta^\star)$ is the distance from the signal $\theta^\star$ to the face of $\Theta$ selected by the direction $u$. 

Since $U=Z/\|Z\|\sim\operatorname{Unif}(\mcS^{d-1})$,
Theorem~\ref{thm:large-sigma-projection} gives
\[
\begin{aligned}
    R_\infty(\theta^\star;\Theta_S)
    -
    R_\infty(\theta^\star;\Theta_L)
    &=
    \E_U\left[
        \rho_{\Theta_S}(U;\theta^\star)^2
        -
        \rho_{\Theta_L}(U;\theta^\star)^2
    \right].
\end{aligned}
\]
Thus, in the diverging-noise limit, risk reversal occurs precisely when the
right-hand side is strictly positive; equivalently, when the exposed face of
the smaller set is farther from $\theta^\star$, on average over a uniformly
random direction, than the exposed face of the larger set.

Whenever this limiting risk gap is strictly positive, the convergence in Theorem~\ref{thm:large-sigma-projection} implies that $R_\sigma(\theta^\star;\Theta_S)>R_\sigma(\theta^\star;\Theta_L)$ for all sufficiently large $\sigma$. The next two corollaries apply this limiting-risk representation in two complementary geometries.

First, Corollary~\ref{cor:balls} shows that nested Euclidean balls with a
common center cannot exhibit diverging-noise risk reversal. In that case, every
direction exposes the boundary point lying in that direction from the center.
Enlarging the ball simply moves each exposed point radially outward by the same
amount. Since the signal lies inside the smaller ball, this increases the
average squared exposed-face distance, so tightening the ball can only decrease
the diverging-noise risk. Second, Corollary~\ref{cor:ellipse} shows that risk reversal can occur in the case of centered axis-aligned ellipses. We consider two such ellipses with the same long horizontal semi-axis. The smaller ellipse is relatively thinner in the vertical direction. Increasing the vertical semi-axis enlarges and rounds the ellipse. As a result, many directions that previously exposed points near the horizontal tips instead expose points closer to the origin. Hence the average exposed-face distance to the signal decreases, and the larger ellipse has smaller diverging-noise risk. This example is depicted in Figure~\ref{fig:ellipse-risk-reversal}.

\begin{corollary}\label{cor:balls}
Let $x_0\in\R^d$ and, for $r>0$, define $B(x_0,r):=\{y\in\R^d:\|y-x_0\|\le r\}$. Let $0<r_S<r_L$, and set
\[
    \Theta_S:=B(x_0,r_S),
    \qquad
    \Theta_L:=B(x_0,r_L).
\]
Then $\Theta_S\subsetneq\Theta_L$, and for every
$\theta^\star\in\Theta_S$, $R_\infty(\theta^\star;\Theta_S) < R_\infty(\theta^\star;\Theta_L)$.
\end{corollary}

\begin{corollary} \label{cor:ellipse}
For $a,b>0$, define the centered axis-aligned ellipse in $\R^2$ by
\[
    \mcE(a,b)
    :=
    \left\{
        y \in\mathbb R^2:
        \frac{y_1^2}{a^2}+\frac{y_2^2}{b^2}\le1
    \right\}.
\]
Let 
\[ \Theta_S:=\mcE(a,b_S), \qquad \Theta_L:=\mcE(a,b_L), \qquad 0<b_S<b_L<\frac a2. 
\]
Then $\Theta_S \subsetneq \Theta_L$ and $R_\infty(0; \Theta_S)
    >
    R_\infty(0;\Theta_L)$.
\end{corollary}

\begin{remark}[Risk reversal is not a boundary or dimensionality artifact]
Recall that our running example has two special features: the signal is placed
at an extreme point, $\theta^\star=v_1$, and the smaller and larger constraint sets have different dimensions,
\[
    \dim(\Theta_S)=1<2=\dim(\Theta_L).
\]
Corollary~\ref{cor:ellipse} demonstrates that neither feature is necessary for the phenomenon to occur. Specifically, risk reversal is possible even for nested smooth, full-dimensional convex bodies and even when $\theta^\star\in \operatorname{int}(\Theta_S) \cap \operatorname{int}(\Theta_L)$.
This contrasts with the vanishing-noise behavior discussed in the remark following Theorem~\ref{thm:NoRRinSmallNoise}. When $\theta^\star$ lies in the
interior of the smaller constraint set, both constraints are locally inactive
near $\theta^\star$, and the two risks are asymptotically identical in that regime.
\end{remark}

\begin{figure}[t]
\centering
\resizebox{0.98\textwidth}{!}{%
\begin{tikzpicture}
    \begin{groupplot}[
        group style={group size=2 by 1, horizontal sep=2.2cm},
        width=0.46\textwidth,
        height=0.34\textwidth,
        axis lines=middle,
        xmin=-4.8, xmax=4.8,
        ymin=-2.6, ymax=2.6,
        xtick=\empty,
        ytick=\empty,
        axis line style={->},
        clip=false
    ]


    \nextgroupplot[
    axis equal image,
]

\def\a{4}
\def\bS{0.8}
\def\bL{1.6}
\def\ticklen{0.18}

\colorlet{dirB}{violet!80!black}
\colorlet{red}{red!95!black}

\addplot[
    domain=0:360,
    samples=200,
    smooth,
    thick,
    blue!60!black
]
({\a*cos(x)},{\bL*sin(x)});

\addplot[
    domain=0:360,
    samples=200,
    smooth,
    thick,
    orange
]
({\a*cos(x)},{\bS*sin(x)});

\node[blue!60!black, font=\scriptsize] at (axis cs:2.5,-1.8) {$\Theta_L$};
\node[orange!97!black, font=\scriptsize] at (axis cs:1.1,-1.15) {$\Theta_S$};

\pgfmathsetmacro{\phib}{73}
\pgfmathsetmacro{\denSb}{sqrt(\a*\a*cos(\phib)^2 + \bS*\bS*sin(\phib)^2)}
\pgfmathsetmacro{\xSb}{(\a*\a*cos(\phib))/\denSb}
\pgfmathsetmacro{\ySb}{(\bS*\bS*sin(\phib))/\denSb}
\pgfmathsetmacro{\denLb}{sqrt(\a*\a*cos(\phib)^2 + \bL*\bL*sin(\phib)^2)}
\pgfmathsetmacro{\xLb}{(\a*\a*cos(\phib))/\denLb}
\pgfmathsetmacro{\yLb}{(\bL*\bL*sin(\phib))/\denLb}

\draw[dirB,thick, ->]
    (axis cs:0,0) -- (axis cs:{3.2*cos(\phib)},{3.2*sin(\phib)});
\node[dirB, anchor=west, font=\scriptsize]
    at (axis cs:{3.35*cos(\phib)},{3.35*sin(\phib)}) {$u$};

\draw[red, dashed, thick]
    (axis cs:0,0) -- (axis cs:\xSb,\ySb);
\draw[red, dashed, thick]
    (axis cs:0,0) -- (axis cs:\xLb,\yLb);

\draw[red, thick]
    (axis cs:{\xSb-\ticklen*sin(\phib)},{\ySb+\ticklen*cos(\phib)})
    --
    (axis cs:{\xSb+\ticklen*sin(\phib)},{\ySb-\ticklen*cos(\phib)});

\draw[red, thick]
    (axis cs:{\xLb-\ticklen*sin(\phib)},{\yLb+\ticklen*cos(\phib)})
    --
    (axis cs:{\xLb+\ticklen*sin(\phib)},{\yLb-\ticklen*cos(\phib)});

\node[text=blue!80!black, font=\scriptsize] at (axis cs:0,0) {$\star$};
\node[blue!80!black, below left, font=\scriptsize] at (axis cs:0,0) {$\theta^\star$};

\nextgroupplot[
    width=0.40\textwidth,
    height=0.30\textwidth,
    xlabel={$b$},
    ylabel style={
        anchor=south,
        font=\scriptsize
    },
    xlabel style={
    anchor=north,
    font=\scriptsize
    },
    xmin=0, xmax=4.5,
    ymin=5.7, ymax=17,
    xtick={0,0.8,1.6,2},
    xticklabels={$0$,$b_S$,$b_L$,$\frac{a}{2}$},
]
\node[anchor=east, font=\scriptsize, xshift=-4pt]
    at (axis description cs:0.0,1.0)
    {$R_\infty(0;\mcE(a,b))$};
\addplot[
    domain=0:3.9,
    samples=200,
    thick,
    black
]
{(16 + x^2 - 4*x)};

\pgfmathsetmacro{\riskS}{(16 + 0.8^2 - 4*0.8)}
\pgfmathsetmacro{\riskL}{(16 + 1.6^2 - 4*1.6)}

\addplot[dashed, orange!97!black] coordinates {(0.8,5.7) (0.8,\riskS)};
\addplot[dashed, blue!60!black] coordinates {(1.6,5.7) (1.6,\riskL)};
\addplot[densely dashed, gray] coordinates {(2,5.7) (2,16.3)};

\addplot[mark=*, only marks, mark size=1.6pt, orange!97!black]
    coordinates {(0.8,{(16 + 0.8^2 - 4*0.8)})};
\addplot[mark=*, only marks, mark size=1.6pt, blue!60!black]
    coordinates {(1.6,{(16 + 1.6^2 - 4*1.6)})};

\addplot[dashed, orange!97!black]
    coordinates {(0,\riskS) (0.8,\riskS)};
\addplot[dashed, blue!60!black]
    coordinates {(0,\riskL) (1.6,\riskL)};

\node[orange!97!black, font=\scriptsize, anchor=east, xshift=-2pt]
    at (axis cs:0,\riskS)
    {$R_\infty(0;\Theta_S)$};

\node[blue!60!black, font=\scriptsize, anchor=east, xshift=-2pt]
    at (axis cs:0,\riskL)
    {$R_\infty(0;\Theta_L)$};

    \end{groupplot}
\end{tikzpicture}
}
\caption{Risk reversal for nested centered ellipses, as constructed in Corollary~\ref{cor:ellipse}. In the left panel, the arrow $u$ denotes an exposing
direction. For this direction, the exposed faces $F_{\Theta_S}(u)$ and
$F_{\Theta_L}(u)$ are singletons, whose locations are indicated by the
small red tangent-line markers. The dashed red segments from
$\theta^\star$ to these exposed points indicate the distances that enter the
diverging-noise risk. The right panel plots $R_\infty(0;\mcE(a,b))$ as a function
of the vertical semi-axis $b$, with $a$ fixed. On the interval $0<b<\frac{a}{2}$, this risk decreases as $b$ increases, so the larger ellipse has
smaller diverging-noise risk.} 
\label{fig:ellipse-risk-reversal}
\end{figure}

\subsubsection{Discrete selection: convex polytopes}\label{sec:cvxPolytopes}
We now restrict to the setting in which $\Theta$ is a convex polytope, that is, 
\[
\Theta = \operatorname{conv}(v_1,\dots,v_K),
\]
where $v_1,\dots,v_K \in \R^d$ are its vertices. This setting includes the running example and covers many standard constraint sets that arise in stochastic optimization, such as probability simplices, $\ell_1$-balls, and bounded feasible regions defined by finitely many linear inequalities. Moreover, the exposed face $F_\Theta(U)$ appearing in
Theorem~\ref{thm:large-sigma-projection} now takes a particularly tractable form:
almost surely, it is a single vertex of $\Theta$. This allows the
diverging-noise risk to be written as a weighted average of squared distances
from $\theta^\star$ to the vertices, shedding further light on the geometric
mechanism behind risk reversal.

Before stating the main result of this section, we introduce some additional notation: let $\mu_{d-1}$ denote the $(d-1)$-dimensional surface
measure on the unit sphere $\mcS^{d-1}$. The uniform distribution $U$ on $\mcS^{d-1}$ is characterized for any measurable $A\subseteq \mcS^{d-1}$ by
\[
\P(U\in A)=\frac{\mu_{d-1}(A)}{\mu_{d-1}(\mcS^{d-1})}.
\]
Further, the normal cone of $\Theta$ at a vertex $v_i$ is 
\[
\mcN_\Theta(v_i)
=
\{ u \in \R^d :
\langle u, v_j - v_i \rangle \le 0
\ \text{for all } j = 1,\dots,K
\}.
\]

\begin{lemma}\label{lem:polytopeMaxAtVertex}
Let $\Theta \subseteq \R^d$ be a nonempty compact convex polytope with vertices
$v_1,\dots,v_K$. There exists a random index $I \in \{1,\dots,K\}$ such that,
almost surely,
\[
F_\Theta(U) = \{v_I\}.
\]
Consequently, in the diverging-noise regime $\sigma\to\infty$, almost surely,
\[
\hattheta_\sigma \to v_I.
\]
Moreover, the limiting risk admits the representation
\[
R_\infty(\theta; \Theta)
=
\sum_{i=1}^K p_i  \|v_i-\theta\|^2,
\qquad
p_i := \P(I=i)
=
\frac{\mu_{d-1}(\mcN_\Theta(v_i)\cap \mathcal S^{d-1})}
{\mu_{d-1}(\mathcal S^{d-1})}.
\]
\end{lemma}

\begin{remark}[Vertex selection and risk redistribution]
In the diverging-noise regime, projection onto a convex polytope asymptotically reduces to a discrete selection problem among its vertices. By Lemma~\ref{lem:polytopeMaxAtVertex}, for each vertex $v_i$ of the polytope $\Theta$, the event $\{I=i\}$ occurs precisely when the random direction $U=Z/\|Z\|$ lies in the normal cone $\mcN_\Theta(v_i)$. Thus, the estimator selects a single vertex $v_I$, and the risk is the expected squared error incurred by this random vertex choice.

From this perspective, what governs the risk is not only which vertices are close to $\theta^\star$, but how often they are selected. Even if a vertex $v_i$ lies close to $\theta^\star$ in squared distance, its contribution to the risk can increase or decrease depending on whether the probability $p_i$ of selecting it grows or shrinks.

This viewpoint clarifies how risk reversal can occur. Consider nested polytopes
$\Theta_S\subsetneq\Theta_L$, and suppose for concreteness that $\Theta_L=\operatorname{conv}(\Theta_S\cup\{v\})$ for some additional vertex $v$. Adding $v$ introduces a new normal cone with positive spherical measure. Since the normal cones partition $\mcS^{d-1}$ up to a null set, this necessarily redistributes selection probability across vertices: some vertices lose probability mass while others gain it. Risk reversal arises when this redistribution transfers probability mass away from vertices that are far from $\theta^\star$ toward vertices that are closer to $\theta^\star$, thereby reducing risk. Conversely, shrinking the constraint can diminish the selection region of a nearby vertex, forcing its probability mass to be reassigned to more distant vertices and increasing risk.

Figure~\ref{fig:angular-mass-redistribution} illustrates this mechanism in the running example with $c=0.5$.
\end{remark}

\begin{figure}
    \centering
\begin{tikzpicture}[scale=1.05, line cap=round, line join=round, font=\small]

\definecolor{myblue}{RGB}{0,114,178}
\definecolor{myorange}{RGB}{230,159,0}
\definecolor{mygreen}{RGB}{0,158,115}

\pgfmathsetmacro{\phi}{63.435}
\pgfmathsetmacro{\titley}{1.8}

\begin{scope}[shift={(-4.4,0)}]
    \node[font=\small] at (1, \titley) {Constraint sets};
    \coordinate (v1) at (0,0);
    \coordinate (v2) at (2,1);
    \coordinate (v3) at (0,1);

    \draw[thick, black] (v1)--(v3)--(v2);
    \draw[thick, black] (v1)--(v2);

    \fill[myblue] (v1) circle (2.5pt)
    node[below left, black, inner sep=1pt] {\small $\theta^\star=v_1$};
    \fill[myorange] (v2) circle (2.5pt) node[right, black] {\small $v_2$};
    \fill[mygreen]  (v3) circle (2.5pt) node[left, black] {\small $v_3$};
\end{scope}

\begin{scope}[shift={(0,0.45)}]
    \node[font=\small] at (0, \titley-0.45) {Normal sectors: $\Theta_S$};

    \fill[myorange!35] (0,0) -- ({-\phi}:1) arc ({-\phi}:{180-\phi}:1) -- cycle;
    \fill[myblue!35]   (0,0) -- ({180-\phi}:1) arc ({180-\phi}:{360-\phi}:1) -- cycle;

    \draw[thin, gray!50] (0,0) circle (1);
    \draw[myorange, ultra thick] ({-\phi}:1) arc ({-\phi}:{180-\phi}:1);
    \draw[myblue, ultra thick]   ({180-\phi}:1) arc ({180-\phi}:{360-\phi}:1);

    \node[myorange, font=\bfseries] at (0.42,0.42) {$p_2$};
    \node[myblue,   font=\bfseries] at (-0.42,-0.42) {$p_1$};
\end{scope}

\begin{scope}[shift={(3.6,0.45)}]
    \node[font=\small] at (0, \titley-0.45) {Normal sectors: $\Theta_L$};

    \fill[myorange!35] (0,0) -- ({-\phi}:1) arc ({-\phi}:90:1) -- cycle;
    \fill[mygreen!35]  (0,0) -- (90:1) arc (90:180:1) -- cycle;
    \fill[myblue!35]   (0,0) -- (180:1) arc (180:{360-\phi}:1) -- cycle;

    \draw[thin, gray!50] (0,0) circle (1);
    \draw[myorange, ultra thick] ({-\phi}:1) arc ({-\phi}:90:1);
    \draw[mygreen, ultra thick]  (90:1) arc (90:180:1);
    \draw[myblue, ultra thick]   (180:1) arc (180:{360-\phi}:1);

    \node[myorange, font=\bfseries] at (0.48,0.28) {$p_2$};
    \node[mygreen,  font=\bfseries] at (-0.35,0.45) {$p_3$};
    \node[myblue,   font=\bfseries] at (-0.42,-0.42) {$p_1$};
\end{scope}

\end{tikzpicture}
\caption{
Illustration of the mass redistribution in the running example with $c=0.5$. The middle and right panels show the spherical normal-cone sectors
$\mcN_\Theta(v_i)\cap\mathcal S^1$, whose arc lengths determine the vertex selection probabilities $p_i=\mathbb P(I=i)$ in the diverging-noise limit. Passing from $\Theta_S$ to $\Theta_L$ introduces the vertex $v_3$ and reallocates selection probability from both $v_1$ and $v_2$. The reduction in the contribution of the distant vertex $v_2$ outweighs the new contribution from $v_3$, thereby lowering the limiting risk.
}
\label{fig:angular-mass-redistribution}
\end{figure}

We apply the theory developed in this section to compute the diverging-noise risk expressions for the running example. We first record a general computation for a one-parameter family of triangles; the formulas for $\Theta_S$ and $\Theta_L$ will then follow by specialization.

\begin{lemma}\label{lem:ThetaARisk}
Fix $c>0$ and $x \in [0,c^{-1}]$. Define $\Theta_x := \operatorname{conv}(v_1, v_2, v_x)$, where
\[
v_1 = (0,0), \qquad
v_2 = (c^{-1},1), \qquad
v_x = (x,1).
\] 
Then, for $\theta^\star := v_1,$
\begin{align*}
R_\infty(\theta^\star; \Theta_x)
=
\inparen{1+\frac{1}{c^2}} \left(\frac14 + \frac{1}{2\pi} \arctan \frac1c \right)
+ (1+x^2)\left(\frac14 - \frac{1}{2\pi} \arctan x\right).
\end{align*}
\end{lemma}

In the running example, taking $x=0$ and $x=c^{-1}$ in
Lemma~\ref{lem:ThetaARisk} gives the diverging-noise limiting risks for
$\Theta_L$ and $\Theta_S$, respectively. This yields the following
corollary.

\begin{corollary}\label{cor:DivergingNoiseThetaSLRisks}
In the running example, as $\sigma \to \infty$,
\begin{align*}
R_\sigma(\theta^\star; \Theta_S)
&=
\frac{\alpha_c}{2}
+ o(1), \\ 
R_\sigma(\theta^\star; \Theta_L)
&=
\alpha_c \inparen{\frac14 + \frac{1}{2\pi} \arctan \frac1c } + \frac14
+ o(1),
\end{align*}
where $\alpha_c = 1+\frac{1}{c^2}$. Consequently, if $0<c<1$, then $R_\sigma(\theta^\star; \Theta_S)
>
R_\sigma(\theta^\star; \Theta_L)$ for all sufficiently large $\sigma$.
\end{corollary}

\begin{remark}[A quantitative diverging-noise threshold]
The preceding corollary shows that risk reversal holds for all sufficiently
large $\sigma$, but does not provide an explicit threshold. In
Appendix~\ref{sec:quantitative-bounds-polytopes}, we make this implication quantitative for general compact convex polytopes by giving an explicit sufficient
diverging-noise threshold. The argument gives a uniform $O(\sigma^{-1})$ bound on the convergence of the finite-noise risks to their diverging-noise limits. This rate is sharp in $\sigma$, although the resulting threshold is
only sufficient and need not coincide with the actual crossing point of the two
finite-$\sigma$ risks. Obtaining sharper finite-noise conditions remains an interesting direction for future work.
\end{remark}

\subsubsection{Worst-case risk}\label{sec:sup-risk}
In this section, we study worst-case risk reversal and develop the technical tools needed to prove Theorem~\ref{thm:worst-case-risk-reversal}. Recall that worst-case risk reversal occurs when two nested constraint sets $\Theta_S \subsetneq \Theta_L$ satisfy
\begin{align*}
    \sup_{\theta \in \Theta_S} R_\sigma(\theta; \Theta_S)
    >
    \sup_{\theta \in \Theta_L} R_\sigma(\theta; \Theta_L),
\end{align*}
so that the constrained estimator associated with the smaller feasible set has strictly larger worst-case risk.

Our first step is to strengthen the pointwise diverging-noise convergence of the risk function established in Theorem~\ref{thm:large-sigma-projection} to uniform convergence over the parameter space.

\begin{lemma}\label{lem:uniform-large-sigma}
Let $\Theta\subseteq\R^d$ be nonempty, compact, and convex. Then, as
$\sigma\to\infty$,

\[
\sup_{\theta\in\Theta}
\left|
R_\sigma(\theta;\Theta)-R_\infty(\theta;\Theta)
\right|
\to 0.
\]
Consequently, $\sup_{\theta\in\Theta}R_\sigma(\theta;\Theta)
\to \sup_{\theta\in\Theta}R_\infty(\theta;\Theta)$.
\end{lemma}

This uniform convergence result plays a central role in our analysis. It allows us to infer worst-case properties of the finite-noise risk $R_\sigma$ from the corresponding properties of the limiting risk $R_\infty$. In particular, a strict separation between worst-case limiting risks implies the same ordering for all sufficiently large but finite noise levels.

\begin{theorem}\label{thm:finite-sigma-reversal}
Let $\Theta_S \subsetneq \Theta_L$ be nonempty, compact, and convex.
Suppose that
\[
\sup_{\theta\in\Theta_S} R_\infty(\theta; \Theta_S)
>
\sup_{\theta\in\Theta_L} R_\infty(\theta; \Theta_L).
\]
Then there exists $\sigma_1<\infty$ such that for all $\sigma\ge\sigma_1$,
\[
\sup_{\theta\in\Theta_S} R_\sigma(\theta;\Theta_S)
>
\sup_{\theta\in\Theta_L} R_\sigma(\theta;\Theta_L).
\]
\end{theorem}

By Theorem~\ref{thm:finite-sigma-reversal}, to establish worst-case risk reversal for all sufficiently large noise levels, it suffices to identify nested constraint sets $\Theta_S\subsetneq \Theta_L$ whose limiting risks exhibit worst-case reversal. This task is substantially simplified by the geometric structure of convex polytopes.
In particular, Lemma~\ref{lem:polytopeMaxAtVertex} shows that the worst-case limiting risk over a polytope is attained at one of its vertices, reducing the comparison to finitely many explicit calculations. We now demonstrate this construction.

\begin{figure}
    \centering
\begin{tikzpicture}
    \begin{axis}[
        width=10cm, height=4.5cm,
        xlabel={$x$},
        grid=major,
        grid style={dashed, gray!30},
        legend pos=south east,
        legend style={nodes={scale=0.9, transform shape}},
        xmin=0, xmax=1.35,
        ymin=0.8, ymax=1.5,
        line width=1.5pt,
        declare function={
            c = 0.75;
            invc = 1.0/c;
            vcsq = invc^2 + 1;
            p2 = 0.25 + atan(invc)/360;
            px(\x) = 0.25 - atan(\x)/360;
            p1(\x) = 1.0 - p2 - px(\x);
            vxsq(\x) = \x^2 + 1;
            diffsq(\x) = (invc - \x)^2;
            Rv1(\x) = p2 * vcsq + px(\x) * vxsq(\x);
            Rv2(\x) = p1(\x) * vcsq + px(\x) * diffsq(\x);
            Rvx(\x) = p1(\x) * vxsq(\x) + p2 * diffsq(\x);
        }
    ]

    \addplot[blue, solid, domain=0.0001:1.33] { Rv1(x) };
    \addlegendentry{$R_\infty(v_1; \Theta_x)$}

    \addplot[green!60!black, dashed, domain=0.0001:1.33] { Rv2(x) };
    \addlegendentry{$R_\infty(v_2; \Theta_x)$}

    \addplot[purple, dotted, domain=0.0001:1.33] { Rvx(x) };
    \addlegendentry{$R_\infty(v_x; \Theta_x)$}

    \addplot[
        red, 
        opacity=0.3, 
        line width=4pt, 
        domain=0.0001:1.33,
        samples=100
    ] { max(Rv1(x), max(Rv2(x), Rvx(x))) };
    \addlegendentry{$\overline{R}_\infty(\Theta_x)$}

    \draw[black, dashed, thin] (axis cs:0.429, 0) -- (axis cs:0.429, 1.324);
    
    \addplot[only marks, mark=*, mark options={fill=black, scale=0.8}] coordinates {(0.429, 1.324)};
    
    \node[anchor=south] at (axis cs:0.429, 1.35) {\footnotesize $x = 0.4290$};
    \end{axis}
\end{tikzpicture}
\caption{
Vertex risks $R_\infty(v_1;\Theta_x)$, $R_\infty(v_2;\Theta_x)$, and $R_\infty(v_x;\Theta_x)$, together with their upper envelope $\overline{R}_\infty(\Theta_x)$, as functions of $x$ for $c=0.75$. The envelope has a minimum at $x=0.4290$. Increasing $x$ tightens the constraint. The non-monotonicity of the envelope implies worst-case risk reversal (see the remark following Lemma~\ref{lem:explicit-worst-case-risk}).
}
    \label{fig:sup-norm-risk}
\end{figure}

\begin{lemma}\label{lem:explicit-worst-case-risk}
Fix $c=0.75$, $a=0.5$, and $b=1.3$. Define
\[
v_1=(0,0), \qquad
v_2=(c^{-1},1), \qquad
v_a=(a,1), \qquad
v_b=(b,1).
\]
Set $\Theta_S:=\operatorname{conv}(v_1,v_2,v_b)$ and  $\Theta_L:=\operatorname{conv}(v_1,v_2,v_a)$. Then $\Theta_S \subsetneq \Theta_L$ and 
\[
\sup_{\theta\in\Theta_S} R_\infty(\theta;\Theta_S)
>
\sup_{\theta\in\Theta_L} R_\infty(\theta;\Theta_L).
\]
\end{lemma}

\begin{remark}[Worst-case risk envelope]
Figure~\ref{fig:sup-norm-risk} provides geometric intuition for the construction in
Lemma~\ref{lem:explicit-worst-case-risk}.
For each $x\in(0,c^{-1})$, the diverging-noise risk $R_\infty(\theta;\Theta_x)$ is a convex quadratic
function of $\theta$, and hence its maximum over the compact convex set $\Theta_x$
is attained at an extreme point.
As a result, the worst-case risk over $\Theta_x$ reduces to the maximum of the risks
evaluated at the three vertices $v_1$, $v_2$, and $v_x$. The figure plots these three vertex risks, as well as their upper envelope
\[
\overline{R}_\infty(\Theta_x)
:=
\max\{
R_\infty(v_1;\Theta_x),
R_\infty(v_2;\Theta_x),
R_\infty(v_x;\Theta_x)
\},
\]
as functions of $x$ for $c=0.75$.
Increasing $x$ tightens the constraint, since $\Theta_{x_2}\subsetneq \Theta_{x_1}$
whenever $x_2>x_1$.
If tightening the constraint could only improve performance, the envelope
$\overline{R}_\infty(\Theta_x)$ would be non-increasing in $x$. Instead, the envelope initially decreases and then increases as $x$ grows. This behavior underlies the strict worst-case risk reversal established in
Lemma~\ref{lem:explicit-worst-case-risk}.
\end{remark}

\subsubsection{Risk reversal beyond Gaussian noise and squared-error risk}
\label{ssec:beyondGaussianSqLoss}

We have so far worked in the Gaussian sequence model and evaluated risk using
squared-error risk. In this section, we show that neither choice is essential.
We consider radially symmetric location models for which the constrained MLE is
still given by Euclidean projection, and we allow risk to be evaluated using
radial loss functions. Thus the estimator remains likelihood-based, while both the
noise distribution and the loss used to evaluate risk are allowed to vary.

We make this precise as follows. We call a function $q:[0,\infty)\to(0,\infty)$ admissible if it is strictly decreasing and 
\[
    0<C_q:=\int_{\R^d} q(\|w\|)\,dw<\infty .
\]
For such a $q$, let $W_q$ be an $\R^d$-valued random vector with density
\[
    f_q(w)
    =
    C_q^{-1}q(\|w\|),
    \qquad w\in\R^d,
\]
and suppose that
\begin{align}
\label{eq:radial-model}
    Y=\theta^\star+\sigma W_q.
\end{align}
We remark that this class contains many examples of heavy-tailed distributions. For instance, for any $\nu>0$, the choice $q_\nu(r) = (1+\frac{r^2}{\nu})^{-(\nu+d)/2}$ gives the spherically symmetric multivariate Student-$t$ density with $\nu$ degrees of freedom, whose tails decay polynomially rather than exponentially.

For fixed $y$, the likelihood is proportional to $\theta
    \mapsto
    q\left(\frac{\|y-\theta\|}{\sigma}\right)$. Since $q$ is strictly decreasing, maximizing the constrained likelihood is
equivalent to minimizing $\|y-\theta\|$ over $\Theta$. Hence the MLE and the
LSE coincide, and both are equal to the Euclidean projection of $Y$ onto
$\Theta$. We evaluate this estimator using the radial risk
\[
    R_{\sigma}^{(q,\eta)}(\theta^\star;\Theta)
    :=
    \E \frac{1}{\eta}
    \|
        \Pi_{\Theta}(Y)-\theta^\star
    \|^\eta,
    \qquad
    \eta>0.
\]

The case $q(r) = e^{-r^2/2}$ and $\eta=2$ recovers the Gaussian squared-error setting up to a factor of $\frac{1}{2}$. Since the law of $W_q$ depends only on $\|w\|$, it is invariant under
rotations. Consequently, the direction $W_q/\|W_q\|$ is uniformly distributed
on $\mcS^{d-1}$. This is precisely the distributional
input used in the diverging-noise projection limit
Theorem~\ref{thm:large-sigma-projection} and in the subsequent polytopal calculations, including Corollary~\ref{cor:DivergingNoiseThetaSLRisks}. Therefore, the exposed-face probabilities coincide with those in the Gaussian-noise case.

The following result applies this observation to show that
pointwise risk reversal persists in the running example under the
radial location model~\eqref{eq:radial-model}. 

\begin{theorem}
\label{thm:radial-beta-eta}
Consider the running example under the radial
location model~\eqref{eq:radial-model}, and evaluate risk using
$R_{\sigma}^{(q,\eta)}$. For each admissible $q$, the
constrained MLE and constrained LSE coincide, and both are equal to the
Euclidean projection onto the constraint set.

Moreover, for every $\eta>1$, there exists $c_\eta>0$ such that the
following holds. For every admissible $q$ and every
$0<c<c_\eta$, there exists $\sigma_1<\infty$ such that
\[
R_{\sigma}^{(q,\eta)}(\theta^\star;\Theta_S)
>
R_{\sigma}^{(q,\eta)}(\theta^\star;\Theta_L)
\]
for all $\sigma\ge\sigma_1$. If $\eta\ge2$, then the same conclusion holds
for every $0<c<1$.
\end{theorem}

Using a separate construction, the next result shows that the same phenomenon
persists at the level of worst-case risk. Thus, for every radial risk exponent
$\eta>1$, risk reversal is not confined to pointwise comparisons.

\begin{theorem}
\label{thm:radial-worst-case-all-eta}
Under the radial location model~\eqref{eq:radial-model}, fix an admissible
$q$ and $\eta>1$. Then there exist nonempty compact convex sets
$\Theta_S\subsetneq \Theta_L\subseteq\mathbb R^2$ and
$\sigma_1<\infty$ such that
\[
    \sup_{\theta\in\Theta_S}
    R_{\sigma}^{(q,\eta)}(\theta;\Theta_S)
    >
    \sup_{\theta\in\Theta_L}
    R_{\sigma}^{(q,\eta)}(\theta;\Theta_L)
\]
for all $\sigma\ge \sigma_1$.
\end{theorem}

\begin{remark}[Beyond projection-type estimators]
The radial models considered in this section preserve the identity between the
constrained MLE and the Euclidean projection. This keeps the estimator
statistically motivated while allowing us to relax Gaussianity and squared-error
risk without changing the projection geometry. This equivalence is not automatic. For example, if the coordinates of the noise are chosen to be independent Laplace random
variables, so that the density is proportional to $e^{-\|w\|_1/b}$, then the
constrained MLE minimizes the $\ell_1$-distance $\|Y-\theta\|_1$ over
$\theta\in\Theta$, rather than the Euclidean distance $\|Y-\theta\|_2$
minimized by the constrained LSE. 

A natural question is whether risk reversal persists in models for which the constrained MLE is not a Euclidean projection. Such estimators remain likelihood-based, and hence
statistically motivated, but their behavior is no longer governed by the
Euclidean projection geometry studied here. Extending the present theory to
that setting would require new tools for analyzing the geometry of the
constrained likelihood estimator itself.
\end{remark}

\section{Proofs}\label{sec:proofs}
We provide here the proofs and supporting results omitted from the previous
sections.

\subsection{Proofs of results in Section~\ref{sec:GSM-Counterexample}}

\begin{proof}[Proof of Theorem~\ref{thm:counter-example}]
The explicit expressions for the risk are provided in Theorem~\ref{thm:hatthetaLExactRisk} and Theorem~\ref{thm:hatthetaSExactRisk}. The vanishing-noise expansions in \eqref{eq:vanishing-noise-diff} follow by Corollary~\ref{cor:VanishingNoiseThetaSLRisks}. Since $c>0$ implies $\arctan (c^{-1}) \in(0,\frac\pi2)$, the leading constant is strictly negative, and therefore
$R_\sigma(\theta^\star; \Theta_S)-R_\sigma(\theta^\star; \Theta_L)<0$ for all sufficiently small $\sigma$.
The diverging-noise expansions \eqref{eq:diverging-noise-diff} follow by Corollary~\ref{cor:DivergingNoiseThetaSLRisks}. 

Write the leading term as
\[
g(c)
:=
\frac{1}{4c^2}-\frac{1+c^2}{2\pi c^2}\arctan \frac{1}{c}
=
\frac{1}{4c^2}\left(1-\frac{2}{\pi}(1+c^2)\arctan \frac{1}{c}\right).
\]
The sign of $g(c)$ is determined by the comparison of $h(c):=(1+c^2)\arctan(c^{-1})$ with $\frac{\pi}{2}$. Specifically, $g(c)>0$ if and only if
$h(c)<\frac{\pi}{2}$, while $g(c)<0$ if and only if $h(c)>\frac{\pi}{2}$. We now show that
\[
h(c)
\begin{cases}
<\frac{\pi}{2}, & 0<c<1,\\[2pt]
=\frac{\pi}{2}, & c=1,\\[2pt]
>\frac{\pi}{2}, & c>1.
\end{cases}
\]
Differentiating gives $h'(c)=-1+2c\arctan (c^{-1})$, and $ h''(c)
    = 2(\arctan (c^{-1})-\frac{c}{1+c^2})$.

The second derivative is strictly positive for all $c>0$, since
\[
    \arctan \frac{1}{c} 
    =
    \int_0^{\frac{1}{c}} \frac{1}{1+t^2}\,dt
    >
    \frac{1}{c}\cdot \frac{1}{1+\frac{1}{c}^2}
    =
    \frac{c}{1+c^2},
\]
where the strict inequality follows because $t\mapsto (1+t^2)^{-1}$ is strictly decreasing on $(0,\infty)$. 
Thus $h$ is strictly convex on $(0,\infty)$. 

We first consider $0<c<1$. Since $\lim_{c\to 0^+} h(c)=\frac{\pi}{2}$ and $h(1)=\frac{\pi}{2}$, strict convexity implies that $h(c)<\frac{\pi}{2}$ for every $0<c<1$.

Next, we consider the case $c>1$. Note that $h'(1)>0$. Since $h$ is strictly convex, $h'$ is strictly increasing. Hence
$h'(c)>0$ for all $c>1$, so $h$ is strictly increasing on
$(1,\infty)$. Therefore $h(c)>h(1)=\frac{\pi}{2}$ for every $c>1$. Consequently,
\[
g(c)
\begin{cases}
>0, & 0<c<1,\\
=0, & c=1,\\
<0, & c>1.
\end{cases}
\]
The stated ordering for sufficiently large $\sigma$ follows from
\eqref{eq:diverging-noise-diff}.
\end{proof}

\subsection{Proofs of results in Section~\ref{sec:vanishing-noise-asymptotic}}
We first present a result summarizing the vanishing-noise risk behavior of the LSE. This recovers the tangent-cone asymptotics of
\cite[Theorem~3.1]{oymak2016sharp} and, in the interior case, strengthens the remainder to be smaller than any polynomial order in $\sigma$. Our proof uses a different approach however, based on differentiability properties of metric projections; see \cite{zarantonello1971projections}, which may be of independent
interest. The proof is deferred to Appendix~\ref{sec:vanishingnoise} of the supplementary material.

\begin{proposition}\label{prop:smallSigmaGeneral}
Let $\Theta \subseteq \R^d$ be a compact and convex set with nonempty interior. Fix $\theta^\star\in\Theta$ and for $\sigma>0$, define
\[
\hattheta_\sigma := \Pi_\Theta(\theta^\star+\sigma Z),
\qquad Z\sim N(0,I_d).
\]
\begin{enumerate}
\item[\textup{(i)}] \textbf{Boundary case.}
If $\theta^\star\in\partial\Theta$, then as $\sigma\to 0^+$,
\begin{equation}\label{eq:scaled-boundary-L2-v2}
\frac{\hattheta_\sigma-\theta^\star}{\sigma}
\stackrel{L^2}{\longrightarrow} \Pi_{T_\Theta(\theta^\star)}(Z),
\end{equation}
and consequently
\begin{equation}\label{eq:risk-boundary-v2}
\E\|\hattheta_\sigma-\theta^\star\|^2
=
\sigma^2\delta (T_{\Theta}(\theta^\star))
+
o(\sigma^2).
\end{equation}

\item[\textup{(ii)}] \textbf{Interior case.}
If $\theta^\star\in\operatorname{int}(\Theta)$, then as $\sigma\to 0^+$,
\begin{equation}\label{eq:scaled-interior-L2-v2}
\frac{\hattheta_\sigma-\theta^\star}{\sigma}
\stackrel{L^2}{\longrightarrow}  Z.
\end{equation}
Moreover, for every integer $m\ge 1$,
\begin{equation}\label{eq:risk-interior-superpoly-v2}
\E\|\hattheta_\sigma-\theta^\star\|^2
=
\sigma^2 d
+
o(\sigma^m).
\end{equation}
\end{enumerate}
\end{proposition}

\begin{proof}[Proof of Theorem~\ref{thm:NoRRinSmallNoise}]
Since $\Theta_S\subseteq\Theta_L$ and $\theta^\star\in\Theta_S$, we have $T_{\Theta_S}(\theta^\star)\subseteq T_{\Theta_L}(\theta^\star)$. By \cite[Proposition~3.1]{amelunxen2014living}, the statistical dimension is
monotone with respect to inclusion of closed convex cones. Hence $\delta(T_{\Theta_S}(\theta^\star)) \le \delta(T_{\Theta_L}(\theta^\star))$. Applying the asymptotic risk expansion of
\cite[Theorem~3.1]{oymak2016sharp} gives the claim.
\end{proof}

\begin{proof}[Proof of Corollary~\ref{cor:VanishingNoiseThetaSLRisks}]

We proceed by computing the statistical dimension of the tangent cone of
$\Theta_S$ and $\Theta_L$ at $\theta^\star$. The result then follows by
applying \cite[Theorem 3.1]{oymak2016sharp}.

Since $\theta^\star=v_1=0$ and $\Theta_S=\{\alpha v_2:0\le \alpha\le1\}$, we have 
\[ 
 T_{\Theta_S}(\theta^\star) = 
 \overline{\bigcup_{t\ge0}t(\Theta_S-\theta^\star)}
 = \{\alpha v_2:\alpha\ge0\}. \] 
Similarly, since $\Theta_L = \operatorname{conv}(v_1,v_2,v_3)$, every point of $\Theta_L-\theta^\star$ can be written as 
\[ 
\alpha v_2+\beta v_3, \qquad \alpha,\beta\ge0,\qquad \alpha+\beta\le1. 
\] 
Taking all nonnegative rescalings gives 
\[ 
T_{\Theta_L}(\theta^\star)
= \overline{\bigcup_{t\ge0}t(\Theta_L-\theta^\star)}=\{\alpha v_2+\beta v_3:\alpha,\beta\ge 0\}. 
\]
Importantly, $T_{\Theta_S}(\theta^\star)$ and
$T_{\Theta_L}(\theta^\star)$ are polyhedral cones, meaning that they are closed convex
cones generated by finitely many rays. In particular,
$T_{\Theta_S}(\theta^\star)$ is generated by the single ray through $v_2$,
while $T_{\Theta_L}(\theta^\star)$ is generated by the two rays through
$v_2$ and $v_3$.

We recall the following definitions \cite[Section 5.1]{amelunxen2014living}.
Let $C \subseteq \R^2$ be a polyhedral cone. Its (conic) intrinsic volumes are the
numbers $\Gamma_k(C)\in[0,1]$, $k=0,1,2$, defined by 
\begin{align*}
\Gamma_k(C)
:=
\mathbb{P}(\Pi_C(Z)\ \text{lies in the relative interior of a
$k$-dimensional face of $C$}),
\end{align*}
for $Z\sim N(0,I_2)$.

Here, a $k$-dimensional face of $C$ means that $F$ is a $k$-dimensional convex subset of $C$ with the property that if $x,y\in C$, $\lambda\in(0,1)$,
and $\lambda x+(1-\lambda)y\in F$, then $x,y\in F$.

The quantities $\Gamma_k$ satisfy $\sum_{k=0}^2 \Gamma_k(C)=1$, and the statistical dimension admits the
representation $\delta(C)=\sum_{k=0}^2 k \Gamma_k(C)$. We apply this result to
$T_{\Theta_S}(\theta^\star)$ and $T_{\Theta_L}(\theta^\star)$.
\begin{itemize}
    \item $T_{\Theta_S}(\theta^\star)$. This cone has exactly two faces: the vertex $v_1$ (dimension $0$) and the ray itself (dimension $1$). Hence $\Gamma_2 (T_{\Theta_S}(\theta^\star))=0$. Moreover, by symmetry of the standard Gaussian, the scalar projection of $Z$ onto the direction $v_2$ is nonnegative with probability $\frac12$, and negative with probability $\frac12$. When it is negative, the closest point in $T_{\Theta_S}(\theta^\star)$ is the origin. Therefore
    \begin{align*}
    \Gamma_0(T_{\Theta_S}(\theta^\star))&=\P(\Pi_{T_{\Theta_S}(\theta^\star)}(Z)=0)=\frac12,\\ \Gamma_1(T_{\Theta_S}(\theta^\star))&=1-\Gamma_0(T_{\Theta_S}(\theta^\star))-\Gamma_2 (T_{\Theta_S}(\theta^\star))=\frac12.
    \end{align*}
    Therefore 
    \[
    \delta(T_{\Theta_S}(\theta^\star))=0\cdot \Gamma_0(T_{\Theta_S}(\theta^\star))+1\cdot \Gamma_1(T_{\Theta_S}(\theta^\star))+2\cdot \Gamma_2 (T_{\Theta_S}(\theta^\star))=\frac12.
    \]

\item $T_{\Theta_L}(\theta^\star)$. 

Let $\vartheta = \arctan (c^{-1})$ denote the angle between the rays spanned
by $v_2$ and $v_3$. Since $T_{\Theta_L}(\theta^\star)$ is a
two-dimensional cone, it has a $2$-dimensional face, namely the cone itself.
Moreover, when $Z\in \operatorname{int}(T_{\Theta_L}(\theta^\star))$, the
projection is unchanged, $\Pi_{T_{\Theta_L}(\theta^\star)}(Z)=Z$.

Thus the projection lies in the relative interior of the $2$-dimensional face
precisely when $Z\in \operatorname{int}(T_{\Theta_L}(\theta^\star))$. Since
$Z\sim N(0,I_2)$ is rotationally invariant, the direction
$Z/\|Z\|$ is uniform on $\mcS^1$. Moreover,
$\mathbb P(Z\in \partial T_{\Theta_L}(\theta^\star))=0$. Hence this
probability is the normalized arc length of
$T_{\Theta_L}(\theta^\star)\cap\mcS^1$. Since this arc has length
$\vartheta$, we obtain
\[
\Gamma_2 (T_{\Theta_L}(\theta^\star))
=
\mathbb{P}(Z\in T_{\Theta_L}(\theta^\star))
=
\frac{\vartheta}{2\pi}.
\]

Next, recall that the polar cone of a cone $C\subseteq \R^2$ is
\[
    C^\circ
    :=
    \{z\in\R^2:\langle z,x\rangle\le0
    \text{ for all }x\in C\}.
\]
For a closed convex cone $C$, the projection $\Pi_C(Z)$ equals $0$ if
and only if $Z\in C^\circ$. Thus the projection lies in the relative interior
of the $0$-dimensional face $\{0\}$ precisely when
$Z\in T_{\Theta_L}(\theta^\star)^\circ$. 
  
In $\R^2$, the polar of a cone with opening angle $\vartheta\in(0,\pi)$ consists of the directions making
obtuse angles with every vector in the original cone, and therefore occupies
the complementary angular sector of size $\pi-\vartheta$. Hence, the polar cone $T_{\Theta_L}(\theta^\star)^\circ$ has opening angle
$\pi-\vartheta$ and so
\[
\Gamma_0(T_{\Theta_L}(\theta^\star))
=
\mathbb{P}(\Pi_{T_{\Theta_L}(\theta^\star)}(Z)=0)
=
\mathbb{P}(Z\in T_{\Theta_L}(\theta^\star)^\circ)
=
\frac{\pi-\vartheta}{2\pi}.
\]
Thus
\[
\Gamma_1(T_{\Theta_L}(\theta^\star))
=
1-\Gamma_0(T_{\Theta_L}(\theta^\star))
-\Gamma_2 (T_{\Theta_L}(\theta^\star))
=
\frac12.
\]
It follows that 
\[
\delta(T_{\Theta_L}(\theta^\star))
=
\Gamma_1(T_{\Theta_L}(\theta^\star))
+
2  \Gamma_2 (T_{\Theta_L}(\theta^\star))
=
\frac12+\frac{\vartheta}{\pi}
=
\frac12+\frac1\pi \arctan \frac1c .
\]
\end{itemize}
Importantly, $\delta(T_{\Theta_S}(\theta^\star))=\frac12$ and $\delta(T_{\Theta_L}(\theta^\star))\in(\frac12,1)$, and so $\delta(T_{\Theta_S}(\theta^\star)) - \delta(T_{\Theta_L}(\theta^\star)) <0$ for any $c > 0$.
\end{proof}

\subsection{Proofs of results in Section~\ref{sec:diverging-noise-asymptotic}}

\begin{proof}[Proof of Theorem~\ref{thm:large-sigma-projection}]
Define, for $\sigma>0$ and $\theta \in \Theta$,
\[
F_\sigma(\theta)
:=
\langle \theta,U\rangle-\frac{1}{2\sigma\|Z\|}\|\theta-\theta^\star\|^2,
\qquad
F_\infty(\theta):=\langle \theta,U\rangle,
\]
on the event $\{Z\neq 0\}$, which has probability one. We first show that $\hattheta_\sigma$ is the unique maximizer of $F_\sigma$ over $\Theta$.
By definition,
\[
\hattheta_\sigma
=
\argmin_{\theta\in\Theta}\|\theta-(\theta^\star+\sigma Z)\|^2.
\]
Expanding the square and discarding terms independent of $\theta$ yields
\[
\hattheta_\sigma
=
\argmin_{\theta\in\Theta}\{\|\theta-\theta^\star\|^2-2\sigma\langle \theta,Z\rangle\}.
\]
Multiplying by $-\frac{1}{2\sigma\|Z\|}$ gives
\[
\hattheta_\sigma
=
\argmax_{\theta\in\Theta}
\{\langle \theta,U\rangle-\frac{1}{2\sigma\|Z\|}\|\theta-\theta^\star\|^2\}
=
\argmax_{\theta\in\Theta}F_\sigma(\theta).
\]
Since $F_\sigma$ is strictly concave on $\R^d$, the maximizer over the convex set $\Theta$
is unique.

Next, since $\Theta$ is bounded, $F_\sigma\to F_\infty$ uniformly on $\Theta$ almost surely as $\sigma \to \infty$.
Indeed, letting $D:=\sup_{\theta\in\Theta}\|\theta-\theta^\star\|<\infty$, we have
\[
\sup_{\theta\in\Theta}|F_\sigma(\theta)-F_\infty(\theta)|
=
\sup_{\theta\in\Theta}\frac{1}{2\sigma\|Z\|}\|\theta-\theta^\star\|^2
\le \frac{D^2}{2\sigma\|Z\|}\to 0
\quad\text{a.s.}
\]

 We now show that if $\sigma_n\to\infty$ and $\hattheta_{\sigma_n}\to\bar\theta$, then $\bar\theta\in F_\Theta(U)$.
Fix $\varepsilon>0$ and choose $n_0$ such that
$\sup_{\theta\in\Theta}|F_{\sigma_n}(\theta)-F_\infty(\theta)|\le \varepsilon$
for all $n\ge n_0$.
Then, for $n\ge n_0$,
\[
F_\infty(\hattheta_{\sigma_n})
\ge
F_{\sigma_n}(\hattheta_{\sigma_n})-\varepsilon
=
\sup_{\theta\in\Theta}F_{\sigma_n}(\theta)-\varepsilon
\ge
\sup_{\theta\in\Theta}F_\infty(\theta)-2\varepsilon.
\]
Letting $n\to\infty$ and using continuity of $F_\infty(\theta)=\langle\theta,U\rangle$ yields
\[
F_\infty(\bar\theta)\ge \sup_{\theta\in\Theta}F_\infty(\theta)-2\varepsilon.
\]
Since $\varepsilon$ is arbitrary, $\bar\theta\in F_\Theta(U)$. We now identify the limit within $F_\Theta(U)$.
Let $\theta^\dagger:=\Pi_{F_\Theta(U)}(\theta^\star)$, which is well-defined and unique since
$F_\Theta(U)$ is a nonempty closed convex subset of $\Theta$.
For each $\sigma>0$, optimality of $\hattheta_\sigma$ gives $F_\sigma(\hattheta_\sigma)\ge F_\sigma(\theta^\dagger)$. Expanding $F_\sigma$ yields
\[
\langle \hattheta_\sigma,U\rangle
-\frac{1}{2\sigma\|Z\|}\|\hattheta_\sigma-\theta^\star\|^2
\ge
\langle \theta^\dagger,U\rangle
-\frac{1}{2\sigma\|Z\|}\|\theta^\dagger-\theta^\star\|^2.
\]
Rearranging,
\[
\|\hattheta_\sigma-\theta^\star\|^2
\le
\|\theta^\dagger-\theta^\star\|^2
-
2\sigma\|Z\|(
\langle \theta^\dagger,U\rangle-\langle \hattheta_\sigma,U\rangle).
\]
Since $\theta^\dagger\in F_\Theta(U)$ maximizes $\langle\theta,U\rangle$ over $\Theta$,
the term in parentheses is non-negative, and hence $\|\hattheta_\sigma-\theta^\star\|^2
\le
\|\theta^\dagger-\theta^\star\|^2$ for all $\sigma>0$.

Passing to the limit along the sequence $\sigma_n$ above and using continuity of the norm gives $\|\bar\theta-\theta^\star\|^2
\le
\|\theta^\dagger-\theta^\star\|^2$. Since $\bar\theta\in F_\Theta(U)$ and $\theta^\dagger$ is the unique minimizer of
$\|\theta-\theta^\star\|^2$ over $F_\Theta(U)$, it follows that
$\bar\theta=\theta^\dagger$.
 It remains to upgrade subsequential convergence to convergence. Let $\sigma_n\to\infty$ be arbitrary. Suppose, for contradiction, that $\hattheta_{\sigma_n}$ does not converge to $\theta^\dagger$. Then there exist $\varepsilon>0$ and a subsequence, still denoted $\sigma_n$, such that $\|\hattheta_{\sigma_n}-\theta^\dagger\|\ge \varepsilon$ for all $n$. Since $\Theta$ is compact and $\hattheta_{\sigma_n}\in\Theta$, this subsequence has a further subsequence converging to some $\bar\theta\in\Theta$. By the argument above, every such subsequential limit must equal $\theta^\dagger$, a contradiction. Hence $\hattheta_{\sigma_n}\to\theta^\dagger$ for every sequence $\sigma_n\to\infty$, and therefore $ \hattheta_\sigma \to \Pi_{F_\Theta(U)}(\theta^\star)$ almost surely as $\sigma\to\infty$. 

Finally, since $\hattheta_\sigma\in\Theta$ for all $\sigma$ and $\Theta$ is bounded,
$\|\hattheta_\sigma-\theta^\star\|^2\le D^2$ almost surely.
By dominated convergence,
\[
\lim_{\sigma\to\infty}\E\|\hattheta_\sigma-\theta^\star\|^2
=
\E\|\Pi_{F_\Theta(U)}(\theta^\star)-\theta^\star\|^2,
\]
which completes the proof.
\end{proof}

\begin{proof}[Proof of Corollary~\ref{cor:balls}]
For $r>0$, write $\Theta_r:=B(x_0,r)$. For any $y\in\Theta_r$ and $u\in\mcS^{d-1}$, by Cauchy-Schwarz 
\[
    \langle y,u\rangle
    =
    \langle x_0,u\rangle+\langle y-x_0,u\rangle
    \le
    \langle x_0,u\rangle+\|y-x_0\|\|u\|
    \le
    \langle x_0,u\rangle+r.
\]
Equality holds if and only if $y-x_0=ru$. Hence the exposed face of
$\Theta_r$ in direction $u$ is the singleton $F_{\Theta_r}(u)=\{x_0+ru\}$. Thus, by the diverging-noise formula in Theorem~\ref{thm:large-sigma-projection}, if
$U\sim\operatorname{Unif}(\mcS^{d-1})$, then
\[
    R_\infty(\theta^\star;\Theta_r)
    =
    \E\|x_0+rU-\theta^\star\|^2  
    =
    \|x_0-\theta^\star\|^2
    +2r\langle x_0-\theta^\star,\E U\rangle
    +r^2\E\|U\|^2.
\]
Since $\E U=0$ and $\|U\|=1$, this gives $R_\infty(\theta^\star;\Theta_r)
    =
    \|x_0-\theta^\star\|^2+r^2$. Therefore
\[
    R_\infty(\theta^\star;\Theta_S)
    =
    \|x_0-\theta^\star\|^2+r_S^2
    <
    \|x_0-\theta^\star\|^2+r_L^2
    =
    R_\infty(\theta^\star;\Theta_L).
\]
\end{proof}

\begin{proof}[Proof of Corollary~\ref{cor:ellipse}]
Let $u=(\cos\varphi,\sin\varphi)\in\mcS^1$. Since $\mcE(a,b)$ is strictly convex, the maximizer of $x\mapsto \langle u,x\rangle$ over $\mcE(a,b)$ is unique, hence $F_{\mcE(a,b)}(u)$ is a singleton.

A Lagrange multiplier calculation gives
\[
    F_{\mcE(a,b)}(u)
    =
    \left\{
    \frac{
        (a^2\cos\varphi,b^2\sin\varphi)
    }{
        \sqrt{a^2\cos^2\varphi+b^2\sin^2\varphi}
    }
    \right\}.
\]
Therefore, writing $y_{a,b}(u)$ for this exposed point,
\[
    \|y_{a,b}(u)\|^2
    =
    \frac{
        a^4\cos^2\varphi+b^4\sin^2\varphi
    }{
        a^2\cos^2\varphi+b^2\sin^2\varphi
    }
    =
    a^2+b^2
    -
    \frac{a^2b^2}{
        a^2\cos^2\varphi+b^2\sin^2\varphi
    }.
\]
Using
\[
    \int_0^{2\pi}
    \frac{d\varphi}{
        a^2\cos^2\varphi+b^2\sin^2\varphi
    }
    =
    \frac{2\pi}{ab},
\]
the diverging-noise formula from Theorem~\ref{thm:large-sigma-projection} gives
\[
    R_\infty(0;\mcE(a,b))
    =
    \frac1{2\pi}
    \int_0^{2\pi}
        \|y_{a,b}(\cos\varphi,\sin\varphi)\|^2\,d\varphi =
    a^2+b^2-ab.
\]
For fixed $a$, the map $b\mapsto a^2+b^2-ab$ has derivative $2b-a$, and is therefore strictly decreasing on
$0<b<\frac{a}{2}$. Hence, if $0<b_S<b_L<\frac{a}{2}$, $R_\infty(0;\mcE(a,b_S))
    >
    R_\infty(0;\mcE(a,b_L))$.
The nesting $\mcE(a,b_S)\subsetneq \mcE(a,b_L)$ is immediate from $0<b_S<b_L$.
\end{proof}

\subsubsection{Proofs of results in Section~\ref{sec:cvxPolytopes}}
\begin{proof}[Proof of Lemma~\ref{lem:polytopeMaxAtVertex}]
Let $\Theta=\operatorname{conv}(v_1,\ldots,v_K) \subseteq \R^d$, and let
$U\sim \operatorname{Unif}(\mathcal S^{d-1})$.
Since a linear functional attains its maximum over a polytope at a vertex,
\[
\max_{\theta\in\Theta}\langle \theta,U\rangle
=
\max_{1\le i\le K}\langle v_i,U\rangle.
\]
If $d=1$, then $\mcS^{0}=\{-1,1\}$ and the maximizer is trivially unique
unless two vertices coincide, in which case the conclusion of the lemma still
holds. We therefore assume $d\ge2$ for the remainder of the proof.
A tie between two distinct vertices $v_i$ and $v_j$ occurs at a direction
$u\in\mathcal S^{d-1}$ if and only if $\langle v_i-v_j,u\rangle=0$.

Define the set of directions for which such a tie occurs by
\[
B
:=
\bigcup_{1\le i<j\le K}
\{ u\in\mathcal S^{d-1}:\ \langle v_i-v_j,u\rangle=0 \}.
\]
For any nonzero vector $a\in\R^d$, the set $\{u\in\mathcal S^{d-1}:\langle a,u\rangle=0\}$ is a smooth $(d-2)$-dimensional sub-manifold of $\mathcal S^{d-1}$ and therefore
has $\mu_{d-1}$-measure zero. Since $B$ is a finite union of such sets, it follows that $\mu_{d-1}(B)=0$. As the uniform distribution on $\mathcal S^{d-1}$ is absolutely continuous with
respect to $\mu_{d-1}$, we conclude that $\P(U\in B)=0$.

Consequently, with probability one, the maximizer of
$\langle \theta,U\rangle$ over $\Theta$ is unique and equals some vertex $v_I$.
Moreover, this vertex is characterized by the condition $U \in \mcN_\Theta(v_I)$. 
By Theorem~\ref{thm:large-sigma-projection}, the limiting estimator is $v_I$.
Therefore,
\[
    R_\infty(\theta;\Theta)
    =
    \E\|v_I-\theta\|^2
    =
    \sum_{i=1}^K p_i\|v_i-\theta\|^2,
\]
with $p_i$ as defined in the statement.
\end{proof}

\begin{proof}[Proof of Lemma~\ref{lem:ThetaARisk}]
 We first prove the result for $x\in[0,c^{-1})$, where the vertices $v_1,v_2,v_x$ are distinct. At the endpoint $x=c^{-1}$, one has
$v_x=v_2$, so $\Theta_x=\operatorname{conv}(v_1,v_2)$, and the displayed formula follows by direct substitution.
By Lemma~\ref{lem:polytopeMaxAtVertex} 
\[
R_\infty(\theta^\star; \Theta_x)
=
\sum_{j} p_j \|v_j-\theta^\star\|^2
=
\sum_{j} p_j \|v_j\|^2
=
\alpha_c p_2 + (1+x^2)p_x,
\]
where $\alpha_c := 1+\frac{1}{c^2}$ and $p_j := \P (F_{\Theta_x}(U)=\{v_j\})$ for $j\in \{1,2,x\}$. For any $U \sim \operatorname{Unif}(\mcS^1)$, 
\[
\langle v_1,U\rangle = 0,
\qquad
\langle v_2,U\rangle = \frac{U_1}{c}+U_2,
\qquad
\langle v_x,U\rangle = x U_1 + U_2.
\]
We now compute the vertex probabilities. Let $R:=\frac{U_2}{U_1}$, which is well defined almost surely.
\begin{itemize}
\item $p_2$. The event 
$\{ F_{\Theta_x}(U) = \{v_2\} \}$ is equivalent to the event $\{ \inp{v_2,U} \ge \inp{v_1, U}, \inp{v_2,U} \ge \inp{v_x, U}\}$, which is equivalent to the event $\{\frac{U_1}{c} + U_2 \ge0, U_1 \ge 0 \}$. 
By 
Lemma~\ref{lem:ratioUnifSphere} and Lemma~\ref{lem:ratioUnifSphereConditionalOnSignU1}, the random variables $R$ and $R \mid\{U_1\ge0\}$ are both standard Cauchy. It follows that 
\begin{align*}
    p_2 &= \P(U_1 \ge 0, R\ge -c^{-1}) 
    = \P(U_1\ge0)\,\P(R \ge -c^{-1} \mid U_1\ge0) \\
    &= \frac12 \P(R \ge -c^{-1}) 
    =\frac14 + \frac{1}{2\pi} \arctan \frac1c .
\end{align*}

\item $p_x$. The event 
$\{ F_{\Theta_x}(U) = \{v_x\} \}$ is equivalent to the event $\{x U_1 + U_2 \ge0, U_1 \le 0 \}$. 
On the event $\{U_1\le 0\}$, we have
\begin{align*}
    x U_1 + U_2 \ge0 \iff U_1(x+R) \ge 0 \iff R \le -x.
\end{align*}
Therefore, by Lemma~\ref{lem:ratioUnifSphereConditionalOnSignU1},

\begin{align*}
    p_x 
    &= \P(U_1 \le 0, R\le -x) 
    = \P(U_1\le0)\,\P(R \le -x \mid U_1\le0) \\
    &= \frac12 \P(R \le -x)
    =\frac14 -\frac{1}{2\pi} \arctan x.
\end{align*}
\end{itemize}
\end{proof}

\begin{proof}[Proof of Corollary~\ref{cor:DivergingNoiseThetaSLRisks}]    
    Taking $x=0$ in Lemma~\ref{lem:ThetaARisk} gives $\Theta_x=\Theta_L$, while taking $x=c^{-1}$ gives $\Theta_x=\Theta_S$. Thus Lemma~\ref{lem:ThetaARisk} gives the displayed formulas for $R_\infty(\theta^\star;\Theta_L)$ and $R_\infty(\theta^\star;\Theta_S)$. By Theorem~\ref{thm:large-sigma-projection}, $R_\sigma(\theta^\star;\Theta)
    =
    R_\infty(\theta^\star;\Theta)+o(1)$ for $\Theta \in \{ \Theta_S, \Theta_L\}$. Substituting these limiting values gives the claimed expansions.
\end{proof}

\subsubsection{Proofs of results in Section~\ref{sec:sup-risk}}

\begin{proof}[Proof of Lemma~\ref{lem:uniform-large-sigma}]
By Theorem~\ref{thm:large-sigma-projection}, for each $\theta\in\Theta$, $R_\sigma(\theta; \Theta)\to R_\infty(\theta; \Theta)$. We upgrade this pointwise convergence to uniform convergence on $\Theta$. To this end, we first show that the family $\{R_\sigma(\cdot; \Theta):\sigma>0\}$ is uniformly bounded and equicontinuous on $\Theta$.
Throughout, write $D_\Theta:=\operatorname{diam}(\Theta)<\infty$. For any $\theta\in\Theta$ and $\sigma>0$, both $\hat\theta_\sigma$ and
$\theta$ belong to $\Theta$, and therefore $\|\hat\theta_\sigma-\theta\|\le D_\Theta$, and so $R_\sigma(\theta; \Theta)\le D_\Theta^2$, from which it follows that the class is uniformly bounded. Next, fix $\sigma>0$ and define $g_\sigma(\theta):=\|\Pi_\Theta(\theta+\sigma Z)-\theta\|^2$. For $\theta,\theta'\in\Theta$, set $a:=\Pi_\Theta(\theta+\sigma Z)-\theta$ and $
b:=\Pi_\Theta(\theta'+\sigma Z)-\theta'$. Then
\[
|g_\sigma(\theta)-g_\sigma(\theta')|
=|\|a\|^2-\|b\|^2|
\le(\|a\|+\|b\|)\|a-b\|.
\]
Since $a,b$ are differences of points in $\Theta$, we have
$\|a\|,\|b\|\le D_\Theta$. Since $\Theta$ is closed and convex, $\Pi_\Theta$ is 1-Lipschitz, and so
\[
\|a-b\|
\le
\|\Pi_\Theta(\theta+\sigma Z)-\Pi_\Theta(\theta'+\sigma Z)\|
+\|\theta-\theta'\|
\le 2\|\theta-\theta'\|.
\]
Hence, almost surely, $|g_\sigma(\theta)-g_\sigma(\theta')|
\le 4D_\Theta\|\theta-\theta'\|$. Taking expectations yields
\[
|R_\sigma(\theta; \Theta)-R_\sigma(\theta';\Theta)|
\le 4D_\Theta\|\theta-\theta'\|,
\qquad
\text{for all }\sigma>0.
\]
Recall that a set of functions with common Lipschitz constant is uniformly equicontinuous. An identical argument applies with $\Pi_\Theta$ replaced by $\Pi_{F_\Theta(U)}$. Indeed, for each realization of $U$, $F_\Theta(U)$ is a nonempty closed convex subset of $\Theta$,
so $\Pi_{F_\Theta(U)}$ is well defined and $d(\theta,F_\Theta(U))\le D_\Theta$ for all $\theta\in\Theta$. Taking expectations shows that $R_\infty(\cdot;\Theta)$ is also $4D_\Theta$-Lipschitz on $\Theta$, and hence continuous.

Let $(\sigma_n)$ be any sequence with $\sigma_n\to\infty$. We claim that
\[
\sup_{\theta\in\Theta}
|R_{\sigma_n}(\theta;\Theta)-R_\infty(\theta;\Theta)|
\to 0.
\]
Suppose this were not true. Then there exist $\varepsilon>0$ and a subsequence,
still denoted $(\sigma_n)$, such that
\[
\sup_{\theta\in\Theta}
|R_{\sigma_n}(\theta;\Theta)-R_\infty(\theta;\Theta)|
\ge \varepsilon
\qquad\text{for all }n.
\]
The family $\{R_{\sigma_n}(\cdot;\Theta)\}$ is uniformly bounded and
equicontinuous on the compact set $\Theta$. Hence, by the
Arzel\`a-Ascoli Theorem \cite[Theorem~2.4.7]{Dudley_2002}, there is a
further subsequence $(\sigma_{n_k})$ and a continuous function
$f:\Theta\to\mathbb R$ such that
\[
\sup_{\theta\in\Theta}
|R_{\sigma_{n_k}}(\theta;\Theta)-f(\theta)|
\to 0.
\]
In particular, $R_{\sigma_{n_k}}(\theta;\Theta)\to f(\theta)$ for every
$\theta\in\Theta$. On the other hand, Theorem~\ref{thm:large-sigma-projection}
gives
\[
R_{\sigma_{n_k}}(\theta;\Theta)\to R_\infty(\theta;\Theta)
\qquad\text{for every }\theta\in\Theta.
\]
Thus $f=R_\infty(\cdot;\Theta)$ pointwise on $\Theta$, and therefore
\[
\sup_{\theta\in\Theta}
|R_{\sigma_{n_k}}(\theta;\Theta)-R_\infty(\theta;\Theta)|
\to 0,
\]
contradicting the choice of the subsequence. This proves the claim.

Since this holds for every sequence $\sigma_n\to\infty$, we conclude that
\[
\sup_{\theta\in\Theta}
|R_{\sigma}(\theta;\Theta)-R_\infty(\theta; \Theta)|
\to 0
\qquad\text{as }\sigma\to\infty.
\]
The convergence of the suprema follows from the inequality
\[
|
\sup_{\theta\in\Theta}R_\sigma(\theta; \Theta)
-\sup_{\theta\in\Theta}R_\infty(\theta; \Theta)
|
\le
\sup_{\theta\in\Theta}
|R_\sigma(\theta; \Theta)-R_\infty(\theta; \Theta)|.
\]
\end{proof}

\begin{proof}[Proof of Theorem~\ref{thm:finite-sigma-reversal}]
Let 
\[
\Delta
:=
\sup_{\theta\in\Theta_S} R_\infty(\theta; \Theta_S)
-
\sup_{\theta\in\Theta_L} R_\infty(\theta; \Theta_L)
>0.
\]
By Lemma~\ref{lem:uniform-large-sigma}, there exists $\sigma_1$ such that for all
$\sigma \ge \sigma_1$,
\[
\sup_{\theta\in\Theta_S}
|R_\sigma(\theta;\Theta_S)-R_\infty(\theta; \Theta_S)|
\le \frac{\Delta}{3},
\qquad
\sup_{\theta\in\Theta_L}
|R_\sigma(\theta;\Theta_L)-R_\infty(\theta; \Theta_L)|
\le \frac{\Delta}{3}.
\]
Therefore, for all 
$\sigma \ge \sigma_1$,
\begin{align*}
\sup_{\theta\in\Theta_S}R_\sigma(\theta;\Theta_S)
\ge
\sup_{\theta\in\Theta_S}R_\infty(\theta; \Theta_S)-\frac{\Delta}{3} 
>
\sup_{\theta\in\Theta_L}R_\infty(\theta; \Theta_L)+\frac{\Delta}{3} 
\ge
\sup_{\theta\in\Theta_L}R_\sigma(\theta;\Theta_L),
\end{align*}
which proves the claim.
\end{proof}

\begin{proof}[Proof of Lemma~\ref{lem:explicit-worst-case-risk}]
For $x\in(0,c^{-1})$, write $\Theta_x:=\operatorname{conv} (v_1,v_2,v_x)$ with $v_x=(x,1)$.
As shown in the proof of Lemma~\ref{lem:ThetaARisk}, the limiting risk takes the form
\[
R_\infty(\theta;\Theta_x)
=
(1-p_2-p_x)
\|\theta\|^2
+
p_2\|v_2-\theta\|^2
+
p_x\|v_x-\theta\|^2,
\]
where
\[
p_2=\frac14+\frac{1}{2\pi}\arctan \frac1c ,
\qquad
p_x=\frac14-\frac{1}{2\pi}\arctan x.
\]

For each $x$, the map $\theta\mapsto R_\infty(\theta;\Theta_x)$ is a convex quadratic function.
Since $\Theta_x$ is compact and convex, its maximum is attained at an extreme point, and hence
\[
\sup_{\theta\in\Theta_x} R_\infty(\theta;\Theta_x)
=
\max\{
R_\infty(v_1;\Theta_x),
R_\infty(v_2;\Theta_x),
R_\infty(v_x;\Theta_x)
\}.
\]
Writing $\alpha_c:=1+c^{-2}$, a direct calculation yields
\begin{align*}
R_\infty(v_1;\Theta_x)
&=\alpha_c p_2+p_x(1+x^2),\\
R_\infty(v_2;\Theta_x)
&=\alpha_c (1-p_2-p_x) +p_x(c^{-1}-x)^2,\\
R_\infty(v_x;\Theta_x)
&=(1-p_2-p_x)(1+x^2)+p_2(c^{-1}-x)^2.
\end{align*}

Evaluating these expressions gives
\[
\sup_{\theta\in\Theta_a} R_\infty(\theta;\Theta_a)
=
\max\{1.324659,\,1.306278,\,0.808860\}
=
1.324659,
\]
and
\[
\sup_{\theta\in\Theta_b} R_\infty(\theta;\Theta_b)
=
\max\{1.385120,\,1.383614,\,1.340221\}
=
1.385120.
\]
Since $a<b$ implies $\Theta_b \subsetneq \Theta_a$, the desired strict inequality follows.
\end{proof}

\subsubsection{Proofs of results in Section~\ref{ssec:beyondGaussianSqLoss}}
\begin{proof}[Proof of Theorem~\ref{thm:radial-beta-eta}]
Fix $\Theta\in\{\Theta_S,\Theta_L\}$. For fixed $y$, the likelihood is
proportional to $\theta
    \mapsto
    q\left(\frac{\|y-\theta\|}{\sigma}\right)$. Since $q$ is strictly decreasing on $[0,\infty)$, maximizing this likelihood
over $\Theta$ is equivalent to minimizing $\|y-\theta\|$ over $\Theta$.
Hence, since $\Theta$ is compact and convex, the constrained MLE is unique and
equals the Euclidean projection $\Pi_\Theta(y)$. Thus the MLE and the
constrained LSE coincide.

We next identify the diverging-noise limit. Since the distribution of $W_q$ depends only on $\|w\|$, the direction $U:=W_q/\|W_q\|$ is uniform on $\mcS^1$ on the event
$\{W_q\neq0\}$, which has probability one. Therefore, by the same diverging-noise
projection argument as in the proof of Theorem~\ref{thm:large-sigma-projection},
\[
    \Pi_\Theta(\theta^\star+\sigma W_q)
    \longrightarrow
    \Pi_{F_\Theta(U)}(\theta^\star)
    \qquad\text{a.s. as } \sigma\to\infty .
\]
Consequently,
\[
    \frac1\eta
    \|
        \Pi_\Theta(\theta^\star+\sigma W_q)-\theta^\star
    \|^\eta
    \longrightarrow
    \frac1\eta
    \|
        \Pi_{F_\Theta(U)}(\theta^\star)-\theta^\star
    \|^\eta
    \qquad\text{a.s. as } \sigma\to\infty .
\]
Moreover, since $\Theta$ is compact and
$\Pi_\Theta(\theta^\star+\sigma W_q)\in\Theta$,
\[
    \frac1\eta
    \|
        \Pi_\Theta(\theta^\star+\sigma W_q)-\theta^\star
    \|^\eta
    \le
    \frac1\eta
    \sup_{\theta\in\Theta}\|\theta-\theta^\star\|^\eta
    <\infty .
\]
Dominated convergence therefore gives
\[
    R_{\sigma}^{(q,\eta)}(\theta^\star;\Theta)
    \longrightarrow
    R_{\infty}^{(\eta)}(\theta^\star;\Theta)
    :=
    \E\frac1\eta
    \|
        \Pi_{F_\Theta(U)}(\theta^\star)-\theta^\star
    \|^\eta
    \qquad\text{as } \sigma\to\infty .
\]
The limiting risk does not depend on $q$, since every admissible radial
density induces the same uniform distribution for the noise direction $U$.

We now compute the limiting gap in the running example, recalling that $\alpha_c := 1 + \frac{1}{c^2}$. The exposed-vertex probabilities are the same as in
Corollary~\ref{cor:DivergingNoiseThetaSLRisks}. For $\Theta_S$, the vertex
$v_2$ is selected with probability $\frac12$, and hence $R_{\infty}^{(\eta)}(\theta^\star;\Theta_S)
    =
    \frac1\eta\cdot \frac12\,\alpha_c^{\eta/2}$. For $\Theta_L$, the vertices $v_2$ and $v_3$ are selected with
probabilities $\frac14+\frac{1}{2\pi}\arctan\frac1c$ and $\frac14$,
respectively. Since
$\|v_2-v_1\|^\eta=\alpha_c^{\eta/2}$ and $\|v_3-v_1\|^\eta=1$, we obtain
\[
    R_{\infty}^{(\eta)}(\theta^\star;\Theta_L)
    =
    \frac1\eta
    \left[
    \left(
        \frac14+\frac{1}{2\pi}\arctan\frac1c
    \right)
    \alpha_c^{\eta/2}
    +
    \frac14
    \right].
\]
Therefore
\[
    R_{\infty}^{(\eta)}(\theta^\star;\Theta_S)
    -
    R_{\infty}^{(\eta)}(\theta^\star;\Theta_L)
    =
    \frac1\eta
    \left[
    \left(
        \frac14-\frac{1}{2\pi}\arctan\frac1c
    \right)
    \alpha_c^{\eta/2}
    -
    \frac14
    \right]
    =
    \frac1\eta\,\Delta_\eta(c).
\]
If $\Delta_\eta(c)>0$, then the limiting gap is positive. Since the
finite-noise risks converge to their corresponding limiting risks as
$\sigma\to\infty$, the same strict inequality holds for all sufficiently
large $\sigma$.

It remains to prove the stated sufficient conditions. As $c\to0^+$,
\[
    \frac14-\frac{1}{2\pi}\arctan\frac1c
    \sim
    \frac{c}{2\pi},
    \qquad
    \alpha_c^{\eta/2}
    \sim
    c^{-\eta}.
\]
Hence $\Delta_\eta(c) \sim \frac{1}{2\pi}c^{1-\eta}-\frac14$. If $\eta>1$, this tends to $+\infty$ as $c\to0^+$. Therefore
$\Delta_\eta(c)>0$ for all sufficiently small $c>0$.

Finally, suppose $\eta\ge2$. Since $\alpha_c\ge1$, we have
$\alpha_c^{\eta/2}\ge\alpha_c$. Also $\frac14-\frac{1}{2\pi}\arctan\frac1c > 0$ for every $c>0$. Therefore
\[
    \Delta_\eta(c)
    \ge
    \left(
        \frac14-\frac{1}{2\pi}\arctan\frac1c
    \right)
    \alpha_c
    -
    \frac14
    =
    \Delta_2(c).
\]
By Corollary~\ref{cor:DivergingNoiseThetaSLRisks}, $\Delta_2(c)>0$ for every
$0<c<1$. Thus $\Delta_\eta(c)>0$ for every $0<c<1$ whenever
$\eta\ge2$.
\end{proof}
\begin{proof}[Proof of Theorem~\ref{thm:radial-worst-case-all-eta}]
Let $ v_-:=(-1,0)$, $v_+:=(1,0)$ and $v_0:=(0,h)$ where $h>0$ will be chosen below. Define $\Theta_S=\operatorname{conv}(v_-,v_+)$ and $\Theta_L=\operatorname{conv}(v_-,v_+,v_0)$. Both sets are nonempty, compact, and convex, and for $h>0$ we have
$\Theta_S\subsetneq\Theta_L$.

We first identify the diverging-noise limit. Fix
$\Theta\in\{\Theta_S,\Theta_L\}$. Arguing exactly as in the proof of
Theorem~\ref{thm:radial-beta-eta}, the almost-sure convergence
\eqref{eq:large-sigma-a.s}, together with dominated convergence, gives, for
each fixed $\theta\in\Theta$,
\[
    R_\sigma^{(q,\eta)}(\theta;\Theta)
    \longrightarrow
    R_{\infty}^{(\eta)}(\theta;\Theta)
    :=
    \E \frac1\eta
    \|
        \Pi_{F_\Theta(U)}(\theta)-\theta
    \|^\eta
    \qquad\text{as } \sigma\to\infty,
\]
where $U\sim\operatorname{Unif}(\mcS^1)$.

We also use uniform convergence over $\theta\in\Theta$. This is
obtained by repeating the proof of Lemma~\ref{lem:uniform-large-sigma}. The only change is that the squared loss $\|x\|^2$ is replaced by
$x\mapsto \eta^{-1}\|x\|^\eta$. Since this map is bounded and Lipschitz on
the compact set $\Theta-\Theta$, the same equicontinuity and compactness
argument used in that proof gives
\[
    \sup_{\theta\in\Theta}
    |
        R_\sigma^{(q,\eta)}(\theta;\Theta)
        -
        R_{\infty}^{(\eta)}(\theta;\Theta)
    |
    \longrightarrow 0
    \qquad\text{as } \sigma\to\infty .
\]

We now compute the limiting worst-case risks. For $\Theta_S$, the two
endpoints are selected with probability $\frac{1}{2}$ each. Thus, for
$\theta=(t,0)\in\Theta_S$, with $t\in[-1,1]$,
\[
    R_{\infty}^{(\eta)}(\theta;\Theta_S)
    =
    \frac1{2\eta}
    \left(
        |t+1|^\eta + |1-t|^\eta
    \right).
\]
This is a convex function of $t$, so its maximum over $[-1,1]$ is attained
at an endpoint. Therefore
\[
    \overline R_{\infty}^{(\eta)}(\Theta_S)
    :=
    \sup_{\theta\in\Theta_S}
    R_{\infty}^{(\eta)}(\theta;\Theta_S)
    =
    \frac1\eta\,2^{\eta-1}.
\]

For $\Theta_L$, let $p_{0,h}$ denote the normal-cone probability of the
vertex $v_0$. By symmetry, the normal-cone probabilities of $v_-$ and
$v_+$ are equal; writing this common value as $p_h$, we have $2p_h+p_{0,h}=1$.

Since $h>0$, $v_0$ is a genuine vertex of $\Theta_L$. Hence $v_0$ is
exposed by a nontrivial set of directions, in the sense that there is a direction
$u_0\in\mcS^1$ and an $\varepsilon>0$ such that, whenever
$\|u-u_0\|<\varepsilon$, the linear functional
$x\mapsto \langle u,x\rangle$ is uniquely maximized over $\Theta_L$ at
$v_0$. Since $U$ is uniform on $\mcS^1$, this event has positive
probability. Thus $p_{0,h}>0$.

The limiting risk over $\Theta_L$ is a weighted sum of functions of the form
$\theta\mapsto \|v-\theta\|^\eta$, where
$v\in\{v_-,v_+,v_0\}$. Since $\eta>1$, each such function is convex in
$\theta$. Hence $R_{\infty}^{(\eta)}(\theta;\Theta_L)$ is convex in
$\theta$. A convex function on a compact polytope attains its maximum at an
extreme point, so it suffices to evaluate the limiting risk at
$v_-$, $v_+$, and $v_0$.

Set
\[
    r_h:=\|v_0-v_-\|=\|v_0-v_+\|=\sqrt{1+h^2}.
\]
At $v_-$ and $v_+$, the limiting risk is
\[
    \frac1\eta\left(p_h 2^\eta+p_{0,h}r_h^\eta\right)
    =
    \frac1\eta
    \left((1-p_{0,h})2^{\eta-1}+p_{0,h}r_h^\eta\right).
\]
At $v_0$, the limiting risk is $\frac1\eta\left(2p_h r_h^\eta\right)
    =
    \frac1\eta (1-p_{0,h})r_h^\eta$.
Since $\eta>1$, we can choose $h>0$ sufficiently small that $r_h^\eta=(1+h^2)^{\eta/2}<2^{\eta-1}$.  For this choice of $h$, the first displayed quantity is a strict convex
combination of $2^{\eta-1}$ and $r_h^\eta$, with positive weight on
$r_h^\eta$, because $p_{0,h}>0$. Hence $(1-p_{0,h})2^{\eta-1}+p_{0,h}r_h^\eta < 2^{\eta-1}$. Moreover, $ (1-p_{0,h})r_h^\eta
    <
    r_h^\eta
    <
    2^{\eta-1}$. Therefore, $\overline R_{\infty}^{(\eta)}(\Theta_L)
    <
    \frac1\eta\,2^{\eta-1}
    =
    \overline R_{\infty}^{(\eta)}(\Theta_S)$.

Thus the limiting worst-case gap is strictly positive, $\overline R_{\infty}^{(\eta)}(\Theta_S)
    -
    \overline R_{\infty}^{(\eta)}(\Theta_L)
    >0$. By the uniform convergence of the finite-noise risks to their diverging-noise
limits, the same strict inequality persists for all sufficiently large
$\sigma$. That is, there exists $\sigma_1<\infty$ such that, for all
$\sigma\ge\sigma_1$,
\[
    \sup_{\theta\in\Theta_S}
    R_{\sigma}^{(q,\eta)}(\theta;\Theta_S)
    >
    \sup_{\theta\in\Theta_L}
    R_{\sigma}^{(q,\eta)}(\theta;\Theta_L).
\]
This proves the claim.
\end{proof}

\section{Conclusion}
This work identifies a previously unrecognized failure mode of the constrained least squares estimator. We show that the risk of the LSE need not improve --- and can in fact worsen --- when the feasibility constraint is tightened, even when the constraint is compact, convex, and correctly specified. 

Importantly, this phenomenon arises at multiple levels. We establish risk reversal both pointwise, at fixed parameter values, and in a worst-case sense, where the worst-case risk over the smaller feasible set exceeds that over the larger set. The latter shows that risk reversal is not a localized pathology, but can reflect a genuine degradation in global performance, underscoring its relevance for statistical practice.

The underlying mechanism is geometric. In the vanishing-noise regime, the leading-order risk is governed by the tangent cone and is monotone under set inclusion. In the diverging-noise regime, by contrast, the estimator is governed by exposed faces selected by the noise direction. Tightening the constraint can then redirect probability toward faces that are farther from the true parameter, producing a larger risk. This perspective yields explicit constructions for polytopes as well as smooth, full-dimensional examples in which the true parameter lies in the interior of both constraint sets.

The phenomenon is also not specific to Gaussian noise or squared-error loss. For a broad class of radial location models and radial loss functions, the constrained MLE remains the Euclidean projection, and the same exposed-face mechanism gives both pointwise and worst-case risk reversal.

These results show that the statistical performance of projection-based estimators depends not only on the size of the feasible set, but also on how its global geometry interacts with the noise. Natural directions for future work include designing statistically principled procedures that avoid risk reversal and determining whether the phenomenon persists when the constrained likelihood estimator is not a Euclidean projection.

\section*{Acknowledgments}
I am grateful to P. Rigollet for many helpful discussions, and I thank S. Kotekal and M. Neykov for insightful feedback on an early draft. This work was supported by funding from the Eric and Wendy Schmidt Center at the Broad Institute of MIT and Harvard.

\appendix
\section*{Supplementary Material}
\addcontentsline{toc}{section}{Supplementary Material}

\section{Finite-noise risk analysis}\label{sec:auxresults}

In this section, we derive explicit, finite-$\sigma$ risk expressions for the LSE over $\Theta_S, \Theta_L$ in the running example.

\subsection{Background on Owen's \texorpdfstring{$T$}{T}-function}
The risk formulas in Theorem~\ref{thm:hatthetaLExactRisk} are expressed in terms of \emph{Owen's $T$-function} \cite{owen1980table}. We recall here its definition and basic properties. For $h,a\in\R$, let
\[
T(h,a)
:=
\phi(h)\int_0^a \frac{\phi(hz)}{1+z^2}dz,
\]
where $\phi$ denotes the probability density function of a unit Gaussian. We write $\Phi$ to denote the cumulative distribution function of a unit Gaussian. 

The function $T$ satisfies the following basic identities:
\[
\begin{gathered}
    T(\infty,a)=0,\qquad
    T(h,0)=0,\qquad
    T(0,a)=\frac{1}{2\pi}\arctan a, \\
    T(h,\infty)=\frac12\Phi(-h)\ \text{for }h\ge0, \qquad
    T(h,-a)=-T(h,a),
\end{gathered}
\]
where expressions involving infinite or zero arguments are understood in the sense of limits.
We also make use of the following Gaussian integral identities.

\begin{proposition}
\label{prop:OwensTProp}
Let $a,b\in\R$ and write $\sfr_{a,b}:=\frac{a}{\sqrt{1+b^2}}$. The following identities hold.
\begin{enumerate}[label=(\roman*)]
\item If $a\neq0$ and $m>0$, then
\begin{align*}
    \int_0^m \phi(z)\Phi(a+bz)\,dz
    &=
    T \inparen{m, \frac{\sfr_{a,b}}{m}}
    + T \inparen{\sfr_{a,b}, \frac{m}{\sfr_{a,b}}}
    - T \inparen{m, \frac{a}{m}+b} \\
    &\quad
    - T \inparen{\sfr_{a,b}, b+\frac{m}{a}(1+b^2)}
    + T \inparen{\sfr_{a,b},b} \\
    &\quad
    + \Phi(m)\Phi \inparen{\sfr_{a,b}}
    -\frac{1}{2}\Phi \inparen{\sfr_{a,b}} .
\end{align*}
\item For any $b\in\R$ and $m\ge 0$,
\begin{align*}
\int_0^m \phi(z)\Phi(bz)dz
=
\frac12\Phi(m)-\frac14-T(m,b)+\frac{1}{2\pi}\arctan b.
\end{align*}
In particular, this identity is obtained by taking the limit $a\to 0$ in \textup{(i)}.
\item
For $a,b\in\R$ and $m\ge 0$,
\begin{align*}
\int_0^m z\phi(z)\phi(a+bz)dz
&=
\frac{1}{1+b^2}\phi \inparen{\sfr_{a,b}}
\insquare{
\phi \inparen{b\sfr_{a,b}}
- \phi \inparen{ m \sqrt{1+b^2} + b\sfr_{a,b}}
} \\
&\quad
+
\frac{ab}{(1+b^2)^{3/2}}\phi \inparen{\sfr_{a,b}}
\insquare{
\Phi \inparen{b\sfr_{a,b}}
- \Phi \inparen{ m \sqrt{1+b^2} + b\sfr_{a,b}}
}.
\end{align*}
\end{enumerate}
\end{proposition}

\begin{proof} 
The first identity is obtained from \cite[Table I, entry $10,010.3$]{owen1980table}. The second follows from
Owen~\cite[Table I, entry $1,010.1$]{owen1980table}. The third follows from Owen~\cite[Table I, entry $111$]{owen1980table}.
\end{proof}

\subsubsection{Risk analysis of \texorpdfstring{$\hattheta_L$}{theta-hat L}}
The LSE $\hattheta_L$ exhibits piecewise behavior depending on the location of $y \in \R^2$.
Define $s_c(y) := \frac{y_1 + c y_2}{1+c^2}$. We partition $\R^2$ into the following regions (see Figure~\ref{fig:side_by_side}), on each of which $\hattheta_L$ admits a simple closed-form expression.
\begin{table}[htbp]
\centering
\caption{Projection regions for  $\hattheta_L$.}
\label{tab:regions}

\begingroup
\renewcommand{\arraystretch}{1.35}
\setlength{\tabcolsep}{6pt}

\begin{tabular}{l|p{5.6cm}|l|l}
\hline
Region 
& Definition $(y_1,y_2)$ 
& $\hat\theta_L(y)$ 
& Face of $\Theta_L$ \\
\hline\hline

$\Theta_L$
& $0 \le y_1 \le \frac{1}{c},\; 0 \le y_2 \le 1,\; y_2 \ge c y_1$
& $(y_1,y_2)$
& Interior \\
\hline

$A_1$
& $y_1 \le -c y_2,\; y_2 \le 0$
& $(0,0)$
& Vertex $v_1$ \\
\hline

$A_2$
& $y_1 \ge \frac{1}{c},\; y_2 \ge 1 - \frac{1}{c}\bigl(y_1 - \frac{1}{c}\bigr)$
& $(\frac{1}{c},1)$
& Vertex $v_2$ \\
\hline

$A_3$
& $y_1 \le 0,\; y_2 \ge 1$
& $(0,1)$
& Vertex $v_3$ \\
\hline

$A_{12}$
& $y_2 \le c y_1,\; 0 \le s_c(y) \le \frac{1}{c}$
& $(s_c(y), c s_c(y))$
& Edge $[v_1,v_2]$ \\
\hline

$A_{13}$
& $y_1 \le 0,\; 0 \le y_2 \le 1$
& $(0,y_2)$
& Edge $[v_1,v_3]$ \\
\hline

$A_{23}$
& $0 \le y_1 \le \frac{1}{c},\; y_2 \ge 1$
& $(y_1,1)$
& Edge $[v_2,v_3]$ \\
\hline
\end{tabular}

\endgroup
\end{table}

\begin{theorem}\label{thm:hatthetaLExactRisk}
    In the running example,
    
\begin{align*}
&R_\sigma(\theta^\star; \Theta_L)
=
\sigma^2\Bigg[
\frac{1}{\sigma}\phi\left(\frac{1}{\sigma}\right)
\left(\frac12-\Phi\left(\frac{1}{c\sigma}\right)\right)
-2T\left(\frac{1}{\sigma},\frac{1}{c}\right)
+\frac{1}{\pi}\arctan \left ( \frac{1}{c} \right)
\Bigg] \\
&
+\alpha_c\Bigg[
\Phi\left(-\frac{\sqrt{\alpha_c}}{\sigma}\right)
\Phi\left(-\frac{1}{c\sigma}\right)
+T\left(\frac{1}{c\sigma},c\sqrt{\alpha_c}\right)
+T\left(\frac{\sqrt{\alpha_c}}{\sigma},\frac{1}{c\sqrt{\alpha_c}}\right)
-T\left(\frac{1}{c\sigma},c\right)
\Bigg] \\
&
+\frac12\Phi\left(-\frac{1}{\sigma}\right) 
+\frac{\sigma^2}{2}\Bigg(
\Phi\left(\frac{\sqrt{\alpha_c}}{\sigma}\right)
-\frac12
-\frac{\sqrt{\alpha_c}}{\sigma}
\phi\left(\frac{\sqrt{\alpha_c}}{\sigma}\right)
\Bigg) \\
&
+\frac{\sigma^2}{2}\Bigg(
\Phi\left(\frac{1}{\sigma}\right)
-\frac{1}{\sigma}\phi\left(\frac{1}{\sigma}\right)
-\frac12
\Bigg) 
+\Phi\left(-\frac{1}{\sigma}\right)
\Bigg[
(1+\sigma^2)
\left(
\Phi\left(\frac{1}{c\sigma}\right)-\frac12
\right)
-\frac{\sigma}{c}\phi\left(\frac{1}{c\sigma}\right)
\Bigg],
\end{align*}
    where $\alpha_c = 1+\frac{1}{c^2}$.
\end{theorem}

\begin{proof}
    We compute the contribution to the risk from each of the seven regions defined
in Table~\ref{tab:regions}, namely
\[
    \Theta_L,\ A_1,\ A_2,\ A_3,\ A_{12},\ A_{13},\ A_{23}.
\]
Throughout, we write $r_\sigma(\theta^\star; A)$ to denote the risk restricted to a region $A$, namely
\[
r_\sigma(\theta^\star; A)
:= \E\big[ \|\hattheta_L - \theta^\star\|^2 \indicator\{Y \in A\} \big].
\]
 The final expression is simply the sum of these partial risks, 
\begin{align*} R_\sigma(\theta^\star; \Theta_L) &= r_\sigma(\theta^\star; \Theta_L) + 
r_\sigma(\theta^\star; A_1) +
r_\sigma(\theta^\star; A_2) + r_\sigma(\theta^\star; A_3) \\ 
& + r_\sigma(\theta^\star; A_{12}) + r_\sigma(\theta^\star; A_{13}) + r_\sigma(\theta^\star; A_{23}). 
\end{align*} 

    \noindent
    \textit{Risk on $\Theta_L$:} $\hattheta_L = (y_1,y_2)$, and so $\|\hattheta_L - \theta^\star\|^2 = y_1^2 + y_2^2$. We have 
\begin{align*}
    r_\sigma(\theta^\star; \Theta_L)
    &=
    \int_{y_2=0}^1 \int_{y_1=0}^{\frac{y_2}{c}} 
    (y_1^2 + y_2^2)
    \frac{\phi\inparen{\frac{y_1}{\sigma}}}{\sigma} \frac{\phi\inparen{\frac{y_2}{\sigma}}}{\sigma} dy_1 dy_2 
    =: T_1+T_2.
\end{align*}
For $T_1$, we have
\begin{align*}
    T_1 &= \sigma^2 \int_{y_2=0}^1 
    \frac{\phi\inparen{\frac{y_2}{\sigma}}}{\sigma}
    \inparen{
    \int_0^{\frac{y_2}{c\sigma}} z^2 \phi(z) dz 
    } dy_2\\
    &=\sigma^2 \int_0^{\frac{1}{\sigma}} \phi(z) \inparen{\Phi\inparen{\frac{z}{c}} - \frac{z}{c} \phi\inparen{\frac{z}{c}} - \frac{1}{2}}dz\\
    &=\sigma^2 \inparen{
    \frac{1}{2\pi} \arctan \frac{1}{c}
    - T \inparen{\frac{1}{\sigma}, \frac{1}{c}} 
    - \frac{\phi(0)}{c \alpha_c} \inparen{\phi(0) - \phi \inparen{\frac{\sqrt{\alpha_c}}{\sigma}}}
    },
\end{align*}
    where the last line follows from Proposition~\ref{prop:OwensTProp} (ii), with $m=\frac1\sigma$ and $b=\frac1c$, and from 
    Proposition~\ref{prop:OwensTProp} (iii) with $m=\frac1\sigma$, $a=0$, and $b=\frac1c$. In the latter application, $\sfr_{a,b}=0$.
    
    For $T_2$,
    \begin{align*}
    T_2
    &= \sigma^2 \int_0^{\frac{1}{\sigma}} z^2 \phi(z)
    \left(\Phi\left(\frac{z}{c}\right)-\frac{1}{2}\right)\,dz \\
    &= \sigma^2 \Bigg(
        \frac{1}{\sigma}\phi\left(\frac{1}{\sigma}\right)
        \left(\frac{1}{2}-\Phi\left(\frac{1}{c\sigma}\right)\right)
        - T\left(\frac{1}{\sigma},\frac{1}{c}\right) \\
    &\qquad\qquad
        + \frac{1}{2\pi} \arctan \left ( \frac{1}{c} \right)
        + \frac{\phi(0)}{c\alpha_c}
        \left[
            \phi(0)-\phi\left(\frac{\sqrt{\alpha_c}}{\sigma}\right)
        \right]
    \Bigg).
\end{align*}
Here we again used Proposition~\ref{prop:OwensTProp} (ii) and 
Proposition~\ref{prop:OwensTProp} (iii) with $m=\frac{1}{\sigma}$, $a=0$, and $b=\frac{1}{c}$, as above.
   
Putting it together, we have 
\begin{align*}
    r_\sigma(\theta^\star; \Theta_L)
    &=\sigma^2 \inparen{
    \frac{1}{\sigma} \phi \inparen{\frac{1}{\sigma}} \inparen{\frac{1}{2} - \Phi \inparen{\frac{1}{c\sigma}}}
    - 2T \inparen{\frac{1}{\sigma}, \frac{1}{c}} 
    + \frac{1}{\pi} \arctan \frac{1}{c}
    }.
\end{align*}
    \noindent
    \textit{Risk on $A_1$:} $\hattheta_L = (0,0)$, so $r_\sigma(\theta^\star; A_1) =0$.

    \noindent

\textit{Risk on $A_2$:} $\hattheta_L = (\frac1c,1)$. We have 
\begin{align*}
    r_\sigma(\theta^\star; A_2)
    &=
    \alpha_c 
    \inparen{ 
    \int_{y_1=\frac1c}^{\infty} \int_{y_2=1}^{\infty} 
    +
    \int_{y_1=\frac1c}^{\infty} \int_{y_2=-\frac{y_1}{c} + \alpha_c}^{1} 
    }
    \frac{\phi\inparen{\frac{y_1}{\sigma}}}{\sigma}
    \frac{\phi\inparen{\frac{y_2}{\sigma}}}{\sigma}
    dy_2 dy_1
    =: T_1 +T_2.
\end{align*}

It is straightforward to show
\begin{align*}
    T_1 = \alpha_c \Phi \inparen{-\frac{1}{c\sigma}}\Phi \inparen{-\frac{1}{\sigma}}.
\end{align*}
For $T_2$, direct calculation gives
\begin{align*}
    T_2 
    &= \alpha_c \int_{\frac{1}{c\sigma}}^\infty  \phi(z) \insquare{\Phi\inparen{\frac{1}{\sigma}}
    -\Phi\inparen{\frac{\alpha_c}{\sigma} - \frac{z}{c}}
    } dz
    =\alpha_c \inparen{
    \Phi \inparen{-\frac{1}{c\sigma}}\Phi \inparen{\frac{1}{\sigma}}
    - T_{21}
    }.
\end{align*}

By 
Proposition~\ref{prop:OwensTProp} (i), applied with \[ m=\frac{1}{c\sigma}, \qquad a=\frac{\alpha_c}{\sigma}, \qquad b=-\frac1c, \qquad \sfr_{a,b}=\frac{\sqrt{\alpha_c}}{\sigma}, \] to the difference \[ \int_m^\infty \phi(z)\Phi(a+bz)\,dz = \int_0^\infty \phi(z)\Phi(a+bz)\,dz - \int_0^m \phi(z)\Phi(a+bz)\,dz, \]
we have

\begin{align*}
    T_{21} &=
    \Phi \inparen{\frac{\sqrt{\alpha_c}}{\sigma}}\Phi \inparen{-\frac{1}{c\sigma}}
    -T \inparen{\frac{1}{c\sigma}, c\sqrt{\alpha_c}}
    -T \inparen{\frac{\sqrt{\alpha_c}}{\sigma}, \frac{1}{c\sqrt{\alpha_c}}} 
    + T \inparen{\frac{1}{c\sigma}, c}.
\end{align*}
Putting it together and simplifying gives
\begin{align*}
     r_\sigma(\theta^\star; A_2)
     &=
     \alpha_c \Bigg(
     \Phi \left(-\frac{\sqrt{\alpha_c}}{\sigma}\right)
     \Phi \left(-\frac{1}{c\sigma}\right) \\
     &\qquad\qquad
     + T \left(\frac{1}{c\sigma}, c\sqrt{\alpha_c}\right)
     + T \left(\frac{\sqrt{\alpha_c}}{\sigma},
        \frac{1}{c\sqrt{\alpha_c}}\right)
     - T \left(\frac{1}{c\sigma}, c\right)
     \Bigg).
\end{align*}
    \noindent
\textit{Risk on $A_3$:} $\hattheta_L = (0,1)$. We have 
\begin{align*}
    r_\sigma(\theta^\star; A_3)
    &=
    \int_{y_2=1}^\infty \int_{y_1=-\infty}^0 \frac{\phi\inparen{\frac{y_1}{\sigma}}}{\sigma} \frac{\phi\inparen{\frac{y_2}{\sigma}}}{\sigma} dy_1 dy_2 = \frac{1}{2} \Phi \inparen{\frac{-1}{\sigma}}.
\end{align*}

\noindent
\textit{Risk on $A_{12}$:} $\hattheta_L = (s_c(y), cs_c(y))$ where $s_c(y) = \frac{y_1+cy_2}{1+c^2}$. Directly computing integrals in this case is particularly tedious and so we opt for a statistical approach instead. Note that 
\begin{align*}
    \|\hattheta_L\|^2 = (1+c^2) (s_c(y))^2 = \frac{(y_1 + cy_2)^2}{1+c^2}.
\end{align*}
Define 
\begin{align*}
    X := \frac{Z_1 + cZ_2}{\sqrt{1+c^2}}, 
    \qquad 
    W := \frac{Z_2 - cZ_1}{\sqrt{1+c^2}}.
\end{align*}
Note that $\|\hattheta_L\|^2 = \sigma^2 X^2$. We now rewrite the event $A_{12}$ in terms of $X,W$. First,  
\begin{align*}
    y_2 \le cy_1 \iff Z_2 \le cZ_1 \iff W \le 0,
\end{align*}
and second
\begin{align*}
    0 \le s_c(y) \le \frac1c \iff 0 \le X \le \frac{\sqrt{\alpha_c}}{\sigma}.
\end{align*}
Hence, $A_{12} = \{ W \le 0\} \cap \{ 0 \le X \le  \frac{\sqrt{\alpha_c}}{\sigma} \}$. 

Since $Z=(Z_1,Z_2)\sim N(0,I_2)$ is rotationally invariant and $(X,W)$ is obtained from $Z$ by the above orthogonal change of coordinates, it follows that $(X,W)\sim N(0,I_2)$. 
In particular, $X,W\overset{\mathrm{i.i.d.}}{\sim}N(0,1)$.

Therefore, 
\begin{align*}
    r_\sigma(\theta^\star; A_{12})
    &= \E[\| \hattheta_L \|^2 \indicator \{A_{12}\}]= \sigma^2 \E \insquare{ X^2 \indicator \{ W \le 0\} \indicator \inbraces{ 0 \le X \le  \frac{\sqrt{\alpha_c}}{\sigma} } }\\
    &= \sigma^2 \E \insquare{ X^2 \indicator \inbraces{ 0 \le X \le  \frac{\sqrt{\alpha_c}}{\sigma} }} \P(W \le 0)
    = \frac{\sigma^2}{2} \int_0^{\frac{\sqrt{\alpha_c}}{\sigma}} x^2 \phi(x)dx \\
    &= \frac{\sigma^2}{2} \inparen{\Phi\inparen{\frac{\sqrt{\alpha_c}}{\sigma}}-\frac{1}{2} - \frac{\sqrt{\alpha_c}}{\sigma} \phi\inparen{\frac{\sqrt{\alpha_c}}{\sigma}}},
\end{align*}
where the third equality follows by the independence of $X,W$.

\noindent
\textit{Risk on $A_{13}$:} $\hattheta_L = (0,y_2)$. We have 
\begin{align*}
    r_\sigma(\theta^\star; A_{13})
    &=
    \int_{y_2=0}^1 \int_{y_1=-\infty}^0 y_2^2 \frac{\phi\inparen{\frac{y_1}{\sigma}}}{\sigma} \frac{\phi\inparen{\frac{y_2}{\sigma}}}{\sigma} dy_1 dy_2
    = \frac{\sigma^2}{2} \int_{0}^{\frac{1}{\sigma}} z^2 \phi(z)dz \\
    &= \frac{\sigma^2}{2} \inparen{ \Phi \inparen{\frac{1}{\sigma}} - \frac{1}{\sigma} \phi \inparen{\frac{1}{\sigma}}-\frac{1}{2}}.
\end{align*}

\noindent
\textit{Risk on $A_{23}$:} $\hattheta_L = (y_1,1)$. We have 
\begin{align*}
    r_\sigma(\theta^\star; A_{23})
    &=
    \int_{y_2=1}^\infty \int_{y_1=0}^{\frac{1}{c}} (1+y_1^2)\frac{\phi\inparen{\frac{y_1}{\sigma}}}{\sigma} \frac{\phi\inparen{\frac{y_2}{\sigma}}}{\sigma} dy_1 dy_2
    \\
    &= 
    \Phi \inparen{\frac{-1}{\sigma}} \int_0^{\frac{1}{c\sigma}} (1+\sigma^2 z^2) \phi(z) dz\\
    &= 
    \Phi \inparen{\frac{-1}{\sigma}} \insquare{
    (1+\sigma^2)\inparen{ \Phi\inparen{\frac{1}{c\sigma}} - \frac{1}{2}} - \frac{\sigma}{c} \phi\inparen{\frac{1}{c\sigma}}
    }.
\end{align*}

\end{proof}

\subsubsection{Risk analysis of \texorpdfstring{$\hattheta_S$}{theta-hat S}}

\begin{theorem}\label{thm:hatthetaSExactRisk}
    In the running example,
    \begin{align*}
        R_\sigma(\theta^\star; \Theta_S)
        &=
        \sigma^2 \inparen{ 
            \Phi \inparen{ \frac{\sqrt{\alpha_c}}{\sigma} }- \frac{1}{2} - \frac{\sqrt{\alpha_c}}{\sigma} \phi \inparen{ \frac{\sqrt{\alpha_c}}{\sigma}}
        }
        +
        \alpha_c \Phi \inparen{-\frac{\sqrt{\alpha_c}}{\sigma}},
    \end{align*}
    where $\alpha_c = 1+\frac{1}{c^2}$.
\end{theorem}
\begin{proof}
We prove a more general result by letting $\theta^\star$ be an arbitrary element of $\Theta_S$ and then specialize to the $\theta^\star = v_1$ case. Note that we can write $\Theta_S = \{ t v_2: t\in [0,1]\}$, and so $\theta^\star = t^\star v_2$ for some $t^\star \in [0,1]$. Since $\Theta_S$ is simply a (truncated) line segment, the orthogonal projection onto $\Theta_S$ takes a particularly simple form. Specifically, we have $\hattheta_S = \hatt v_2$, where 
\begin{align*}
    \hatt := \operatorname{clip} \inparen{\frac{\inp{Y, v_2}}{\|v_2\|^2},0,1},
    \qquad 
    \operatorname{clip}(x,0,1) := \min(1, \max(0,x)).
\end{align*}
Recall $\|v_2\|^2 = 1+\frac{1}{c^2}=\alpha_c$. Now, we have 
\begin{align*}
    Y = \theta^\star + \sigma Z = t^\star v_2 + \sigma Z
\end{align*}

which implies 
\begin{align*}
    \frac{\inp{Y,v_2}}{\alpha_c} 
    = t^\star + \sigma \frac{\inp{Z, v_2}}{\alpha_c}
    = t^\star + \frac{\sigma}{\sqrt{\alpha_c}}g,
\end{align*}
where we define the scalar random variable $g:= \frac{\inp{Z, v_2}}{ \sqrt{\alpha_c}}$. Since $\inp{Z, v_2} \sim N(0, \alpha_c)$, it follows that $g \sim N(0,1)$.
Using this, we define the un-clipped estimator
\begin{align*}
    \tildet := \frac{\inp{Y,v_2}}{\alpha_c} = t^\star + \frac{\sigma}{\sqrt{\alpha_c}} g,
\end{align*}
and so $\hatt = \operatorname{clip}(\tildet,0,1)$. The risk is then 
\begin{align*}
    R_\sigma(\theta^\star; \Theta_S) = \E \| \hatt v_2 - t^\star v_2\|^2 = \alpha_c \E (\hatt - t^\star)^2.
\end{align*} 
Note that the event $\{\tildet \in [0,1]\}$ is equivalent to the event $\{ g \in [-\frac{\sqrt{\alpha_c}t^\star}{\sigma}, \frac{\sqrt{\alpha_c}(1-t^\star)}{\sigma}]\}$. To ease notation, for the remainder of the proof write: $a:= -\frac{\sqrt{\alpha_c}t^\star}{\sigma}$ and $b:= \frac{\sqrt{\alpha_c}(1-t^\star)}{\sigma}$. Direct calculation then yields 
\begin{align*}
    \E (\hatt - t^\star)^2
    &=
    \E (\hatt - t^\star)^2\indicator \{ g \le a \}
    +
    \E (\hatt - t^\star)^2\indicator \{ g \in [a,b] \}
    +
    \E (\hatt - t^\star)^2\indicator \{ g \ge b \}\\
    &=
    (t^\star)^2 \P ( g \le a )
    +
    \E (\tildet - t^\star)^2\indicator \{ g \in [a,b] \}
    +
    (1 - t^\star)^2  \P ( g \ge b)\\
    &=
    (t^\star)^2 \Phi(a)
    +
    \frac{\sigma^2}{\alpha_c }\E[ g^2 \indicator \{g \in [a,b] \} ]
    +
    (1 - t^\star)^2  \Phi(-b)\\
    &=
    (t^\star)^2 \Phi(a)
    +
    \frac{\sigma^2}{\alpha_c } \inparen{\Phi(b) - \Phi(a) - b\phi(b) +a \phi(a)}
    +
    (1 - t^\star)^2  \Phi(-b),
\end{align*}
where we have used the fact that 
\begin{align*}
    \E[ g^2 \indicator \{g \in [a,b] \} ]
    = \int_a^b z^2 \phi(z)dz = \Phi(b) - \Phi(a) - b\phi(b) + a \phi(a).
\end{align*}
The result follows by taking $t^\star = 0$.
\end{proof}

\section{Ratios of coordinates of a uniform random vector}\label{sec:coordratios}
We collect two elementary facts on ratios of coordinates of a uniform random vector on the unit circle.
\begin{lemma}\label{lem:ratioUnifSphere}
    Let $U =(U_1,U_2) \sim \operatorname{Unif}(\mcS^{1})$ and define $R := \frac{U_2}{U_1}$. Since $\P(U_1=0)=0$, $R$ is well defined almost surely. Then $R \sim \mathrm{Cauchy}(0,1)$.
\end{lemma}
 \begin{proof}
    Let $Z = (Z_1,Z_2) \sim N(0,I_2)$. By rotational invariance, $U\equald \frac{Z}{\|Z\|}$. Therefore
    \[
        \frac{U_2}{U_1}
        \equald
        \frac{Z_2}{Z_1}.
    \]
    Since the ratio of two independent standard normal random variables has the standard Cauchy distribution, the result follows.
\end{proof}

\begin{lemma}\label{lem:ratioUnifSphereConditionalOnSignU1}
    Let $U =(U_1,U_2) \sim \operatorname{Unif}(\mcS^{1})$ and define
    $R := \frac{U_2}{U_1}$. Since $\P(U_1=0)=0$, $R$ is well defined
    almost surely. Then
    \[
        R \mid \{U_1\ge 0\} \sim \mathrm{Cauchy}(0,1),
    \]
    and 
    \[
        R \mid \{U_1\le 0\} \sim \mathrm{Cauchy}(0,1).
    \]
\end{lemma}

\begin{proof}
    We prove the first claim. Let $Z = (Z_1,Z_2) \sim N(0,I_2)$. By rotational invariance,
    $U\equald \frac{Z}{\|Z\|}$. Therefore $\frac{U_2}{U_1} \equald \frac{Z_2}{Z_1}$, and $\{U_1\ge0\} \equald \{Z_1\ge0\}$.
    Let
    \[
        X:=Z_1,
        \qquad
        R:=\frac{Z_2}{Z_1},
    \]
    and consider the change of variables $(x,r)\mapsto (z_1,z_2)=(x,xr)$. Equivalently, $Z_1=X$ and $Z_2=XR$. A direct calculation gives the joint density
    \[
        f_{X,R}(x,r)
        =
        \frac{|x|}{2\pi}
        \exp\left(-\frac{(1+r^2)x^2}{2}\right).
    \]
    Conditioning on $\{Z_1\ge0\}$ corresponds in $(x,r)$-coordinates to the region $\{x\ge0\}$. Thus
    \begin{align*}
        f_{R\mid Z_1\ge0}(r)
        &=
        \frac{1}{\P(Z_1\ge0)}
        \int_0^\infty f_{X,R}(x,r)\,dx \\
        &=
        2
        \int_0^\infty
        \frac{x}{2\pi}
        \exp\left(-\frac{(1+r^2)x^2}{2}\right)\,dx \\
        &=
        \frac{1}{\pi(1+r^2)}.
    \end{align*}
    Thus $R \mid \{U_1\ge0\}$ has the standard Cauchy distribution. The proof of the claim for $R \mid \{U_1\le0\}$ follows similarly, replacing $\{Z_1\ge0\}$ by $\{Z_1\le0\}$ and integrating over $(-\infty,0]$ instead of $[0,\infty)$.
\end{proof}

\section{Vanishing-noise behavior of the LSE}
\label{sec:vanishingnoise}
In this section we prove Proposition~\ref{prop:smallSigmaGeneral}. The argument uses differentiability properties of Euclidean projections onto closed convex sets. In the boundary case, the rescaled projection converges to the projection of the noise onto the tangent cone. In the interior case, the constraint is inactive with overwhelmingly high probability, which yields the sharper remainder estimate.

\begin{proof}[Proof of Proposition~\ref{prop:smallSigmaGeneral}]
For $\sigma>0$, define the scaled mapping
\[
G_\sigma(u):=\frac{\Pi_\Theta(\theta^\star+\sigma u)-\theta^\star}{\sigma},
\qquad u\in\R^d,
\]
so that $(\hattheta_\sigma-\theta^\star)/\sigma=G_\sigma(Z)$. Since $\Theta$ is closed and convex, the Euclidean projection $\Pi_\Theta$ is $1$-Lipschitz, and since $\Pi_\Theta(\theta^\star)=\theta^\star$, for all $\sigma>0$ and $u \in \R^d$

\begin{align}
\|G_\sigma(u)\|
=
\frac{\|\Pi_\Theta(\theta^\star+\sigma u)-\Pi_\Theta(\theta^\star)\|}{\sigma}
\le \|u\|
\tag{$\dagger$}\label{eq:dominate-v2}
\end{align}
Also, for any closed convex cone $C$, $\|\Pi_C(u)\|\le \|u\|$.

\noindent\textit{Boundary case.}
Assume $\theta^\star\in\partial\Theta$. Since $\Theta$ is compact, convex, and has nonempty interior, the projection $\Pi_\Theta$ is one-sided directionally differentiable at every boundary point, see \cite{zarantonello1971projections} and \cite{shapiro2016differentiability}.
More precisely, for each $u\in\R^d$,
\[
\lim_{t\to0^+}\frac{\Pi_\Theta(\theta^\star+t u)-\theta^\star}{t}
=
\Pi_{S_\Theta(\theta^\star)}(u),
\]
where $S_\Theta(\theta^\star)$ is the support cone of $\Theta$ at $\theta^\star$. 

By Definition~2.3 of \cite{zarantonello1971projections}, the
support cone $S_\Theta(\theta^\star)$ is the smallest closed convex cone
containing $\Theta-\theta^\star$. Since $\Theta$ is convex and $\theta^\star\in\Theta$, the set
$\bigcup_{t\ge0}t(\Theta-\theta^\star)$ is a convex cone containing
$\Theta-\theta^\star$. Therefore, its closure is precisely the smallest
closed convex cone containing $\Theta-\theta^\star$. By our definition of $T_\Theta(\theta^\star)$, it follows that $S_\Theta(\theta^\star)=T_\Theta(\theta^\star)$.
Taking $t=\sigma$ shows that $G_\sigma(u)\to\Pi_{T_\Theta(\theta^\star)}(u)$
for each fixed $u$. Evaluating at $Z$, we obtain almost sure convergence
$G_\sigma(Z)\to\Pi_{T_\Theta(\theta^\star)}(Z)$. Moreover, by \eqref{eq:dominate-v2},
\[
\|G_\sigma(Z)-\Pi_{T_\Theta(\theta^\star)}(Z)\|^2
\le
2\|G_\sigma(Z)\|^2+2\|\Pi_{T_\Theta(\theta^\star)}(Z)\|^2
\le
4\|Z\|^2,
\]
and since $\E\|Z\|^2<\infty$, dominated convergence yields
\[
\E\|G_\sigma(Z)-\Pi_{T_\Theta(\theta^\star)}(Z)\|^2\to0,
\]
which proves \eqref{eq:scaled-boundary-L2-v2}.
Then \eqref{eq:risk-boundary-v2}~follows from
\[
\E\|\hattheta_\sigma-\theta^\star\|^2
=\sigma^2\E\|G_\sigma(Z)\|^2
=\sigma^2(\E\|\Pi_{T_\Theta(\theta^\star)}(Z)\|^2+o(1)).
\]

\noindent\textit{Interior case.}
Assume $\theta^\star\in\operatorname{int}(\Theta)$ and let $ r:=\inf_{\theta\in\partial\Theta}\|\theta^\star-\theta\|>0$. Then $B(\theta^\star,r) \subseteq \Theta$.
On the event $\{\|Z\|\le r/\sigma\}$, $\theta^\star+\sigma Z\in\Theta$ and hence $G_\sigma(Z)=Z$.
Therefore $G_\sigma(Z)\to Z$ almost surely. To obtain $L^2$ convergence, note that by \eqref{eq:dominate-v2},
\[
\E\|G_\sigma(Z)-Z\|^2
=
\E\big[\|G_\sigma(Z)-Z\|^2\indicator_{\{\|Z\|>r/\sigma\}}\big]
\le
4\E\big[\|Z\|^2\indicator_{\{\|Z\|>r/\sigma\}}\big].
\]

Since $r/\sigma\to\infty$ as $\sigma\to0^+$, we have
$\indicator_{\{\|Z\|>r/\sigma\}}\to0$ almost surely. Moreover, $\|Z\|^2$ is
integrable, so  $\E\big[\|Z\|^2\indicator_{\{\|Z\|>r/\sigma\}}\big]\to0$ by dominated convergence. This proves \eqref{eq:scaled-interior-L2-v2}.

By \eqref{eq:dominate-v2}, we have $\|\hattheta_\sigma-\theta^\star\|
    =
    \sigma\|G_\sigma(Z)\|
    \le
    \sigma\|Z\|$. Hence, almost surely, 
\[
    0
    \le
    \sigma^2\|Z\|^2-\|\hattheta_\sigma-\theta^\star\|^2.
\]
Moreover, on the event $\{\|Z\|\le r/\sigma\}$, we have
$\theta^\star+\sigma Z\in\Theta$, and therefore
\[
    \hattheta_\sigma
    =
    \Pi_\Theta(\theta^\star+\sigma Z)
    =
    \theta^\star+\sigma Z.
\]
Thus, on the same event, $\sigma^2\|Z\|^2-\|\hattheta_\sigma-\theta^\star\|^2 = 0$, and consequently,
\begin{align*}
0
\le
\sigma^2 d-\E\|\hattheta_\sigma-\theta^\star\|^2 
=
\E\left[
    \sigma^2\|Z\|^2-\|\hattheta_\sigma-\theta^\star\|^2
\right] 
\le
\sigma^2
\E\big[\|Z\|^2\indicator_{\{\|Z\|>r/\sigma\}}\big].
\end{align*}

Thus
\[
0\le \sigma^2 d - \E\|\hattheta_\sigma-\theta^\star\|^2
\le
\sigma^2\E\big[\|Z\|^2\indicator_{\{\|Z\|>r/\sigma\}}\big].
\]

It remains to show that, for every $m\ge 1$,
\[
    \sigma^2
    \E\bigl[
        \|Z\|^2
        \indicator_{\{\|Z\|>r/\sigma\}}
    \bigr]
    =
    o(\sigma^m).
\]
Since $\|Z\|^2\sim\chi_d^2$, it has finite moments of all orders.
Fix $m\ge 1$, and choose an integer $k\ge 0$ such that
$2+2k>m$. On the event $\{\|Z\|^2>r^2/\sigma^2\}$, we have
\[
    1
    \le
    \left(\frac{\sigma^2\|Z\|^2}{r^2}\right)^k .
\]
Therefore
\[
    \|Z\|^2
    \indicator_{\{\|Z\|^2>r^2/\sigma^2\}}
    \le
    \|Z\|^2
    \left(\frac{\sigma^2\|Z\|^2}{r^2}\right)^k
    =
    \frac{\sigma^{2k}}{r^{2k}}\|Z\|^{2(k+1)}.
\]
Multiplying by $\sigma^2$ and taking expectations yields
\[
    \sigma^2
    \E\bigl[
        \|Z\|^2
        \indicator_{\{\|Z\|^2>r^2/\sigma^2\}}
    \bigr]
    \le
    \frac{\E[\|Z\|^{2(k+1)}]}{r^{2k}}
    \sigma^{2+2k}.
\]
Since $\E[\|Z\|^{2(k+1)}]<\infty$, and since $2+2k>m$, the right-hand
side is $o(\sigma^m)$. Hence
\[
    \sigma^2
    \E\bigl[
        \|Z\|^2
        \indicator_{\{\|Z\|>r/\sigma\}}
    \bigr]
    =
    o(\sigma^m).
\]
Since $m\ge1$ was arbitrary, the claim follows. Combining this estimate with the preceding display gives, for every $m\ge1$,
\[
    \sigma^2 d-\E\|\hattheta_\sigma-\theta^\star\|^2
    =
    o(\sigma^m).
\]
Equivalently,
\[
    \E\|\hattheta_\sigma-\theta^\star\|^2
    =
    \sigma^2 d+o(\sigma^m)
    \qquad
    \text{for every } m\ge1,
\]
which proves \eqref{eq:risk-interior-superpoly-v2}.
\end{proof}

\section{Quantitative bounds for polytopes}
\label{sec:quantitative-bounds-polytopes}
In Theorem~\ref{thm:large-sigma-projection}, we established that a positive
limiting gap
\[
    \Delta_\infty
    :=
    R_\infty(\theta^\star;\Theta_S)
    -
    R_\infty(\theta^\star;\Theta_L)
    >
    0
\]
implies risk reversal for all sufficiently large $\sigma$. In this section,
we make this implication quantitative for the special case of compact convex
polytopes. The key step is an effective bound on the convergence
$R_\sigma\to R_\infty$. For polytopes, the projection in the finite-noise
setting can differ from its diverging-noise limiting projection only in directions
where a vertex outside the exposed face is nearly tied for the maximum in that
direction. Controlling the spherical measure of this set of directions gives an
$O(\sigma^{-1})$ bound for $R_\infty-R_\sigma$, uniformly over
$\theta^\star$. This yields an explicit geometry-dependent sufficient threshold for finite-noise risk reversal.

\begin{proposition}\label{prop:quant-diverging-noise-polytope}
Fix $d\ge2$, and for $m \ge 2$ let $v_1,\ldots,v_m\in\R^d$ be the distinct vertices of a compact, convex polytope $\Theta=\operatorname{conv}(v_1,\ldots,v_m)\subseteq\R^d$. Set
\[
    D_{\Theta}:=\operatorname{diam}(\Theta),
    \qquad
    \delta_\Theta:=\min_{i\neq j}\|v_i-v_j\|.
\]
Then there exists a constant $C_d<\infty$, depending only on $d$, such that, uniformly over $\theta^\star\in\Theta$,
\[
0
\le
R_\infty(\theta^\star;\Theta)-R_\sigma(\theta^\star;\Theta)
\le
\frac{
    C_d \binom{m}{2}D_\Theta^4
}{
    \delta_\Theta\sigma
},
\qquad \sigma>0.
\]
\end{proposition}

\begin{proof}
Fix $\theta^\star\in\Theta$. Write $Z=\|Z\|U$, where $U\sim\operatorname{Unif}(\mcS^{d-1})$ and $U$ is independent of $\|Z\|$. For $u\in\mcS^{d-1}$, set 
\[
F_\Theta(u) := \operatorname*{arg\,max}_{\theta\in\Theta} \langle u,\theta\rangle,
\qquad
Q(u):=\Pi_{F_\Theta(u)}(\theta^\star).
\] 
Then Theorem~\ref{thm:large-sigma-projection} gives $R_\infty(\theta^\star;\Theta)
=
\E\|Q(U)-\theta^\star\|^2$.
On the other hand, writing $P_\tau(u):=\Pi_\Theta(\theta^\star+\tau u)$ for a given radius $\tau>0$, the finite-noise risk can be written as 
\[
R_\sigma(\theta^\star;\Theta) 
= \E \bigl\| P_{\sigma\|Z\|}(U)-\theta^\star \bigr\|^2. 
\]

We now prove two deterministic bounds relating the finite-radius projection $P_\tau(u)$ and the limiting exposed-face projection $Q(u)$ for any $u \in \mcS^{d-1}$. 
First, note that $P_\tau(u)$ maximizes
\[
    \theta\mapsto
    \langle u,\theta\rangle
    -
    \frac{1}{2\tau}\|\theta-\theta^\star\|^2
\]
over $\Theta$. Comparing this objective at $P_\tau(u)$ and at
$Q(u)\in F_\Theta(u)$, we obtain
\[
    h_\Theta(u)-\langle u,P_\tau(u)\rangle
    \le
    \frac{
    \|Q(u)-\theta^\star\|^2
    -
    \|P_\tau(u)-\theta^\star\|^2
    }{2\tau}.
\]
Since the left-hand side is nonnegative, we have
\[
    0
    \le
    \|Q(u)-\theta^\star\|^2
    -
    \|P_\tau(u)-\theta^\star\|^2.
\]
Second, we argue that the two projections coincide unless $u$ lies close to a
boundary where the exposed face $F_\Theta(u)$ changes.

To make this precise, let $V(\Theta)$ denote the vertex set of $\Theta$, and define 
\[
\gamma_\Theta(u)
:=
h_\Theta(u)
-
\max\{ \langle u,v\rangle: v\in V(\Theta),\ v\notin F_\Theta(u) \},
\qquad
h_\Theta(u):=\max_{\theta\in\Theta}\langle u,\theta\rangle,
\] 
with the convention that $\gamma_\Theta(u)=+\infty$ if all vertices lie in $F_\Theta(u)$. We call $\gamma_\Theta(u)$ the directional margin. It is the gap between the support value $h_\Theta(u)$ and the largest value of
$\langle u,v\rangle$ among vertices not belonging to the exposed face
$F_\Theta(u)$. We claim that $P_\tau(u)$ and $Q(u)$ differ only on the small-margin set
\[
\left\{ u\in\mcS^{d-1}: \gamma_\Theta(u)\le \frac{D_{\Theta}^2}{\tau} \right\}. 
\]

Let $y=\theta^\star+\tau u$. Since $P_\tau(u)=\Pi_\Theta(y)$, to prove $P_\tau(u)=Q(u)$ it is enough
to show that $Q(u)$ satisfies the variational characterization of the
projection of $y$ onto $\Theta$. Since $Q(u)\in F_\Theta(u)\subseteq\Theta$,
this amounts to verifying that
\[
    \langle y-Q(u),\theta-Q(u)\rangle\le0
    \qquad
    \text{for all }\theta\in\Theta.
\]
Write $\theta=\sum_{i=1}^m\lambda_i v_i$, where
$\lambda_i\ge0$ and $\sum_i\lambda_i=1$. Let $I:=\{i:v_i\in F_\Theta(u)\}$. For $i\in I$, since $v_i\in F_\Theta(u)$ and
$Q(u)=\Pi_{F_\Theta(u)}(\theta^\star)$,
\[
    \langle \theta^\star-Q(u),v_i-Q(u)\rangle\le0,
    \qquad
    \langle u,v_i-Q(u)\rangle=0.
\]
Here the first inequality is the projection optimality condition, while the
second identity follows because both $v_i$ and $Q(u)$ lie in the exposed
face $F_\Theta(u)$. 
For $i\notin I$, by Cauchy-Schwarz,
\[
    \langle \theta^\star-Q(u),v_i-Q(u)\rangle
    \le
    \|\theta^\star-Q(u)\|\,\|v_i-Q(u)\|
    \le
    D_{\Theta}^2,
\]
while, by the definition of $\gamma_\Theta(u)$,
\[
    \langle u,v_i-Q(u)\rangle
    =
    \langle u,v_i\rangle-h_\Theta(u)
    \le
    -\gamma_\Theta(u).
\]
Therefore
\[
    \langle y-Q(u),\theta-Q(u)\rangle
    =
    \sum_{i=1}^m
    \lambda_i
    \langle \theta^\star-Q(u)+\tau u,v_i-Q(u)\rangle
    \le
    \sum_{i\notin I}
    \lambda_i
    \bigl(D_{\Theta}^2-\tau\gamma_\Theta(u)\bigr)
    \le0
\]
whenever $\tau\gamma_\Theta(u)\ge D_{\Theta}^2$. This proves the claim.

Combining the preceding observations gives, for every $u\in\mcS^{d-1}$
and $\tau>0$,
\[
    0
    \le
    \|Q(u)-\theta^\star\|^2
    -
    \|P_\tau(u)-\theta^\star\|^2
    \le
    D_{\Theta}^2
    \mathbf 1
    \left\{
        \gamma_\Theta(u)
        \le
        \frac{D_{\Theta}^2}{\tau}
    \right\}.
\]

It remains to control the probability of small margins. If
$\gamma_\Theta(u)\le s$, then there exist distinct vertices $v_i,v_j$ such
that $|\langle u,v_i-v_j\rangle|\le s$. Indeed, one may take $v_i\in F_\Theta(u)$ and a vertex $v_j\notin F_\Theta(u)$
whose value of $\langle u,v_j\rangle$ is within $s$ of the maximum. Thus
\[
    \{u:\gamma_\Theta(u)\le s\}
    \subseteq
    \bigcup_{i<j}
    \{
        u\in\mcS^{d-1}:
        |\langle u,v_i-v_j\rangle|\le s
    \}.
\]
Using rotational invariance, for every nonzero $a\in\R^d$, $\langle U,a\rangle
\stackrel{d}{=}\|a\|U_1$,
where $U_1$ is the first coordinate of $U$. Since $U_1$ has bounded density in a neighborhood of zero, there exists a constant $C_d<\infty$, depending only on $d$, such that for every $t>0$,
\[
    \P(|U_1|\le t)\le C_d\min\{1,t\}.
\]
Therefore, for every $s>0$,
\[
    \P(|\langle U,a\rangle|\le s)
    =
    \P\left(|U_1|\le \frac{s}{\|a\|}\right)
    \le
    C_d\min\left\{1,\frac{s}{\|a\|}\right\}.
\]
Using $\|v_i-v_j\| \ge \delta_{\Theta}$, a union bound then gives
\[
    \P\bigl(\gamma_\Theta(U)\le s\bigr)
    \le
    C_d \binom{m}{2}\frac{s}{\delta_\Theta},
    \qquad s>0.
\]

We now apply the deterministic estimate with
$\tau=\sigma\|Z\|$. Conditioning on $\|Z\|$, we get
\[
\begin{aligned}
0
&\le
R_\infty(\theta^\star;\Theta)
-
R_\sigma(\theta^\star;\Theta)  \\
&\le
D_{\Theta}^2
\E\left[
\P\left(
    \gamma_\Theta(U)
    \le
    \frac{D_{\Theta}^2}{\sigma\|Z\|}
    \,\middle|\, \|Z\|
\right)
\right] \\
&\le
D_{\Theta}^2
\E\left[
    C_d \binom{m}{2}
    \frac{D_{\Theta}^2}{\delta_\Theta\sigma\|Z\|}
\right] \\
&=
C_d \binom{m}{2}
\frac{D_{\Theta}^4}{\delta_\Theta\sigma}
\E\|Z\|^{-1}.
\end{aligned}
\]
The bound is uniform in $\theta^\star$. Since $d\ge2$,
\[
    \E\|Z\|^{-1}
    =
    2^{-1/2}
    \frac{\Gamma((d-1)/2)}{\Gamma(d/2)}
    <
    \infty.
\]
Absorbing the finite factor $\E\|Z\|^{-1}$, which depends only on $d$, into $C_d$
gives the claimed bound.
\end{proof}

\begin{remark}[Sharpness of the $\sigma^{-1}$ dependence]
The order $\sigma^{-1}$ in Proposition~\ref{prop:quant-diverging-noise-polytope}
is sharp in general. To see this, take $\Theta=[-1,1]^d$ and $\theta^\star=0$. Then
\[
    \Pi_\Theta(\sigma Z)
    =
    \bigl(\operatorname{clip}(\sigma Z_i,-1,1)\bigr)_{i=1}^d,
\]
and therefore, if $g\sim N(0,1)$,
\[
    R_\sigma(\theta^\star;\Theta)
    =
    d\,\E\bigl[\min(\sigma^2g^2,1)\bigr].
\]
On the other hand, in the diverging-noise limit the projection selects a vertex
of the cube almost surely, so $R_\infty(\theta^\star;\Theta)=d$. Hence
\[
    R_\infty(\theta^\star;\Theta)-R_\sigma(\theta^\star;\Theta)
    =
    d\,\E\left[
        (1-\sigma^2g^2)\mathbf 1_{\{|g|\le 1/\sigma\}}
    \right].
\]
Writing $\phi$ for the standard normal density and changing variables
$x=\sigma g$, we obtain
\[
    R_\infty(\theta^\star;\Theta)-R_\sigma(\theta^\star;\Theta)
    =
    \frac{d}{\sigma}
    \int_{-1}^{1}(1-x^2)\phi(x/\sigma)\,dx.
\]
Consequently,
\[
    \sigma\bigl(R_\infty(0;\Theta)-R_\sigma(0;\Theta)\bigr)
    \to
    d\phi(0)\int_{-1}^{1}(1-x^2)\,dx
    =
    \frac{4d}{3\sqrt{2\pi}}.
\]
Thus
\[
    R_\infty(0;\Theta)-R_\sigma(0;\Theta)
    \sim
    \frac{4d}{3\sqrt{2\pi}}\frac1\sigma,
    \qquad \sigma\to\infty.
\]
Therefore the $\sigma^{-1}$ rate cannot be improved in general for compact
convex polytopes.
\end{remark}

\begin{corollary}
\label{cor:quant-threshold-polytope}
Let $d\ge2$, and define
\[
    \Theta_S=\operatorname{conv}(v_1,\ldots,v_{m_S}),
    \qquad
    \Theta_L=\operatorname{conv}(w_1,\ldots,w_{m_L})
\]
to be compact convex polytopes in $\R^d$, where
$v_1,\ldots,v_{m_S}$ are pairwise distinct and
$w_1,\ldots,w_{m_L}$ are pairwise distinct. Suppose further that $\Theta_S\subsetneq\Theta_L$.

Fix
$\theta^\star\in\Theta_S$, and suppose
\[
    \Delta_\infty
    :=
    R_\infty(\theta^\star;\Theta_S)
    -
    R_\infty(\theta^\star;\Theta_L)
    >
    0.
\]
Set
\[
    D_{\Theta_S}:=\operatorname{diam}(\Theta_S),
    \qquad
    \delta_{\Theta_S}:=\min_{i\neq j}\|v_i-v_j\|.
\]
Then risk reversal holds whenever
\[
    \sigma
    >
    \frac{
        C_d \binom{m_S}{2} D_{\Theta_S}^4
    }{
        \delta_{\Theta_S}\Delta_\infty
    }.
\]
That is, for every such $\sigma$, $R_\sigma(\theta^\star;\Theta_S) >R_\sigma(\theta^\star;\Theta_L)$.
\end{corollary}
\begin{proof}
For a compact polytope $\Theta$, define
\[
    \varepsilon_\sigma(\Theta)
    :=
    R_\infty(\theta^\star;\Theta)
    -
    R_\sigma(\theta^\star;\Theta).
\]
By Proposition~\ref{prop:quant-diverging-noise-polytope}, applied to
$\Theta_S$, we have
\[
    0
    \le
    \varepsilon_\sigma(\Theta_S)
    \le
    \frac{
        C_d \binom{m_S}{2}D_{\Theta_S}^4
    }{
        \delta_{\Theta_S}\sigma
    }.
\]
Similarly, applying Proposition~\ref{prop:quant-diverging-noise-polytope} to
$\Theta_L$ gives $\varepsilon_\sigma(\Theta_L)\ge 0$. Therefore,
\begin{align*}
    R_\sigma(\theta^\star;\Theta_S)
    -
    R_\sigma(\theta^\star;\Theta_L)
    &=
    \Bigl(R_\infty(\theta^\star;\Theta_S)
    -
    \varepsilon_\sigma(\Theta_S)\Bigr)
    -
    \Bigl(R_\infty(\theta^\star;\Theta_L)
    -
    \varepsilon_\sigma(\Theta_L)\Bigr) \\
    &=
    \Delta_\infty
    -
    \varepsilon_\sigma(\Theta_S)
    +
    \varepsilon_\sigma(\Theta_L) \\
    &\ge
    \Delta_\infty
    -
    \frac{
        C_d \binom{m_S}{2}D_{\Theta_S}^4
    }{
        \delta_{\Theta_S}\sigma
    }.
\end{align*}
The right-hand side is strictly positive whenever
\[
    \sigma
    >
    \frac{
        C_d \binom{m_S}{2}D_{\Theta_S}^4
    }{
        \delta_{\Theta_S}\Delta_\infty
    }.
\]
Hence, for every such $\sigma$, $R_\sigma(\theta^\star;\Theta_S)
    >
    R_\sigma(\theta^\star;\Theta_L)$, as claimed.
\end{proof}

\bibliographystyle{plainnat}
\bibliography{references}

\begin{thebibliography}{25}
\providecommand{\natexlab}[1]{#1}
\providecommand{\url}[1]{\texttt{#1}}
\expandafter\ifx\csname urlstyle\endcsname\relax
  \providecommand{\doi}[1]{doi: #1}\else
  \providecommand{\doi}{doi: \begingroup \urlstyle{rm}\Url}\fi

\bibitem[Amelunxen et~al.(2014)Amelunxen, Lotz, McCoy, and
  Tropp]{amelunxen2014living}
Dennis Amelunxen, Martin Lotz, Michael~B McCoy, and Joel~A Tropp.
\newblock {Living on the edge: Phase transitions in convex programs with random
  data}.
\newblock \emph{Information and Inference: A Journal of the IMA}, 3\penalty0
  (3):\penalty0 224--294, 2014.

\bibitem[Aolaritei et~al.(2025)Aolaritei, Jordan, Pathak, and
  Ulichney]{aolaritei2025revisiting}
Liviu Aolaritei, Michael~I Jordan, Reese Pathak, and Annie Ulichney.
\newblock {Revisiting mean estimation over $\ell_p$ balls: Is the MLE optimal?}
\newblock \emph{arXiv preprint arXiv:2506.10354}, 2025.

\bibitem[Birg{\'e} and Massart(1993)]{birge1993rates}
Lucien Birg{\'e} and Pascal Massart.
\newblock Rates of convergence for minimum contrast estimators.
\newblock \emph{Probability Theory and Related Fields}, 97\penalty0
  (1):\penalty0 113--150, 1993.

\bibitem[Bousquet et~al.(2022)Bousquet, Daniely, Kaplan, Mansour, Moran, and
  Stemmer]{bousquet2022monotone}
Olivier~J Bousquet, Amit Daniely, Haim Kaplan, Yishay Mansour, Shay Moran, and
  Uri Stemmer.
\newblock {Monotone learning}.
\newblock In \emph{Conference on Learning Theory}, pages 842--866. PMLR, 2022.

\bibitem[Chatterjee(2014)]{Chatterjee14}
Sourav Chatterjee.
\newblock {A new perspective on least squares under convex constraint}.
\newblock \emph{The Annals of Statistics}, 42\penalty0 (6):\penalty0 2340 --
  2381, 2014.
\newblock \doi{10.1214/14-AOS1254}.

\bibitem[Dudley(2002)]{Dudley_2002}
R.~M. Dudley.
\newblock \emph{Real Analysis and Probability}.
\newblock Cambridge Studies in Advanced Mathematics. Cambridge University
  Press, 2nd edition, 2002.

\bibitem[Fang and Guntuboyina(2019)]{fang2019risk}
Billy Fang and Adityanand Guntuboyina.
\newblock On the risk of convex-constrained least squares estimators under
  misspecification.
\newblock \emph{Bernoulli}, 25\penalty0 (3):\penalty0 2206--2244, 2019.
\newblock \doi{10.3150/18-BEJ1051}.

\bibitem[Han(2023)]{han2023noisy}
Qiyang Han.
\newblock {Noisy linear inverse problems under convex constraints: Exact risk
  asymptotics in high dimensions}.
\newblock \emph{The Annals of Statistics}, 51\penalty0 (4):\penalty0
  1611--1638, 2023.

\bibitem[Johnstone(2019)]{johnstone2019gaussian}
Iain~M. Johnstone.
\newblock Gaussian estimation: Sequence and wavelet models.
\newblock Draft, 2019.

\bibitem[Kur et~al.(2020)Kur, Rakhlin, and Guntuboyina]{kur2020suboptimality}
Gil Kur, Alexander Rakhlin, and Adityanand Guntuboyina.
\newblock On suboptimality of least squares with application to estimation of
  convex bodies.
\newblock In \emph{Conference on Learning Theory}, pages 2406--2424. PMLR,
  2020.

\bibitem[Kur et~al.(2023)Kur, Putterman, and Rakhlin]{kur2023variance}
Gil Kur, Eli Putterman, and Alexander Rakhlin.
\newblock On the variance, admissibility, and stability of empirical risk
  minimization.
\newblock \emph{Advances in Neural Information Processing Systems},
  36:\penalty0 37527--37539, 2023.

\bibitem[Kur et~al.(2024)Kur, Gao, Guntuboyina, and Sen]{kur2024convex}
Gil Kur, Fuchang Gao, Adityanand Guntuboyina, and Bodhisattva Sen.
\newblock Convex regression in multidimensions: Suboptimality of least squares
  estimators.
\newblock \emph{The Annals of Statistics}, 52\penalty0 (6):\penalty0
  2791--2815, 2024.

\bibitem[Mhammedi(2021)]{mhammedi2021risk}
Zakaria Mhammedi.
\newblock Risk monotonicity in statistical learning.
\newblock \emph{Advances in Neural Information Processing Systems},
  34:\penalty0 10732--10744, 2021.

\bibitem[Neykov(2023)]{Matey23}
Matey Neykov.
\newblock {On the Minimax Rate of the Gaussian Sequence Model Under Bounded
  Convex Constraints}.
\newblock \emph{IEEE Transactions on Information Theory}, 69\penalty0
  (2):\penalty0 1244--1260, 2023.
\newblock \doi{10.1109/TIT.2022.3213141}.

\bibitem[Noll(1995)]{noll1995directional}
Dominikus Noll.
\newblock Directional differentiability of the metric projection in {Hilbert}
  space.
\newblock \emph{Pacific Journal of Mathematics}, 170\penalty0 (2):\penalty0
  567--592, 1995.

\bibitem[Owen(1980)]{owen1980table}
Donald~Bruce Owen.
\newblock A table of normal integrals.
\newblock \emph{Communications in Statistics-Simulation and Computation},
  9\penalty0 (4):\penalty0 389--419, 1980.

\bibitem[Oymak and Hassibi(2016)]{oymak2016sharp}
Samet Oymak and Babak Hassibi.
\newblock {Sharp MSE bounds for proximal denoising}.
\newblock \emph{Foundations of Computational Mathematics}, 16\penalty0
  (4):\penalty0 965--1029, 2016.

\bibitem[Oymak et~al.(2013)Oymak, Thrampoulidis, and Hassibi]{oymak2013squared}
Samet Oymak, Christos Thrampoulidis, and Babak Hassibi.
\newblock The squared-error of generalized {LASSO}: A precise analysis.
\newblock In \emph{2013 51st Annual Allerton Conference on Communication,
  Control, and Computing (Allerton)}, pages 1002--1009. IEEE, 2013.

\bibitem[Prasadan and Neykov(2025)]{prasadan2025some}
Akshay Prasadan and Matey Neykov.
\newblock {Some facts about the optimality of the LSE in the Gaussian sequence
  model with convex constraint}.
\newblock \emph{IEEE Transactions on Information Theory}, 2025.

\bibitem[Shapiro(2016)]{shapiro2016differentiability}
Alexander Shapiro.
\newblock {Differentiability properties of metric projections onto convex
  sets}.
\newblock \emph{Journal of Optimization Theory and Applications}, 169\penalty0
  (3):\penalty0 953--964, 2016.

\bibitem[Tsybakov(2009)]{tsybakov2009nonparametric}
Alexandre~B. Tsybakov.
\newblock \emph{{Introduction to Nonparametric Estimation}}.
\newblock Springer Series in Statistics. Springer, New York, 2009.

\bibitem[van~der Vaart(2000)]{van2000asymptotic}
Aad~W. van~der Vaart.
\newblock \emph{Asymptotic statistics}, volume~3.
\newblock {Cambridge University Press}, 2000.

\bibitem[Viering et~al.(2020)Viering, Mey, and Loog]{viering2020making}
Tom~Julian Viering, Alexander Mey, and Marco Loog.
\newblock {Making learners (more) monotone}.
\newblock In \emph{International Symposium on Intelligent Data Analysis}, pages
  535--547. Springer, 2020.

\bibitem[Zarantonello(1971)]{zarantonello1971projections}
Eduardo~H Zarantonello.
\newblock Projections on convex sets in {Hilbert} space and spectral theory:
  Part {I}. projections on convex sets: Part {II}. spectral theory.
\newblock In \emph{Contributions to nonlinear functional analysis}, pages
  237--424. Elsevier, 1971.

\bibitem[Zhang(2013)]{zhang2013nearly}
Li~Zhang.
\newblock {Nearly Optimal Minimax Estimator for High-Dimensional Sparse Linear
  Regression}.
\newblock \emph{The Annals of Statistics}, 41\penalty0 (4):\penalty0
  2506--2537, 2013.
\newblock \doi{10.1214/13-AOS1141}.

\end{thebibliography}

\end{document}